\documentclass[a4paper, 12pt, reqno,final]{amsart}
\usepackage[foot]{amsaddr} %Adresse auf erste Seite
\usepackage[utf8x]{inputenc}
\usepackage[T1]{fontenc}
\usepackage[english]{babel}
\usepackage{amsmath, amsfonts, amsthm, amssymb, amscd}
\usepackage[final]{graphicx}
\usepackage[below]{placeins}
\usepackage{subfig}
\usepackage{multirow}
\usepackage{mathtools}

\graphicspath{{}}

\usepackage[hidelinks]{hyperref}
\hypersetup{
  pdfauthor = {Stefan Metzger},
  pdftitle = {doubleObstacle}
}

\usepackage{showkeys} 
\usepackage{esint}
\usepackage{caption}
\usepackage{dsfont}

\usepackage{tikz}
\usetikzlibrary{patterns}
\usetikzlibrary{shapes.geometric}

\usepackage[]{todonotes}
\usepackage{diagbox}

\newtheorem{theorem}{Theorem}[section]
\newtheorem{lemma}[theorem]{Lemma}
\newtheorem*{lemma*}{Lemma}
\newtheorem{corollary}[theorem]{Corollary}

\numberwithin{equation}{section}

\newcommand{\comment}[1]{}


\newcounter{IMASTYLE}
\makeatletter
\newcommand{\labitem}[2]{%
\def\@itemlabel{\textbf{#1}}
\item
\def\@currentlabel{#1}\label{#2}}
\makeatother
\newcommand{\norm}[1]{\left\|{#1}\right\|}
\newcommand{\abs}[1]{\left|{#1}\right|}

\newcommand{\rkla}[1]{{\left(#1\right)}}
\newcommand{\trkla}[1]{{(#1)}}
\newcommand{\gkla}[1]{{\left\{#1\right\}}}
\newcommand{\tgkla}[1]{{\{#1\}}}
\newcommand{\skla}[1]{{\left\langle#1\right\rangle}}
\newcommand{\ekla}[1]{{\left[#1\right]}}
\newcommand{\tekla}[1]{{[#1]}}
\newcommand{\tabs}[1]{|{#1}|}

\newcommand{\bs}[1]{\boldsymbol{#1}}

\newcommand{\D}{\mathcal{D}}
\newcommand{\iD}{\int_{\D}}

\newcommand{\jnorm}[1]{\left|\left|\left| #1 \right|\right|\right|}

\newcommand{\dx}{\, \mathrm{d}{x}}
\newcommand{\dy}{\, \mathrm{d}{y}}
\newcommand{\dt}{\, \mathrm{d}t}
\newcommand{\ds}{\, \mathrm{d}s}

\renewcommand{\div}{\operatorname{div}}

\newcommand{\ids}{I}

\newcommand{\nn}{^{n}}
\newcommand{\no}{^{n-1}}

\newcommand{\tl}{^{\tau}}
\newcommand{\tp}{^{\tau,+}}
\newcommand{\tm}{^{\tau,-}}
\newcommand{\tpm}{^{\tau,(\pm)}}

\newcommand{\h}{_{h}}

\newcommand{\per}{{\mathrm{per}}}

\newcommand{\Ihop}{\mathcal{I}_h}
\newcommand{\Ih}[1]{\Ihop\gkla{#1}}

\newcommand{\diam}{\operatorname{diam}}

\newcommand{\g}[1]{\mathfrak{g}_{#1}}

\newcommand{\Zh}{\mathds{Z}\h}

\newcommand{\dbeta}[1]{\,\operatorname{d}\!\beta_{#1}}
\newcommand{\W}{\mathcal{W}}
\newcommand{\dW}{\operatorname{d}\!\bs{\W}}
\newcommand{\sinc}[1]{\bs{\delta}^{#1}\bs{\W}}
\newcommand{\sincb}[3]{\delta^{#1}\beta_{#2}^{#3}}

\newcommand{\p}{{\mathfrak{p}}}

\newcommand{\expected}[1]{\mathds{E}\ekla{#1}}
\newcommand{\Prob}{\mathds{P}}

\DeclareMathOperator*{\esssup}{ess\,sup}

\newcommand{\Ito}{It\^{o}}

\newcommand{\mysigma}{\sigma}

\begin{document}
\title[Approximation of the non-smooth stochastic Cahn--Hilliard equation]{Numerical Approximation of the stochastic Cahn--Hilliard equation with singular potential}
\date{\today}
\author[\v{L}.~Ba\v{n}as,~S.~Metzger]{\v{L}ubom\'ir Ba\v{n}as,~Stefan Metzger}
\address[\v{L}.~Ba\v{n}as]{Department of Mathematics, Bielefeld University, 33501 Bielefeld, Germany}
\email{banas@math.uni-bielefeld.de}
\address[S.~Metzger]{Friedrich--Alexander--Universität Erlangen--Nürnberg,~Cauerstraße 11,~91058~Erlangen,~Germany}
\email{stefan.metzger@fau.de}

%\author[S.~Metzger,~G.~Gr\"un]{Stefan Metzger and G\"unther Gr\"un}
%\address{Friedrich--Alexander Universit\"at Erlangen--N\"urnberg --- Cauerstrasse 11 --- 91058 Erlangen --- Germany}
%\email{stefan.metzger@fau.de, gruen@math.fau.de}

%Primary

%60H35 Computational methods for stochastic equations (aspects of stochastic analysis)
%65M60   	Finite elements, Rayleigh-Ritz and Galerkin methods, finite methods

%Secondary
% 	60H15   	Stochastic partial differential equations (aspects of stochastic analysis)
%   65M12   	Stability and convergence of numerical methods for initial value and initial-boundary value problems involving PDEs
%

%
\selectlanguage{english}

\allowdisplaybreaks

\begin{abstract}
We discuss the numerical approximation of the stochastic Cahn--Hilliard equation with a singular double-obstacle potential and multiplicative conservative noise.
We propose a regularised fully discrete finite element approximation scheme for the problem and show that it satisfies
stability estimates which are uniform with respect to the discretization parameters.
We show convergence of the approximation for vanishing discretization parameters towards a regularised version of the singular stochastic Cahn--Hilliard equation by monotonicity arguments.
Hence, thanks to a uniform $\mathds{H}^1$-estimate for the regularised problem we show that the regularised solution converges towards the pathwise unique probabilistically strong solution of the original singular stochastic Cahn--Hilliard equation.
We conclude by presenting numerical simulations where we compare the regularised numerical approximation to its unregularised counterpart and illustrate the effect of the conservative noise.
\end{abstract}
\maketitle

\section{Introduction}
The deterministic Cahn--Hilliard equation
\begin{align}\label{eq:CHgeneral}
\partial_t u=\Delta \rkla{\varepsilon^{-1}\Psi_{\log}^\prime\trkla{u}-\varepsilon\Delta u},
\end{align}
was originally introduced in \cite{CahnAndHilliard} to model phase separation and de-mixing processes in binary alloys, which occur if the alloys are rapidly cooled below a critical temperature $\theta_c$.
The solution $u$, the so-called phase-field or order parameter,
describes the local composition of the binary alloy (the concentration or volume fraction of its respective components) and
$\Psi_{\log}:\mathds{R}\rightarrow\mathds{R}$ is a double well-potential of the form
\begin{align}
\Psi_{\log}\trkla{u}=\frac{\theta}2\ekla{\trkla{1+u}\ln\trkla{1+u}+\trkla{1-u}\ln\trkla{1-u}}-\frac{\theta_c}2u^2,
\end{align}
which restricts the solution to the physically meaningful range $(-1,+1)$ (here $\theta<\theta_c$ denotes the temperature of the alloy).
If the quench is shallow, i.e., the temperature $\theta$ is still close to the critical temperature $\theta_c$ the logarithmic double-well potential can be approximated by a regular polynomial double-well potential of the form
\begin{align}
\Psi_{\mathrm{pol}}\trkla{u}=\tfrac14\trkla{u^2-1}^2\,.
\end{align}
The Cahn--Hilliard equation with polynomial nonlinearities is well-studied and can be approximated by standard numerical techniques, see for instance \cite{phasefield_review}.
However, this model has the drawback that the phase-field parameter is not restricted to the physically relevant interval $\tekla{-1,+1}$.

In the deep quench limit, i.e., when the temperature is significantly smaller than the critical temperature, an approximation with a singular double-obstacle potential of the form
\begin{align*}
\Psi_{\mathrm{obst}}\trkla{u}:=\left\{\begin{array}{cl}
\tfrac{\theta_c}2\trkla{1-u^2}&\text{if~}\tabs{u}\leq1\,,\\
\infty&\text{if~}\tabs{u}>1\,,
\end{array}\right.
\end{align*}
is the more appropriate choice, cf. \cite{betti} and the references therein. \\
It is also well-known that in particular the early stages of the separation process are heavily influenced by thermal fluctuations which are not included in the deterministic description.
Thus, a stochastic variant of \eqref{eq:CHgeneral} with conservative noise was proposed already in \cite{Cook1970}.

In this work, we consider the numerical approximation of the (singular) stochastic Cahn--Hilliard equation with double-obstacle potential and multiplicative conservative noise.
In particular, we consider the following stochastic variational inclusion
\begin{subequations}\label{eq:model:inequality}
\begin{align}
\mathrm{d}u-\Delta w\dt&=\Phi\trkla{u}\dW\,,\\
w&\in-\varepsilon\Delta u+\frac1\varepsilon \partial\Psi\trkla{u}\,,
\end{align}
\end{subequations}
where $\D$ is a periodic domain (a $d$-dimensional torus) and
and $\partial\Psi$ denotes the subdifferential of the singular energy
\begin{align}\label{eq:obst}
\Psi\trkla{s}:=\left\{\begin{array}{cl}
\tfrac12\trkla{1-s^2}&\text{if~}\tabs{s}\leq1\,,\\
\infty&\text{if~}\tabs{s}>1\,.
\end{array}\right.
\end{align}
We take $\bs{\W}$ to be a sufficiently regular $\mathcal{Q}$-Wiener process (see \eqref{eq:defWienerProcess} below) and
the operator $\Phi$ maps the stochastic process $u$ into the space of Hilbert--Schmidt operators from $\rkla{\mathcal{Q}^{1/2} \mathds{L}^2}^d$ to $\trkla{\mathds{L}^2}^d$.
We consider \eqref{eq:model:inequality} with multiplicative divergence-type noise of the form
%Hence, we consider \eqref{eq:model:inequality} with multiplicative noise of the form
\begin{align}\label{noise}
\Phi\trkla{u}\dW:=\sum_{i=1}^d\sum_{k\in\mathds{Z}}\partial_i\gkla{\mysigma\trkla{u}\lambda_k\g{k}}\dbeta{k}^i\,,
\end{align}
where $\mysigma\,:\,\mathds{R}\rightarrow\mathds{R}$ is a sufficiently regular function (cf.~assumption \ref{assumption:C} below) which satisfies $\sigma(s) = 0$ for $|s| \geq 1$.
The choice of the multiplicative noise \eqref{noise} is motivated by the fact that it preserves the mean value of the solution $u$ in \eqref{eq:model:inequality},
which is also an intrinsic property of the deterministic model \eqref{eq:CHgeneral}.

An equivalent formulation of \eqref{eq:model:inequality} is given as
\begin{subequations}\label{eq:model:Lagrange}
\begin{align}
\mathrm{d}u-\Delta w\dt&=\Phi\trkla{u}\dW\,,\\
w&=-\varepsilon\Delta u-\frac{1}{\varepsilon}u+ \frac{1}{\varepsilon} p\,,
\end{align}
\end{subequations}
where $p$ can be interpreted as a Lagrange multiplier for the constraint $\tabs{u}\leq1$ and $p - u \in\partial\Psi\trkla{u}$.

Due to the versatility of the phase-field modeling framework, the numerical approximation of deterministic and stochastic phase-field models has attracted a lot of attention,
we refer to \cite{phasefield_review} for an overview.
Regarding the recent developments on the numerical approximation of stochastic phase-field models with smooth nonlinearities we mention \cite{larsson18}, \cite{sac_brehier}, \cite{stochch_sharp21}, \cite{stochch_aposter23}, \cite{bp24}, \cite{stochch_sharp25}, \cite{Metzger2024}, \cite{Metzger2025b}, and \cite{Metzger2025}.

One of the earliest results on the numerical approximation of the deterministic Cahn--Hilliard models with obstacle potential (also referred to as the Cahn--Hilliard variational inequality) is \cite{BloweyAndElliott92}; for further developments we refer to
\cite{BarrettAndBlowey97}, \cite{BarrettAndBlowey99}, \cite{voids}, \cite{voids3d}, \cite{aposter_ch}, \cite{gk09}, \cite{chnsmulti} and the references therein.

Analysis of singular stochastic phase-field models has been considered, e.g., in \cite{debussche11}, \cite{Bauzetetal2017}, \cite{scarpa18}.
We note that the analysis of \cite{scarpa18}, which considers the stochastic Cahn--Hilliard equation with singular energy and possibly non-conservative noise, can not be applied in the case of the obstacle potential \eqref{eq:obst}.
In particular, the Lagrange multiplier $p$ in \eqref{eq:model:Lagrange}
can not be obtained as a weak limit of $\mathds{L}^1$-bounded sequences due to the lack of equiintegrability properties in the case of the singular potential \eqref{eq:obst}. 

As far as we are aware, the only result on the numerical approximation of stochastic singular phase-field problems is the recent work \cite{bauzet25}, which considers finite-volume based numerical approximation of the stochastic Allen--Cahn equation with double-obstacle energy.
%\sm{The analysis in \cite{bauzet25} relies on a $\mathds{L}^2$-estimate which is obtained using a regularisation procedure along with monotonicity arguments.
%Since the Cahn--Hilliard equation is an $\mathds{H}^{-1}$-gradient flow, the corresponding monotonicity arguments can only be used to derive an $\mathds{H}^{-1}$-estimate.
%The derivation of the $\mathds{H}^{-1}$-estimate relies on the preservation of the mean value of the solution and therefore requires the (multiplicative) mean-value preserving stochastic forcing \eqref{noise} of divergence-type.}
%\sm{In turn, the divergence-type noise complicated the derivation of the $\mathds{H}^1$-estimate, which is crucial for establishing the limit of the regularised problem and is only obtained only for moments strictly below order $2$ (cf. Lemma~\ref{lem:H1} below).}
The analysis in \cite{bauzet25} utilizes an intrinsic $\mathds{L}^2$-estimate obtained through a regularization procedure and monotonicity arguments. Since the Cahn–Hilliard equation is an $\mathds{H}^{-1}$-gradient flow, monotonicity arguments
naturally yield an $\mathds{H}^{-1}$-estimate. Deriving this estimate relies on the mean value preservation property of the solution,
which motivates the use of the (multiplicative) mean-value preserving stochastic forcing \eqref{noise}.
However, this divergence-type noise complicates the derivation of an $\mathds{H}^1$-estimate that is needed to pass to the limit in the regularized problem;
consequently, the estimate is valid only for moments strictly below order $2$ (cf. Lemma 6.1 below).

%\footnote{to be discussed: in \cite{DiPrimioGrasselliScarpa2024} they seem remove the degeneracy of the noise condition}

In this work we propose an implementable numerical approximation of the singular stochastic Cahn--Hilliard equation with conservative noise \eqref{eq:model:inequality}.
The numerical approximation employs a regularisation of the singular energy, backward Euler discretization in time and a finite element approximation in space.
We establish convergence of the approximation towards the probabilistically strong (and pathwise unique) solution of \eqref{eq:model:Lagrange} by first passing to the limit with respect
to the time and space discretization parameters and then with respect to the regularisation.
The convergence analysis relies on uniform a priori estimates for the approximate solution and monotonicity arguments. % and does not employ any Skorokhod-type compactness arguments.
%Thus, we obtain convergence towards probabilistically strong solutions, which can be shown to be pathwise unique.

The remainder of this paper is structured as follows:
In Section \ref{sec:notation}, we introduce the relevant finite and infinite dimensional spaces, state the necessary assumptions, and propose a regularised fully discrete scheme.
In Section \ref{sec:boundsht}, we prove the existence of solutions to this scheme and derive $\trkla{h,\tau}$-independent estimates which we use in Section \ref{sec:limitht} to pass to the limit $\trkla{h,\tau}\rightarrow0$.
The pathwise uniqueness of the limit process is established in Section \ref{sec:uniquenessdelta}. 
In Sections \ref{sec:est} and \ref{sec:limitdelta}, we derive uniform estimates which are necessary to pass to the limit with respect to the regularisation parameter and establish the existence of probabilistically strong pathwise unique solutions to \eqref{eq:model:Lagrange}.
We conclude by presenting numerical simulations in Section \ref{sec:numerics}.

\section{Notation and preliminaries}\label{sec:notation}
We discuss the numerical approximation of the {singular stochastic Cahn--Hilliard equation} on the $d$-dimensional torus $\D\equiv\trkla{0,1}^d\subset\mathds{R}^d$ with $d\in\tgkla{1,2,3}$.
We denote the space of $k$-times weakly differentiable periodic functions with weak derivatives in $\mathds{L}^p\equiv L^p\trkla{\D}$ by $\mathds{W}^{k,p}_{\per}$. 
For $p=2$, we denote the Hilbert spaces $\mathds{W}^{k,2}_{\per}$ by $\mathds{H}^k_{\per}$.
We will denote the $\mathds{L}^2$-inner product by $\trkla{\cdot,\cdot}$ and the corresponding norm by $\norm{\cdot}_{\mathds{L}^2}\equiv\norm{\cdot}$.
The symbol $\skla{\cdot,\cdot}$ stands for the duality pairing between $\mathds{H}^1_{\per}$ and its dual space which is denoted by $\mathds{H}^{-1}_{\per}$.
For $u\in\mathds{H}^{-1}_{\per}$ we define a generalized mean value by
\begin{align}
\trkla{u}_{\D}:=\abs{\D}^{-1}\skla{u,1}\,.
\end{align}
The inner product in $\mathds{H}^{-1}_{\per}$ is defined as
\begin{align}
\trkla{a,b}_{-1}:=\rkla{\nabla\trkla{-\Delta^{-1}}\trkla{a-\trkla{a}_{\D}},\nabla\trkla{-\Delta^{-1}}\trkla{b-\trkla{b}_{\D}}}+\trkla{a}_{\D}\trkla{b}_{\D}
\end{align}
and the associated norm is defined via $\norm{a}_{-1}:=\sqrt{\trkla{a,a}}$.
Given a Banach space $X$ and a set $I$, the symbol $L^p\trkla{I;X}$ ($p\in\tekla{1,\infty}$) stands for the Bochner space of strongly measurable $L^p$-integrable functions on $I$ with values in $X$.
If $X$ is only separable (and not reflexive), we follow the notation used in \cite[Chapter 0.3]{FeireislNovotny17} and denote the dual space of $L^{p/\trkla{p-1}}\trkla{I;X}$ ($p\in\trkla{1,\infty}$) by $L^p_{\operatorname{weak-(*)}}\trkla{I,X^\prime}$.
The symbol $C^{k,\alpha}\trkla{I;X}$ with $k\in\mathds{N}_0$ and $\alpha\in(0,1]$ denotes the space of $k$-times continuously differentiable functions from $I$ to $X$ whose $k$-th derivatives are Hölder continuous with Hölder exponent $\alpha$. If $I=X=\mathds{R}$, we shall write $C^{k,\alpha}\trkla{\mathds{R}}$.\\

The stochastic source term on the right-hand side of \eqref{eq:model:inequality} is governed by a $\mathcal{Q}$-Wiener process $\bs{\W}$ defined on a filtered probability space $\trkla{\Omega,\mathcal{A},\mathcal{F},\mathds{P}}$ with a complete right-continuous filtration $\mathcal{F}$ and an operator $\Phi$ mapping into the space of Hilbert--Schmidt operators.
We shall assume that
\begin{itemize}
\labitem{\textbf{(W1)}}{assumption:W1} $\mathcal{Q}\,:\,\mathds{L}^2\rightarrow\mathds{L}^2$ is a trace class operator with the representation
\begin{align}
\mathcal{Q} g:=\sum_{k\in\mathds{Z}}\abs{\lambda_k}^2\trkla{g,\g{k}}\g{k}&&\text{for~}g\in\mathds{L}^2\,,
\end{align}
where $\trkla{\lambda_k}_{k\in\mathds{Z}}$ are given positive real numbers and $\trkla{\g{k}}_{k\in\mathds{Z}}$ is an orthonormal basis of $\mathds{L}^2$ consisting of smooth periodic functions.
\end{itemize}
Thus, $\mathcal{Q}^{1/2}$ is a Hilbert--Schmidt operator from $\mathds{L}^2$ to $\mathds{L}^2$.
Denoting the Cartesian basis vectors of $\mathds{R}^d$ by $\trkla{\mathfrak{e}_i}_{i=1,\ldots,d}$, we can write the vector-valued $\mathcal{Q}$-Wiener process $\bs{\W}$ as
\begin{align}\label{eq:defWienerProcess}
\bs{\W}\trkla{t}:=\sum_{i=1}^d\sum_{k\in\mathds{Z}}\lambda_k\g{k}\mathfrak{e}_i\beta_k^i\trkla{t}\,,
\end{align}
where $\trkla{\beta_k^i}_{k,i}$ is a family of independent Brownian motions.
We shall assume that this process is colored in the sense that
\begin{itemize}
\labitem{\textbf{(W2)}}{assumption:W2} There exists a positive constant $\widehat{C}$ such that
\begin{align}
\sum_{k\in\mathds{Z}}\abs{\lambda_k}^2\norm{\g{k}}_{\mathds{W}^{1,\infty}}^2+\sum_{k\in\mathds{Z}}\abs{\lambda_k}^2\norm{\g{k}}_{\mathds{H}^{2}}^2\leq\widehat{C}\,.
\end{align}
\end{itemize}
The operator $\Phi$ maps a stochastic process $u$ into the space of Hilbert--Schmidt operators from $\trkla{\mathcal{Q}^{1/2}\mathds{L}^2}^d$ to $\mathds{L}^2$ and is defined via
\begin{align}
\Phi\trkla{u}\bs{g}:=\sum_{i=1}^d\sum_{k\in\mathds{Z}}\partial_i\tgkla{\mysigma\trkla{u}\trkla{\bs{g}\cdot\mathfrak{e}_i,\g{k}}\g{k}}
\end{align}
for all $\bs{g}\in\rkla{\mathcal{Q}^{1/2}\mathds{L}^2}^d$.
Here, $\mysigma$ satisfies
\begin{itemize}
\labitem{\textbf{(C)}}{assumption:C} $\mysigma\in W^{2,\infty}\trkla{\mathds{R}}\cap C_0^{1}\trkla{\trkla{-1,1}}$. 
\end{itemize}

We shall now collect our assumptions on the discrete approximation.
Concerning the discretization with respect to time, we shall assume that
\begin{itemize}
\labitem{\textbf{(T)}}{assumption:T} the time interval $I:=\tekla{0,T}$ is subdivided into $N$ equidistant intervals $I_n$ given by nodes $\trkla{t\nn}_{n=0,\ldots,N}$ with $t^0=0$, $t^N=T$, and $t\nn-t\no=:\tau=\tfrac{T}N$.
\end{itemize}
Our spacial discretization is based on partitions $\mathcal{T}\h$ of $\D$ depending on a discretization parameter $h>0$ satisfying the following assumptions:
\begin{itemize}
\labitem{\textbf{(S1)}}{assumption:S1} The family of $\tgkla{\mathcal{T}\h}_{h>0}$ is a quasi-uniform family of partitions of $\D$ into disjoint, open simplices $K$ which satisfy
\begin{align*}
\overline{\D}=\bigcup_{K\in\mathcal{T}\h}\overline{K}&&\text{width~}\max_{K\in\mathcal{T}\h}\diam\trkla{K}\leq h\,.
\end{align*}
\labitem{\textbf{(S2)}}{assumption:S2} The subdivisions are right-angled in the sense that for each simplicial element $K\in\mathcal{T}\h$, there exists a vertex $\bs{x}_0\trkla{K}$ such that all edges connecting $\bs{x}_0\trkla{K}$ with the remaining vertices of $K$ are perpendicular to each other.
\end{itemize}
Using these subdivisions of $\D$, we define the spaces $\trkla{\mathds{V}\h}_{h>0}$ of periodic, continuous, piecewise linear finite element functions on $\trkla{\mathcal{T}\h}\h$, which are spanned by functions $\tgkla{\chi_{h,i}}_{i\in J\h}$ ($J\h:=\tgkla{1,\ldots,\dim\mathds{V}\h}$) forming a dual basis to the vertices $\tgkla{\bs{x}_{h,i}}_{i\in J\h}$.
For each $h>0$, we introduce the nodal interpolation operator $\Ihop\,:\,C^0_{\mathrm{per}}\trkla{\overline{\D}}\rightarrow\mathds{V}\h$ which is defined by
\begin{align}
\Ih{a}:=\sum_{k=1}^{\dim\mathds{V}\h}a\trkla{\bs{x}_{h,k}}\chi_{h,k}\,.
\end{align}
The interpolation operator allows us to define discrete versions of the standard $\mathds{L}^p$-norms, which are equivalent to their continuous counterparts, i.e., for $p\in[1,\infty)$ there exist $h$-independent constants $c,C>0$ such that
\begin{align}
c\norm{U}_{\mathds{L}^p}^p\leq \iD\Ih{\abs{U}^p}\dx\leq C\norm{U}_{\mathds{L}^p}^p
\end{align}
for all $U\in\mathds{V}\h$.
We also consider the discrete $L^2$-inner product $\trkla{U,V}\h:=\iD\Ih{UV}\dx$ and its induced norm $\norm{\cdot}\h$.
For future reference, we recall the following error estimates which were proven e.g.~in Lemma 2.1 in \cite{Metzger2020}.
\begin{lemma}\label{lem:interpolation}
Let $\mathcal{T}\h$ satisfy \ref{assumption:S1} and let $p\in[1,\infty)$ and $1\leq q\leq\infty$.
Then, for $q^*=\tfrac{q}{q-1}$, if $q<\infty$, and $q^*=1$, if $q=\infty$, the estimates
\begin{subequations}
\begin{align*}
\norm{\trkla{I-\Ihop}\tgkla{U V}}_{\mathds{L}^p}&\leq Ch^2\norm{\nabla U}_{\mathds{L}^{pq}}\norm{\nabla V}_{\mathds{L}^{pq^*}}\,,\\
\norm{\trkla{I-\Ihop}\tgkla{UV}}_{\mathds{W}^{1,p}}&\leq Ch\norm{\nabla U}_{\mathds{L}^{pq}}\norm{\nabla V}_{\mathds{L}^{pq^*}}
\end{align*}
\end{subequations}
hold true for all $U,V\in\mathds{V}\h$.
\end{lemma}
We define the discrete Laplacian $\Delta\h\,:\,\mathds{V}\h\rightarrow\mathds{V}\h$ as
\begin{align}
\trkla{\Delta\h U, V}\h =-\iD\nabla U\cdot\nabla V\dx&&\text{for~}U,V\in\mathds{V}\h\,.
\end{align}
On the subset of finite element functions with mean-value zero $\mathds{V}_h^0\subset\mathds{V}\h$ we can define the inverse of the discrete Laplacian as
\begin{align}\label{eq:invLap}
\rkla{\nabla\trkla{-\Delta\h^{-1}}U,\nabla V}=\trkla{U,V}\h\qquad\forall V\in\mathds{V}\h\,,U\in\mathds{V}\h^0\,,
\end{align}
and a discrete version of the $\mathds{H}^{-1}$-inner product via
\begin{align}
\trkla{U,V}_*=\rkla{\nabla\trkla{-\Delta\h}^{-1}\trkla{U-\trkla{U}_{\D}},\nabla\trkla{-\Delta\h}^{-1}\trkla{V-\trkla{V}_{\D}}} +\trkla{U}_{\D}\trkla{V}_{\D}\,.
\end{align}
The induced norm will be denoted by $\norm{U}_*:=\sqrt{\trkla{U,U}_*}$.
Note that 
\begin{align*}
\norm{U}_{-1}^2 & =\norm{\nabla\trkla{-\Delta^{-1}}\trkla{U-\trkla{U}_{\D}}}^2+\trkla{U}_{\D}^2\\
 & \leq C\norm{\nabla\trkla{-\Delta\h^{-1}}\trkla{U-\trkla{U}_{\D}}}^2+\trkla{U}_{\D}^2\,.
\end{align*}
Below we denote $\norm{U}_*^2 := \norm{\nabla\trkla{-\Delta\h^{-1}}\trkla{U-\trkla{U}_{\D}}}^2+\trkla{U}_{\D}^2$.

We shall denote the standard $\mathds{H}^1$-stable $\mathds{L}^2$-projection onto $\mathds{V}\h$ as $\mathcal{P}:\mathds{L}^2\rightarrow \mathds{V}\h$ and the Ritz projection $\mathcal{R}:\mathds{W}^{1,p}_{\per} \rightarrow \mathds{V}\h$ as
\begin{subequations}
\begin{align}
\iD\nabla \mathcal{R}v\cdot\nabla U\dx&\,=\iD\nabla v\cdot\nabla U\dx&&\forall U\in\mathds{V}\h\text{~and~}v\in\mathds{W}^{1,p}_{\per}\,,\\
\iD\mathcal{R}v\dx&=\iD v\dx\,.
\end{align}
\end{subequations}
The Ritz projection satisfies (cf.~\cite[Theorem A.2]{GiraultAndRaviart})
\begin{align}\label{eq:convergenceRitz}
\norm{v-\mathcal{R}v}_{\mathds{L}^p} +h\norm{\nabla v-\nabla\mathcal{R}v}_{\mathds{L}^p}\leq Ch^{s+1}\norm{v}_{\mathds{W}^{s+1,p}},
\end{align}
for all $v\in\mathds{W}^{s+1,p}_{\per}$, $s\in\tgkla{0,1}$, and $p\geq2$.
For simplicity, we assume that the continuous and discrete initial data are deterministic and satisfy
\begin{itemize}
\labitem{\textbf{(I)}}{assumption:I} $u_0\in\mathds{H}^1_{\per}$ with $\abs{u_0}\leq 1$ almost everywhere in $\D$, $U^0:=\mathcal{P}{u_0}\in\mathds{V}\h$ with $\mathfrak{m}:=\trkla{u_0}_{\D}=\trkla{U^0}_{\D}\in\trkla{-1,1}$.
\end{itemize}
In our discrete scheme, we shall only consider a finite amount of modes of the $\mathcal{Q}$-Wiener process. Thus, we approximate $\Phi$ by
\begin{align}\label{eq:defPhih}
\Phi\h\trkla{U}\bs{g}:=\sum_{i=1}^d\sum_{k\in\mathds{Z}\h}\partial_i\Ih{\mysigma\trkla{U}\trkla{\bs{g}\cdot\mathfrak{e}_i,\g{k}}\g{k}},
\end{align}
for all $\bs{g}\in\trkla{\mathcal{Q}^{1/2}\mathds{L}^2}^d$.
Here, $\tgkla{\mathds{Z}\h}\h$ is an $h$-dependent family of finite subsets of $\mathds{Z}$ satisfying
\begin{itemize}
\labitem{\textbf{(Z)}}{assumption:Z} $\displaystyle\lim_{h\searrow0}\mathds{Z}\h=\mathds{Z}$.
\end{itemize}
The fully discrete finite element approximation of \eqref{eq:model:inequality} reads:
for a given $\mathds{V}\h$-valued random variable $U\no$, find a $\mathds{V}\h\times\mathds{V}\h$-valued random variable $\trkla{U\nn,W\nn}$ such that
\begin{subequations}\label{eq:model:disc:inequality}
\begin{align}
\trkla{U\nn-U\no,\varphi\h}\h+\tau\trkla{\nabla W\nn,\nabla\varphi\h}&=\trkla{\Phi\h\trkla{U\no}\sinc{n},\varphi\h}\,,\\
\varepsilon\rkla{\nabla U\nn,\nabla\tekla{\psi\h-U\nn}}&\geq \rkla{W\nn+\frac1\varepsilon U\no,\psi\h-U\nn}\h\,,
\end{align}
for all $\varphi\h,\psi\h\in\mathds{V}\h$.
Here, we denoted the increment of the $\mathcal{Q}$-Wiener process by $\displaystyle \sinc{n}:=\bs{\W}\trkla{t\nn}-\bs{\W}\trkla{t\no}=:\sum_{i=1}^d\sum_{k\in\mathds{Z}}\lambda_k\g{k}\sincb{n}{k}{i}$.
\end{subequations}

Following the approach in \cite{BloweyAndElliott91}, for $\delta > 0$ we approximate the singular part of the potential $\Psi$ by a convex function $F_\delta$
\begin{align}\label{eq:def:Fdelta}
F_\delta\trkla{r}:=\left\{\begin{array}{cl}
\tfrac1{2\delta}\trkla{r-\trkla{1+\tfrac\delta2}}^2+\tfrac\delta{24} &\text{if~}r\geq1+\delta\,,\\
\tfrac{1}{6\delta^2}\trkla{r-1}^3&\text{if~}1<r<1+\delta\,,\\
0&\text{if~}-1\leq r\leq 1\,,\\
-\tfrac1{6\delta^2}\trkla{r+1}^3&\text{if~}-1-\delta<r<-1\,,\\
\tfrac1{2\delta}\trkla{r+\trkla{1+\tfrac\delta2}}^2+\tfrac\delta{24}&\text{if~}r\leq -1-\delta\,.
\end{array}\right.
\end{align}

In the subsequent sections, we shall analyze the following regularised version of \eqref{eq:model:disc:inequality}:
\begin{subequations}\label{eq:model:disc:regularised}
\begin{align}
\rkla{U\nn-U\no,\varphi\h}\h+\tau\rkla{\nabla W\nn,\nabla\varphi\h}&=\rkla{\Phi\h\trkla{U\no}\sinc{n},\varphi\h}&&\forall\varphi\h\in\mathds{V}\h\,,\label{eq:model:disc:regularised:1}\\
\varepsilon\rkla{\nabla U\nn,\nabla\psi\h}+\frac1\varepsilon\rkla{f_\delta\trkla{U\nn}-U\no,\psi\h}\h&=\rkla{W\nn,\psi\h}\h&&\forall\psi\h\in\mathds{V}\h\,,\label{eq:model:disc:regularised:2}
\end{align}
\end{subequations}
where $f_\delta=F^\prime_\delta$ is the derivative of the convex function $F_\delta$. 
Recalling the definition of $\Phi\h$ in \eqref{eq:defPhih} provides $\trkla{U\nn}_{\D}=\trkla{U^0}_{\D}=\mathfrak{m}$ for all $n$.

The existence of solutions to \eqref{eq:model:disc:regularised} follows by an energy-type argument along the lines of \cite[Lemmas 5.4]{stochch_sharp25}.

\section{$\trkla{h,\tau}$-independent bounds}\label{sec:boundsht}
In this section, we study the problem for fixed $\delta>0$ and derive $\trkla{h,\tau}$-independent bounds which will be used in Section \ref{sec:limitht} to pass to the limit $\trkla{h,\tau}\searrow0$.
\begin{lemma}\label{lem:h-1estimate}
Let the assumptions \ref{assumption:T}, \ref{assumption:S1}, \ref{assumption:I}, \ref{assumption:W1}, \ref{assumption:W2}, \ref{assumption:C} hold true and let $0<\delta\leq1/8$.
Then, for every $1\leq \p<\infty$, there exists a constant $C\equiv C\trkla{T,\p}$ which is independent of $h$, $\tau$, and $\delta$ such that
\begin{multline}
\expected{\max_{0\leq\tilde{N}\leq N}\norm{U^{\tilde{N}}-\mathfrak{m}}_*^{2\p}}+\expected{\rkla{\sum_{n=1}^N\norm{U\nn-U\no}_*^2}^\p} +\expected{\rkla{\sum_{n=1}^N\tau\varepsilon\norm{\nabla U\nn}_{\mathds{L}^2}^2}^\p}\\
+\expected{\rkla{\sum_{n=1}^N\frac\tau\varepsilon\norm{U\nn-U\no}\h^2}^\p} +\expected{\rkla{\sum_{n=1}^N\frac\tau\varepsilon \rkla{f_\delta\trkla{U\nn},U\nn-\mathfrak{m} }\h}^\p} \leq C
\end{multline}
for all $\p\in[1,\infty)$.
Furthermore, the estimates $\rkla{f_\delta\trkla{U\nn},U\nn-\mathfrak{m} }\h\geq \rkla{F_\delta\trkla{U\nn},1}\h\geq0$ and $\trkla{F_\delta\trkla{U\nn},1}\h-2\norm{U\nn}\h^2\geq -\mathcal{C}$ with $\mathcal{C}$ independent of $h$, $\tau$, and $\delta$ hold true.
\end{lemma}
\begin{proof}
As $\trkla{U\nn-\mathfrak{m}}\in\mathds{V}\h^0$, we can choose $\varphi\h=-\Delta\h^{-1}\rkla{U\nn-\mathfrak{m}}$ and $\psi\h=U\nn$ in \eqref{eq:model:disc:regularised} and obtain using \eqref{eq:invLap} and Young's inequality
\begin{align}\label{eq:h-1estimate}
\begin{split}
\tfrac12&\,\norm{U\nn-\mathfrak{m}}_*^2 +\tfrac14\norm{U\nn-U\no}_*^2-\tfrac12\norm{U\no-\mathfrak{m}}_*^2 +\tau\varepsilon\norm{\nabla U\nn}_{\mathds{L}^2}^2\\
&\,+\frac\tau\varepsilon\rkla{f_\delta\trkla{U\nn},U\nn-\mathfrak{m}}\h-\frac\tau\varepsilon\rkla{U\no,U\nn-\mathfrak{m}}\h\\
\leq&\,\rkla{\Phi\h\trkla{U\no}\sinc{n},-\Delta\h^{-1}\rkla{U\no-\mathfrak{m}}} +\norm{\sum_{i=1}^d\sum_{k\in\Zh}\Ih{\mysigma\trkla{U\no}\g{k}\lambda_k}\sincb{n}{k}{i}}_{\mathds{L}^2}^2\,.
\end{split}
\end{align}
Using the convexity of $F_\delta$, we immediately obtain
\begin{align}
\rkla{f_\delta\trkla{U\nn},U\nn-\mathfrak{m}}\h\geq \rkla{F_\delta\trkla{U\nn},1}\h-\rkla{F_\delta\trkla{\mathfrak{m}},1}\h=\rkla{F_\delta\trkla{U\nn},1}\h\geq0\,,
\end{align}
as $\mathfrak{m}\in\rkla{-1,1}$ (cf.~Assumption \ref{assumption:I}).
For future reference, we shall also show that $\rkla{F_\delta\trkla{U\nn},1}\h-2\norm{U\nn}\h^2$ is bounded from below if $\delta\leq 1/8$:
Denoting the dual basis to the nodes $\trkla{\bs{x}_{h,i}}_{i\in J\h}$ by $\trkla{\chi_{h,i}}_{i\in J\h}$, we introduce the subsets $J\h^+:=\gkla{i\in J\h\,:\,U\nn\trkla{\bs{x}_{h,i}}\geq8}$, $J\h^-:=\gkla{i\in J\h\,:\,U\trkla{\bs{x}_{h,i}}\leq -8}$, and $J\h^b:= J\h\setminus\trkla{J\h^+\cup J\h^-}$ and compute
\begin{align}
\begin{split}
\trkla{F_\delta\trkla{U\nn},1}\h&-2\norm{U\nn}\h^2=\sum_{j\in J\h}\rkla{F_\delta\trkla{U\nn\trkla{\bs{x}_{h,i}}}-2\abs{U\nn\trkla{\bs{x}_{h,i}}}^2}\iD\chi_{h,i}\dx\\
=&\,\sum_{i\in J\h^+} \rkla{\tfrac1{2\delta}\rkla{U\nn\trkla{\bs{x}_{h,i}}-\trkla{1+\tfrac\delta2}}^2+\tfrac\delta{24}-2\abs{U\nn\trkla{\bs{x}_{h,i}}}^2}\iD\chi_{h,i}\dx\\
&\,+\sum_{i\in J\h^+} \rkla{\tfrac1{2\delta}\rkla{U\nn\trkla{\bs{x}_{h,i}}+\trkla{1+\tfrac\delta2}}^2+\tfrac\delta{24}-2\abs{U\nn\trkla{\bs{x}_{h,i}}}^2}\iD\chi_{h,i}\dx\\
&\,+\sum_{i\in J\h^b}\rkla{F_\delta\trkla{U\nn\trkla{\bs{x}_{h,i}}}-2\abs{U\nn\trkla{\bs{x}_{h,i}}}^2}\iD\chi_{h,i}\dx\,.
\end{split}
\end{align}
As $\iD\chi_{h,i}\dx\geq0$ for all $i\in J\h$, we can estimate the first term using $\trkla{1+\tfrac\delta2}\leq 2$ and $U\nn\geq8$ to deduce
\begin{multline}
\tfrac1{2\delta}\abs{U\nn\trkla{\bs{x}_{h,i}}}^2-\tfrac1\delta U\nn\trkla{\bs{x}_{h,i}}\trkla{1+\tfrac\delta2}+\tfrac1{2\delta}\trkla{1+\tfrac\delta2}^2+\tfrac\delta{24}-2\abs{U\nn\trkla{\bs{x}_{h,i}}}^2\\
\geq\tfrac1{2\delta}\abs{U\nn\trkla{\bs{x}_{h,i}}}^2-\tfrac1{4\delta}\abs{U\nn\trkla{\bs{x}_{h,i}}}^2-2\abs{U\nn\trkla{\bs{x}_{h,i}}}^2\geq 0
\end{multline}
for all $i\in J\h^+$ and $\delta\leq \tfrac18$.
The second term can be dealt with analogously.
To estimate the last term, we use the nonnegativity of $F_\delta$ and obtain
\begin{align}
\sum_{i\in J\h^b}\rkla{F_\delta\trkla{U\nn\trkla{\bs{x}_{h,i}}}-2\abs{U\nn\trkla{\bs{x}_{h,i}}}^2}\iD\chi_{h,i}\dx\geq -128\abs{\D}\,,
\end{align}
which shows $\displaystyle \trkla{F_\delta\trkla{U\nn},1}\h-2\norm{U\nn}\h^2\geq-128\tabs{\D}=:\mathcal{C}$.\\
We continue by rewriting $-\trkla{U\no,U\nn-\mathfrak{m}}_h$ using the well-known identity $2\trkla{a-b,a}=a^2-b^2+\trkla{a-b}^2$ which provides
\begin{align}
\begin{split}
-&\trkla{U\no,U\nn-\mathfrak{m}}\h=\trkla{U\no,U\no-U\nn}\h-\norm{U\no}\h^2+\tabs{\D}\mathfrak{m}^2\\
&=\,\tfrac12\norm{U\no}\h^2 +\tfrac12\norm{U\nn-U\no}\h^2-\tfrac12\norm{U\nn}\h^2-\norm{U\no}\h^2 +\tabs{\D}\mathfrak{m}^2\\
&=\,-\tfrac12\norm{U\no}\h^2+\tfrac12\norm{U\nn-U\no}\h^2-\tfrac12\norm{U\nn}\h^2+\tabs{\D}\mathfrak{m}^2\,.
\end{split}
\end{align}
Hence, summing \eqref{eq:h-1estimate} from $n=1$ to $\tilde{N}\leq N$ provides
\begin{align}
\begin{split}
\tfrac12&\norm{U^{\tilde{N}}-\mathfrak{m}}_*^2+\tfrac14\sum_{n=1}^{\tilde{N}}\norm{U\nn-U\no}_*^2 +\varepsilon\sum_{n=1}^{\tilde{N}}\tau\norm{\nabla U\nn}_{\mathds{L}^2}^2\\
&+\frac\tau{2\varepsilon}\sum_{n=1}^{\tilde{N}}\rkla{f_\delta\trkla{U\nn},U\nn-\mathfrak{m}}\h +\frac\tau{2\varepsilon}\sum_{n=1}^{\tilde{N}}\rkla{\trkla{F_\delta\trkla{U\nn},1}\h-2\norm{U\nn}\h^2}\\
&+\frac\tau{2\varepsilon}\sum_{n=1}^{\tilde{N}}\norm{U\nn-U\no}\h^2\\
\leq&\,\tfrac12\norm{U^0-\mathfrak{m}}_*^2 +\sum_{n=1}^{\tilde{N}}\rkla{\Phi\h\trkla{U\no}\sinc{n},-\Delta\h^{-1}\trkla{U\no-\mathfrak{m}}}\\
&+\sum_{n=1}^{\tilde{N}}\norm{\sum_{i=1}^d\sum_{k\in\Zh} \Ih{\mysigma\trkla{U\no}\g{k}\lambda_k}\sincb{n}{k}{i}}_{\mathds{L}^2}^2\,.
\end{split}
\end{align}
Taking the $\p$-th power on both sides and recalling the lower bound on $\rkla{F_\delta\trkla{U\nn},1}\h-2\norm{U\nn}\h^2$, we obtain
\begin{align}\label{eq:h-1estimate:p}
\begin{split}
&\norm{U^{\tilde{N}}-\mathfrak{m}}_*^{2\p}+\rkla{\sum_{n=1}^{\tilde{N}}\norm{U\nn-U\no}_*^2}^{\p} +\rkla{\varepsilon\sum_{n=1}\tau\norm{\nabla U\nn}_{\mathds{L}^2}^2}^{\p}\\
&+\rkla{\frac{\tau}{\varepsilon}\sum_{n=1}^{\widetilde{N}}\rkla{f_\delta\trkla{U\nn},U\nn-\mathfrak{m}}\h}^{\p} +\rkla{\frac{\tau}{\varepsilon}\sum_{n=1}^{\tilde{N}}\norm{U\nn-U\no}\h^2}^\p\\
\leq&\,C\norm{U^0-\mathfrak{m}}_*^{2\p} + \rkla{\sum_{n=1}^{\tilde{N}}\rkla{\Phi\h\trkla{U\no}\sinc{n},-\Delta\h^{-1}\trkla{U\no-\mathfrak{m}}}}^\p\\
&\,+C\rkla{\sum_{n=1}^{\tilde{N}}\norm{\sum_{i=1}^d\sum_{k\in\Zh} \Ih{\mysigma\trkla{U\no}\g{k}\lambda_k}\sincb{n}{k}{i}}_{\mathds{L}^2}^2}^\p +C\\
=:&\,C\norm{U^0-\mathfrak{m}}_*^{2\p} +R_1 +R_2 +C\,,
\end{split}
\end{align}
where the constants depend on $\p$ but not on $h$, $\tau$, or $\delta$.
To estimate the moments of the martingales by their quadratic variation processes, we use the Burkholder--Davis--Gundy inequality and deduce
\begin{align}
\begin{split}
\expected{R_1}\leq&\,C\expected{\max_{1\leq l\leq\tilde{N}} \abs{\sum_{n=1}^l\sum_{i=1}^d\sum_{k\in\Zh}\rkla{ \Ih{\mysigma\trkla{U\no}\lambda_k\g{k}}\sincb{n}{k}{i},\partial_i\trkla{-\Delta\h^{-1}}\rkla{U\no-\mathfrak{m}}}}^\p}\\
\leq&\,C\sum_{n=1}^{\tilde{N}}\tau\expected{\rkla{\sum_{k\in\Zh}\sum_{i=1}^d\abs{\rkla{\Ih{\mysigma\trkla{U\no}\lambda_k\g{k}},\partial_i\trkla{-\Delta\h^{-1}}\rkla{U\no-\mathfrak{m}}}}^2 }^{\p/2}}\\
\leq&\, C\sum_{n=1}^{\tilde{N}}\tau\expected{\norm{U\no-\mathfrak{m}}_*^{2\p}}+ C\tilde{N}\tau\,,
\end{split}\\
\begin{split}
\expected{R_2}\leq&\,C\sum_{n=1}^{\tilde{N}}\tau\expected{\rkla{\sum_{i=1}^d\sum_{k\in\Zh}\norm{\Ih{\mysigma\trkla{U\no}\lambda_k\g{k}}}_{\mathds{L}^2}^2}^\p}\leq C\tilde{N}\tau\,.
\end{split}
\end{align}
Hence, taking the expected value in \eqref{eq:h-1estimate:p} provides
\begin{align}
\expected{\norm{U^{\tilde{N}}-\mathfrak{m}}_*^{2\p}}\leq C\norm{U^0-\mathfrak{m}}_*^{2\p} +C\sum_{n=1}^{\tilde{N}}\tau\expected{\norm{U\no-\mathfrak{m}}_*^{2\p}} +C\,.
\end{align}
Thus, an application of (a discrete version of) Gronwall's inequality yields
\begin{align}
\max_{0\leq\tilde{N}\leq N}\expected{\norm{U^{\tilde{N}}-\mathfrak{m}}_*^{2\p}}\leq C\,.
\end{align}
We conclude the proof by taking the maximum with respect to $\tilde{N}$ in \eqref{eq:h-1estimate:p} before taking the expected value, which provides
\begin{align}
\begin{split}
&\expected{\max_{0\leq\tilde{N}\leq N}\norm{U^{\tilde{N}}-\mathfrak{m}}_*^{2\p}}+\expected{\rkla{\sum_{n=1}^{{N}}\norm{U\nn-U\no}_*^2}^{\p}} +\expected{\rkla{\varepsilon\sum_{n=1}^N\tau\norm{\nabla U\nn}_{\mathds{L}^2}^2}^{\p}}\\
&\qquad\quad+\expected{\rkla{\frac{\tau}{\varepsilon}\sum_{n=1}^{N}\rkla{f_\delta\trkla{U\nn},U\nn-\mathfrak{m}}\h}^{\p}} +\expected{\rkla{\frac{\tau}{\varepsilon}\sum_{n=1}^{N}\norm{U\nn-U\no}\h^2}^\p}\\
&\quad\leq\, C\norm{U^0-\mathfrak{m}}_*^{2\p} +C\sum_{n=1}^{N}\tau\expected{\norm{U\no-\mathfrak{m}}_*^{2\p}} +C\leq C\,.
\end{split}
\end{align}
\end{proof}
Combining the results from Lemma \ref{lem:h-1estimate} with a straightforward interpolation argument shows that the solutions are also uniformly bounded with respect to a discrete $L^q\trkla{L^4\trkla{L^2}}$-norm:
\begin{corollary}\label{cor:lql4l2}
Let the assumptions of Lemma \ref{lem:h-1estimate} hold true. Then
\begin{align}
\expected{\rkla{\sum_{n=1}^N\tau\norm{U\nn}\h^4}^q}\leq C
\end{align}
for all $q\in[1,\infty)$ with $C$ depending on $q$ but not on $h$, $\tau$, or $\delta$.
\end{corollary}
\begin{proof}
The claim follows by the straightforward computation
\begin{align}
\begin{split}
&\expected{\rkla{\sum_{n=1}^N\tau\norm{U\nn}\h^4}^q}\leq \expected{\rkla{\sum_{n=1}^N\tau\norm{U\nn-\mathfrak{m}}\h^4}^q}+C\\
&\qquad\leq \expected{\rkla{\sum_{n=1}^N\tau\norm{\nabla U\nn}_{\mathds{L}^2}^2\norm{U\nn-\mathfrak{m}}_*^2}^q}+C\\
&\qquad\leq \expected{\max_{1\leq n\leq N}\norm{U\nn-\mathfrak{m}}_*^{4q}}^{1/2}\expected{\rkla{\sum_{n=1}^N\tau\norm{\nabla U\nn}_{\mathds{L}^2}^2}^{2q}}^{1/2}\leq C\,.
\end{split}
\end{align}
\end{proof}
Next, we improve the $\mathds{H}^{-1}$-estimate from Lemma \ref{lem:h-1estimate} and deduce the following $\mathds{L}^2$-estimate:
\begin{lemma}\label{lem:l2estimate}
Let the assumptions \ref{assumption:T}, \ref{assumption:S1}, \ref{assumption:S2}, \ref{assumption:I}, \ref{assumption:W1}, \ref{assumption:W2}, \ref{assumption:C} hold true and let $\delta\leq 1/8$.
Then, there exists a constant $C$ which is independent of $h$ and $\tau$ such that
\begin{align}
\expected{\max_{1\leq\tilde{N}\leq N}\norm{U^{\tilde{N}}}\h^{3}}+\expected{\rkla{\sum_{n=1}^N\norm{U\nn-U\no}\h^2}^{3/2}} +\expected{\rkla{\sum_{n=1}^N\tau\varepsilon\norm{\Delta\h U\nn}\h^2}^{3/2}}\leq C\,.
\end{align}
\end{lemma}
\begin{proof}
We choose $\varphi\h=U\nn$ and $\psi\h=-\Delta\h U\nn$ in \eqref{eq:model:disc:regularised} and obtain by using \eqref{eq:invLap}
\begin{multline}\label{eq:tmp:l2:1}
\rkla{U\nn-U\no,U\nn}\h+\varepsilon\tau\norm{\Delta\h U\nn}\h^2 +\frac{\tau}\varepsilon\rkla{\nabla\Ih{f_\delta\trkla{U\nn}},\nabla U\nn}\\
= \rkla{\Phi\h\trkla{U\no}\sinc{n},U\nn}+\frac{\tau}\varepsilon\rkla{U\no,-\Delta\h U\nn}\h\,.
\end{multline}
Due to assumption \ref{assumption:S2} and the convexity of $F_\delta$, we get that
\begin{align}
\rkla{\nabla\Ih{f_\delta\trkla{U\nn}},\nabla U\nn}\geq0\,.
\end{align}
Hence, estimating the first term in \eqref{eq:tmp:l2:1} in the usual way, applying Young's inequality, and using the equivalence of $\norm{.}_{\mathds{L}^2}$ and $\norm{.}\h$, we obtain
\begin{multline}
\tfrac12\norm{U\nn}\h^2+\tfrac14\norm{U\nn-U\no}\h^2-\tfrac12\norm{U\no}\h^2+\tfrac34\tau\varepsilon\norm{\Delta\h U\nn}\h^2\\
\leq \rkla{\Phi\h\trkla{U\no}\sinc{n},U\no}+C\norm{\Phi\h\trkla{U\no}\sinc{n}}_{\mathds{L}^2}^2+\tau\tfrac1{\varepsilon^3}\norm{U\no}\h^2\,.
\end{multline}
Summing from $n=1$ to $\tilde{N}\leq N$ and taking the $3/2$-th power, we obtain
\begin{align}\label{eq:tmp:L3L2}
\begin{split}
\norm{U^{\tilde{N}}}\h^{3}&\,+\rkla{\sum_{n=1}^{\tilde{N}}\norm{U\nn-U\no}\h^2}^{3/2}+\rkla{\sum_{n=1}^{\tilde{N}}\tau\varepsilon\norm{\Delta\h U\nn}\h^2}^{3/2}\\
\leq &\,C\rkla{\sum_{n=1}^{\tilde{N}} \rkla{\Phi\h\trkla{U\no}\sinc{n},U\no}}^{3/2} +C\rkla{\sum_{n=1}^{\tilde{N}}\norm{\Phi\h\trkla{U\no}\sinc{n}}_{\mathds{L}^2}^2}^{3/2}\\
&\,+C\rkla{\sum_{n=1}^{\tilde{N}}\tau\tfrac{1}{\varepsilon^3}\norm{U\no}\h^2}^{3/2} + C\norm{U^0}\h^{3}\\
=:&\,R_1+R_2+R_3+R_4\,.
\end{split}
\end{align}
Using the Burkholder--Davis--Gundy inequality we deduce
\begin{align}
\begin{split}
\expected{R_2}\leq&\,C\tilde{N}^{1/2}\sum_{n=1}^{\tilde{N}}\expected{\norm{\Phi\h\trkla{U\no}\sinc{n}}_{\mathds{L}^2}^{3}}\\
\leq&\,C\tilde{N}^{1/2}\sum_{n=1}^{\tilde{N}}\expected{\rkla{\tau\sum_{i=1}^{d}\sum_{k\in\Zh}\norm{\partial_i\Ih{\mysigma\trkla{U\no}\lambda_k\g{k}}}_{\mathds{L}^2}^2}^{3/2}}\\
\leq &\,C\sum_{n=1}^{\tilde{N}}\tau\expected{\rkla{1+\norm{\nabla U\no}_{\mathds{L}^2}^2}^{3/2}}\leq C+C\sum_{n=1}^{\tilde{N}}\tau\expected{\norm{\nabla U\no}_{\mathds{L}^2}^{3}}\,.
\end{split}
\end{align}
We continue by applying Hölder's inequality, which yields
\begin{align}
\begin{split}
\norm{\nabla U\no}_{\mathds{L}^2}^4&\leq \norm{U\no-\mathfrak{m}}\h^{2}\norm{\Delta\h U\no}\h^2\\
&\leq \norm{U\no-\mathfrak{m}}_*\norm{\nabla U\no}_{\mathds{L}^2}\norm{\Delta\h U\no}\h^2\,.
\end{split}
\end{align}
Hence, we obtain by applying Young's inequality
\begin{multline}
\expected{\sum_{n=2}^{\tilde{N}}\tau\norm{\nabla U\no}_{\mathds{L}^2}^3}\leq \expected{\rkla{\max_{n=0,\ldots,N}\norm{U\no-\mathfrak{m}}_*}\sum_{n=2}^{\tilde{N}}\tau\norm{\Delta\h U\no}\h^2}\\
\leq C_\alpha\expected{\max_{n=0,...,N}\norm{U\no-\mathfrak{m}}_*^3} +\alpha\expected{\rkla{\sum_{n=2}^{\tilde{N}}\tau\norm{\Delta\h U\no}\h^2 }^{3/2}}\,.
\end{multline}
This implies
\begin{align}
\expected{R_2}\leq C +\alpha\expected{\rkla{\sum_{n=2}^{\tilde{N}}\tau\norm{\Delta\h U\no}\h^2}^{3/2}}+C\tau\expected{\norm{\nabla U^0}_{\mathds{L}^2}^3}\,,
\end{align}
due to Lemma \ref{lem:h-1estimate}.
%{\color{orange} I kept the $\nabla U^0$-term. Otherwise, we need uniform bound on $\Delta\h U^0$.}
Using Young's inequality, the Burkholder--Davis--Gundy inequality, an integration by parts, and assumptions \ref{assumption:W2} and \ref{assumption:C}, we compute for $R_1$
\begin{align}
\begin{split}
\expected{R_1}\leq&\,C\expected{\max_{1\leq\tilde{N}\leq N}\rkla{\sum_{n=1}^{\tilde{N}}\rkla{\Phi\h\trkla{U\no}\sinc{n},U\no}}^2}+C\\
=&\,C\sum_{n=1}^{N}\tau\expected{\sum_{i=1}^d\sum_{k\in\Zh}\abs{\iD\partial_i\Ih{\mysigma\trkla{U\no}\lambda_k\g{k}}U\no\dx}^2}  +C\\
\leq&\,C\sum_{n=1}^{N}\tau\expected{\norm{\nabla U\no}_{\mathds{L}^2}^2} +C\,.
\end{split}
\end{align}
An application of Gronwall's lemma provides 
\begin{align}
\max_{1\leq\tilde{N}\leq N}\expected{\norm{U^{\tilde{N}}}\h^{3}}+\expected{\rkla{\sum_{n=1}^N\norm{U\nn-U\no}\h^2}^{3/2}} +\expected{\rkla{\sum_{n=1}^N\tau\varepsilon\norm{\Delta\h U\nn}\h^2}^{3/2}}\leq C\,.
\end{align}
Thus, we can go back to \eqref{eq:tmp:L3L2}, take the maximum with respect to $\tilde{N}$ before taking the expected value, and repeat the estimates for $R_1$ and $R_2$ to obtain
\begin{align}
\begin{split}
&\expected{\max_{1\leq\tilde{N}\leq N}\norm{U^{\tilde{N}}}\h^3}+\expected{\rkla{\sum_{n=1}^N\norm{U\nn-U\no}\h^2}^{3/2}} +\expected{\rkla{\sum_{n=1}^N\tau\varepsilon\norm{\Delta\h U\nn}\h^2}^{3/2}}\\
&\quad\leq C\sum_{n=1}^N\tau\expected{\norm{\nabla U\no}_{\mathds{L}^2}^2} +\alpha\expected{\rkla{\sum_{n=2}^N\tau\norm{\Delta\h U\no}\h^2}^{3/2}} +C\tau\norm{\nabla U^0}_{\mathds{L}^2}^3\\
&\qquad\qquad+C\expected{\rkla{\sum_{n=1}^N\tau\varepsilon^{-3}\norm{U\no}\h^2}^{3/2}}+C\norm{U^0}\h^3+C\\
&\quad\leq C\,.
\end{split}
\end{align}
\end{proof}

With the next lemma we establish $\delta$-dependent time regularity results that guarantees that linear time-interpolants of the discrete solutions are uniformly Hölder continuous with exponent $1/8$.
\begin{lemma}\label{lem:nikolskii}
Let the assumptions \ref{assumption:T}, \ref{assumption:S1}, \ref{assumption:S2}, \ref{assumption:I}, \ref{assumption:W1}, \ref{assumption:W2}, \ref{assumption:C} hold true and let $\delta\leq 1/8$.
Then, there exists a positive constant $C\trkla{\delta^{-1}}$ such that
\begin{align}
\expected{\sum_{m=0}^{N-l}\tau\norm{U^{m+l}-U^m}_*^{4}}\leq C\trkla{\delta^{-1}}\trkla{l\tau}^{3/2}
\end{align}
for all $l\in\tgkla{1,\ldots,N}$.
%{\color{orange} This estimate provides compactness with respect to time. We also obtain that the linear time interpolation is uniformly Hölder continuous with exponent 1/8}
\end{lemma}
\begin{proof}
Summing \eqref{eq:model:disc:regularised:1} from $n=m+1$ to $m+l$ and choosing $\varphi\h=\trkla{-\Delta\h^{-1}}\trkla{U^{m+l}-U^m}$, we obtain
\begin{align}
\begin{split}\label{eq:nikolskii:1}
\norm{U^{m+l}-U^m}_*^{2} \leq&\,-\rkla{\sum_{n=m+1}^{m+l}\tau W\nn,U^{m+l}-U^m}\h +\tfrac14\norm{U^{m+l}-U^m}_*^{2}\\
&\,+\norm{\sum_{n=m+1}^{m+l} \sum_{i=1}^d\sum_{k\in\Zh} \Ih{\mysigma\trkla{U\no}\lambda_k\g{k}}\sincb{n}{k}{i}}_{\mathds{L}^2}^2 
\end{split}
\end{align}
due to Young's inequality.
Choosing $\psi\h=U^{m+l}-U^m$ in \eqref{eq:model:disc:regularised:2}, we deduce
\begin{align}
\begin{split}
&\,-\rkla{\sum_{n=m+1}^{m+l}\tau W\nn,U^{m+l}-U^m}\h=-\varepsilon\rkla{\sum_{n=m+1}^{m+l}\tau\nabla U\nn,\nabla U^{m+l}-\nabla U^m}\\
&\qquad\quad-\tfrac1\varepsilon\rkla{\sum_{n=m+1}^{m+l}\tau\rkla{f_\delta\trkla{U\nn}-U\no},U^{m+l}-U^m}\h\\
&\quad\leq\,\varepsilon\sum_{n=m+1}^{m+l}\tau\abs{\rkla{\nabla U\nn,\nabla U^{m+l}-\nabla U^m}} +\tfrac14\norm{U^{m+l}-U^m}_*^2\\
&\qquad\quad +C\frac1{\varepsilon^2}\rkla{\sum_{n=m+1}^{m+l}\tau\norm{\nabla\Ih{f_\delta\trkla{U\nn} -U\no}}_{\mathds{L}^2}}^2\,.
\end{split}
\end{align}
Hence, by taking the 2nd power in \eqref{eq:nikolskii:1}, multiplying by $\tau$, summing with respect to $m$, and taking the expectation, we obtain
\begin{align}
\begin{split}
&\expected{\sum_{m=0}^{N-l}\tau\norm{U^{m+l}-U^m}_*^{4}}\leq C\expected{\sum_{m=0}^{N-l}\tau\rkla{\sum_{n=m+1}^{m+l}\tau\abs{\rkla{\nabla U\nn,\nabla U^{m+l}-\nabla U^m}}}^2 }\\
&\qquad\quad+C\expected{\sum_{m=0}^{N-l}\tau\rkla{\sum_{n=m+1}^{m+l}\tau\norm{\nabla\Ih{f_\delta\trkla{U\nn}-U\no}}_{\mathds{L}^2}}^4}\\
&\qquad\quad+C\expected{\sum_{m=0}^{N-l}\tau\norm{\sum_{n=m+1}^{m+l}\sum_{i=1}^d\sum_{k\in\Zh}\Ih{\mysigma\trkla{U\no}\lambda_k\g{k}}\sincb{n}{k}{i}}_{\mathds{L}^2}^4}\\
&\quad=:A_1+A_2+A_3\,.
\end{split}
\end{align}
For $A_3$, we compute by using the Burkholder--Davis--Gundy inequality
\begin{align}
\begin{split}
A_3\leq&\,C\sum_{m=0}^{N-l}\tau\trkla{l\tau}\sum_{n=m+1}^{m+l}\tau\expected{\rkla{\sum_{i=1}^d\sum_{k\in\Zh}\norm{\lambda_k\Ih{\mysigma\trkla{U\no}\g{k}}}_{\mathds{L}^2}^2}^2}\leq C\trkla{l\tau}^2\,.
\end{split}
\end{align}
For $A_2$ we obtain
\begin{align}
\begin{split}
A_2\leq&\,C\expected{\sum_{m=0}^{N-l}\tau\rkla{\sum_{n=m+1}^{m+l}\tau\delta^{-1}\norm{\nabla U\nn}_{\mathds{L}^2}}^4}+C\expected{\sum_{m=0}^{N-l}\tau\rkla{\sum_{n=m+1}^{m+l}\tau\norm{\nabla U\no}_{\mathds{L}^2}}^4}\\
\leq&\,C\expected{\sum_{m=0}^{N-l}\tau\rkla{\sum_{n=m+1}^{m+l}\tau\delta^{-2}\norm{\nabla U\nn}_{\mathds{L}^2}^2}^2\trkla{l\tau}^2}\\
&\,+C\expected{\sum_{m=0}^{N-l}\tau\rkla{\sum_{n=m+1}^{m+l}\tau\norm{\nabla U\no}_{\mathds{L}^2}^2}^2\trkla{l\tau}^2}\\
\leq&\, C\trkla{\delta^{-4}+1}\trkla{l\tau}^2
\end{split}
\end{align}
due to Lemma \ref{lem:h-1estimate}. 
As we do not assume that $\Delta\h U^0$ is bounded independently of $h$ and $\tau$, we need to treat the case $m=0$ separately when deriving the bound for $A_1$.
In particular, integrating by parts and using Young's inequality shows
\begin{align}
\begin{split}
&\expected{\tau\rkla{\sum_{n=1}^{l}\tau\abs{\nabla U\nn,\nabla U^l-\nabla U^0}}^2}\leq \trkla{l\tau}\tau^{1/2}\expected{\tau^{1/2}\norm{U^l-U^0}\h^2\sum_{n=1}^{l}\tau\norm{\Delta\h U\nn}\h^2}\\
&\quad\leq C\trkla{l\tau}^{3/2}\expected{\rkla{\tau^{1/2}\norm{U^l-U^0}\h^2}^3 +\rkla{\sum_{n=1}^l\tau\norm{\Delta\h U\nn}\h^2}^{3/2}}\\
&\quad\leq C\trkla{l\tau}^{3/2}\expected{\rkla{\sum_{n=0}^N\tau\norm{U\nn}\h^4}^{3/2} +\rkla{\sum_{n=1}^N\tau\norm{\Delta\h U\nn}\h^2}^{3/2}} \leq C\trkla{l\tau}^{3/2}\,,
\end{split}
\end{align}
due to Corollary \ref{cor:lql4l2} and Lemma \ref{lem:l2estimate}.
Thus, integrating by parts and using Hölder's inequality, we obtain for $A_1$
\begin{align}
\begin{split}
A_1\leq&\,C\expected{\sum_{m=1}^{N-l}\tau\norm{\Delta\h\rkla{U^{m+l}-U^m}}\h^2 \rkla{\sum_{n=m+1}^{m+l}\tau\norm{U\nn}\h}^2} +C\trkla{l\tau}^{3/2}\\
\leq&\,C\expected{\rkla{\sum_{m=1}^N\tau\norm{\Delta\h U^m}\h^2}^{4/3}}^{3/4}\expected{\rkla{\sum_{n=1}^{N}\tau\norm{ U\nn}\h^4}^2}^{1/4}\trkla{l\tau}^{3/2} +C\trkla{l\tau}^{3/2}\,.
\end{split}
\end{align}
Here, we again made use of the regularity results from Corollary \ref{cor:lql4l2} and Lemma \ref{lem:l2estimate}.

\end{proof}

\section{Limit $\trkla{h,\tau}\searrow 0$}\label{sec:limitht}
Before passing to the limit, we introduce the following time-interpolants of a time-discrete function $\trkla{a\nn}_{n=0,\ldots,N}$:
\begin{subequations}
\begin{align}
a\tl\trkla{.,t}&:=\tfrac{t-t\no}\tau a\nn\trkla{.}+\tfrac{t\nn-t}\tau a\no\trkla{.}&&t\in\tekla{t\no,t\nn}\,,~n\geq1\,,\\
a\tm\trkla{.,t}&:=a\no\trkla{.}&&t\in[t\no,t\nn)\,,~n\geq1\,,\\
a\tp\trkla{.,t}&:=a\nn\trkla{.}&&t\in[t\no,t\nn)\,,~n\geq1\,.
\end{align}
\end{subequations}
If a result is valid for all three interpolations, we will write $a\tpm$.
Using this notation, we collect the results established so far.
As $\norm{U^0}_{\mathds{H}^1}\leq C\norm{u_0}_{\mathds{H}^1}$, we have the $\trkla{h,\tau,\delta}$-independent estimate
\begin{align}\label{eq:boundhtau:independent}
\begin{split}
&\norm{U\tpm}_{L^p\trkla{\Omega;L^\infty\trkla{0,T;\mathds{H}^{-1}}}}+\norm{U\tpm}_{L^p\trkla{\Omega;L^4\trkla{0,T;\mathds{L}^2}}}+\norm{U\tpm}_{L^3\trkla{\Omega;L^\infty\trkla{0,T;\mathds{L}^2}}}\\
&\quad +\norm{U\tpm}_{L^p\trkla{\Omega;L^2\trkla{0,T;\mathds{H}^1}}}  +\norm{\Delta\h U\tp}_{L^3\trkla{\Omega;L^2\trkla{0,T;\mathds{L}^2}}}\\
&\quad +\norm{F_\delta\trkla{U\tp}}_{L^p\trkla{\Omega;L^1\trkla{0,T;\mathds{L}^1}}} +\norm{\Ih{\mysigma\trkla{U\tpm}}}_{L^\infty\trkla{\Omega;L^\infty\trkla{0,T;\mathds{L}^\infty}}}\\
&\quad+\norm{\Ih{\mysigma\trkla{U\tpm}}}_{L^p\trkla{\Omega;L^2\trkla{0,T;\mathds{H}^1}}}\leq C
 \end{split}
\end{align}
for all $p\in[1,\infty)$. %{\color{orange} I removed the bound on $F_\delta\trkla{U\tm}$ as we do not guarantee $\abs{U^0}\leq1$.}
In addition, we have the following estimate, which depends on $\delta$ but is independent of $\trkla{h,\tau}$:
\begin{multline}\label{eq:boundhtau:deltadependent}
\norm{U\tl}_{L^4\trkla{\Omega;C^{0,1/8}\trkla{\tekla{0,T};\mathds{H}^{-1}}}} +\norm{\Ih{f_\delta\trkla{U\tpm}}}_{L^p\trkla{\Omega;L^4\trkla{0,T;\mathds{L}^2}}}\\
+\norm{W\tp}_{L^3\trkla{\Omega;L^2\trkla{0,T;\mathds{L}^2}}}\leq C\trkla{\delta^{-1}}\,.
\end{multline}
The bound for the first term on the left-hand side in \eqref{eq:boundhtau:deltadependent} is a consequence of Lemma \ref{lem:nikolskii}, \cite[Lemma 3.2]{Banas2013}, and the embedding results in \cite{Simon1990},
while the bound for the second term is a straightforward consequence of the definition of $F_\delta$ and the bounds on $U\tpm$.
{To obtain the estimate for $W\tp$ in \eqref{eq:boundhtau:deltadependent}, we write \eqref{eq:model:disc:regularised:2} as
\begin{align}
W\nn=-\varepsilon\Delta\h U\nn+\varepsilon^{-1}\Ih{f_\delta\trkla{U\nn}-U\no}\,.
\end{align}
Thus, combining the $L^3\trkla{\Omega;L^2\trkla{0,T;\mathds{L}^2}}$-bound on $\Delta\h U\tp$ and the $\delta$-dependent bound on $\Ih{f_\delta\trkla{U\tp}}$, we obtain the desired result.
}

To pass to the limit on the right-hand side, we introduce a piecewise linear interpolation of the stochastic integral
\begin{align}
&&S\tl\trkla{t}&:=\trkla{S\nn-S\no}\frac{t-t\no}\tau +S\no&&\text{for~}t\in\tekla{t\no,t\nn}\,,\\
\text{with}&&S^m&:=\sum_{n=1}^m\sum_{i=1}^d\sum_{k\in\mathds{Z}\h}\partial_i\Ih{\mysigma\trkla{U\no}\lambda_k\g{k}}\sincb{n}{k}{i}&&\text{for~}m\in\tgkla{1,\ldots,N}\,.
\end{align}
This family of approximations is uniformly bounded in $L^2\trkla{\Omega;L^2\trkla{0,T;\mathds{L}^2}}$:
\begin{lemma}\label{lem:linearnoise}
Let the assumptions \ref{assumption:T}, \ref{assumption:S1}, \ref{assumption:I}, \ref{assumption:W1}, \ref{assumption:W2}, and \ref{assumption:C} hold true and let $0<\delta\leq1/8$.
Then, there exists a constant $C>0$ which is independent of $h$ and $\tau$ such that
\begin{align}
\expected{\int_0^T\norm{S\tl\trkla{t}}_{\mathds{L}^2}^2\dt}\leq C\sum_{n=1}^N\tau\expected{1+\norm{\nabla U\no}_{\mathds{L}^2}^2}\leq C\,.
\end{align}
\end{lemma}
\begin{proof}
Using \Ito's isometry, we derive
\begin{align}
\begin{split}
\expected{\int_0^T\norm{S\tl\trkla{t}}_{\mathds{L}^2}^2\dt}=&\,\sum_{n=1}^N\int_{t\no}^{t\nn}\expected{\norm{\trkla{S\nn-S\no}\frac{t-t\no}\tau+S\no}_{\mathds{L}^2}^2}\dt\\
\leq&\,C\sum_{n=1}^N\int_{t\no}^{t\nn}\expected{\norm{S\nn-S\no}_{\mathds{L}^2}^2}+\expected{\norm{S\no}_{\mathds{L}^2}^2}\dt\\
=&\,C\sum_{n=1}^N\tau\expected{\norm{\sum_{i=1}^d\sum_{k\in\mathds{Z}\h}\partial_i\Ih{\mysigma\trkla{U\no}\lambda_k\g{k}}\sincb{n}{k}{i}}_{\mathds{L}^2}^2} \\
&\,+C\sum_{n=1}^N\tau\expected{\norm{\sum_{m=1}^{n-1}\sum_{i=1}^d\sum_{k\in\mathds{Z}\h}\partial_i\Ih{\mysigma\trkla{U^{m-1}}\lambda_k\g{k}}\sincb{m}{k}{i}}_{\mathds{L}^2}^2}\\
\leq&\,C\sum_{n=1}^N\tau^2\sum_{m=1}^n\expected{\sum_{i=1}^d\sum_{k\in\mathds{Z}\h}\norm{\partial_i\Ih{\mysigma\trkla{U^{m-1}}\lambda_k\g{k}}}_{\mathds{L}^2}^2}\,.
\end{split}
\end{align}
Recalling Lemma \ref{lem:interpolation}, we deduce
\begin{align}
\begin{split}
\norm{\partial_i\Ih{\mysigma\trkla{U\no}\lambda_k\g{k}}}_{\mathds{L}^2}^2\leq &\,\norm{\partial_i\rkla{\Ih{\mysigma\trkla{U\no}}\Ih{\lambda_k\g{k}}}}_{\mathds{L}^2}^2\\
&\,+\norm{\partial_i\trkla{1-\Ihop}\gkla{\Ih{\mysigma\trkla{U\no}}\Ih{\lambda_k\g{k}}}}_{\mathds{L}^2}^2\\
\leq&\,\norm{\Ih{\mysigma\trkla{U\no}}\partial_i\Ih{\lambda_k\g{k}}}_{\mathds{L}^2}^2+\norm{\partial_i\Ih{\mysigma\trkla{U\no}}\Ih{\lambda_k\g{k}}}_{\mathds{L}^2}^2\\
&\,+Ch\norm{\nabla\Ih{\mysigma\trkla{U\no}}}_{\mathds{L}^2}^2\norm{\lambda_k\g{k}}_{\mathds{L}^\infty}^2\,.
\end{split}
\end{align}
Thus, using the Lipschitz continuity of $\mysigma$, we conclude 
\begin{align}
\expected{\int_0^T\norm{S\tl\trkla{t}}_{\mathds{L}^2}^2\dt}\leq C\sum_{n=1}^N\tau\expected{1+\norm{\nabla U\no}_{\mathds{L}^2}^2}\leq C\,.
\end{align}
\end{proof}

With these uniform bounds, we can pass to the limit $\trkla{h,\tau}\rightarrow0$ using weak and weak* convergence for subsequences:
\begin{theorem}\label{thm:weakht}
Let the assumptions \ref{assumption:T}, \ref{assumption:S1}, \ref{assumption:S2}, \ref{assumption:I}, \ref{assumption:W1}, \ref{assumption:W2}, and \ref{assumption:C} hold true and let $\delta$ be sufficiently small. Then, there exist
\begin{align*}
u_\delta\in &\,L^p_{\operatorname{weak-(*)}}\trkla{\Omega;L^\infty\trkla{0,T;\mathds{H}^{-1}_{\per}}}\cap L^3_{\operatorname{weak-(*)}}\trkla{\Omega;L^\infty\trkla{0,T;\mathds{L}^{2}}}\cap L^p\trkla{\Omega;L^4\trkla{0,T;\mathds{L}^2}}\\
&\,\cap L^p\trkla{\Omega;L^2\trkla{0,T;\mathds{H}^1_{\per}}}\cap L^3\trkla{\Omega;L^2\trkla{0,T;\mathds{H}^2_{\per}}}\cap L^4\trkla{\Omega;C^{0,1/8}\trkla{\tekla{0,T};\mathds{H}^{-1}_{\per}}}\,,\\
w_\delta\in&\,L^3\trkla{\Omega;L^2\trkla{0,T;\mathds{L}^2}}\,,\\
p_\delta\in&\,L^p\trkla{\Omega;L^4\trkla{0,T;\mathds{L}^2}}\,,\\
\mysigma_\delta\in&\, L^p\trkla{\Omega;L^2\trkla{0,T;\mathds{H}^1_{\per}}}\,
\end{align*}
satisfying
\begin{subequations}\label{eq:limitht}
\begin{align}\label{eq:limitht:u}
\mathrm{d} u_\delta-\Delta w_\delta\dt= \div\gkla{\mysigma_\delta\dW}&&\text{in~} \mathds{H}^{-2}_{\per}
\end{align}
$\Prob$-almost surely and for all $t\in\trkla{0,T}$, where $w_\delta$ is given as
\begin{align}\label{eq:limitht:w}
w_\delta = -\varepsilon \Delta u_\delta +\tfrac1\varepsilon p_\delta -\tfrac1\varepsilon u_\delta&&\text{in~}\mathds{L}^2\,
\end{align}
$\mathds{P}$-almost surely and for almost all $t\in\trkla{0,T}$.
\end{subequations}
Furthermore, for $\trkla{h,\tau}\rightarrow0$ we have for (not relabeled) subsequences
\begin{subequations}\label{eq:weakhtau}
\begin{align}
U\tpm\stackrel{*}{\rightharpoonup}&\,u_\delta&&\text{in~}L^p_{\operatorname{weak-(*)}}\trkla{\Omega;L^\infty\trkla{0,T;\mathds{H}^{-1}_{\per}}}\cap L^3_{\operatorname{weak-(*)}}\trkla{\Omega;L^\infty\trkla{0,T;\mathds{L}^{2}}}\,,\\
U\tpm\rightharpoonup &\,u_\delta&&\text{in~}L^p\trkla{\Omega;L^4\trkla{0,T;\mathds{L}^2}}\cap L^p\trkla{\Omega;L^2\trkla{0,T;\mathds{H}^1_{\per}}}\,,\\
\Delta\h U\tp\rightharpoonup &\,\Delta u_\delta&&\text{in~}L^3\trkla{\Omega;L^2\trkla{0,T;\mathds{L}^2}}\,,\\
W\tp \rightharpoonup&\,w_\delta&&\text{in~}L^3\trkla{\Omega;L^2\trkla{0,T;\mathds{L}^2}}\,,\\
\Ih{f_\delta\trkla{U\tpm}}\rightharpoonup&\,p_\delta&&\text{in~}L^p\trkla{\Omega;L^4\trkla{0,T;\mathds{L}^2}}\,,\\
\Ih{\mysigma\trkla{U\tpm}}\rightharpoonup&\,\mysigma_\delta&&\text{in~}L^p\trkla{\Omega;L^2\trkla{0,T;\mathds{H}^1_{\per}}}\,,
\end{align}
\end{subequations}
and
\begin{align}\label{eq:weakhtau:noise}
U\tl-S\tl\rightharpoonup&\,u_\delta-\sum_{i=1}^d\sum_{k\in\mathds{Z}}\int_0^\cdot \partial_i\rkla{\mysigma_\delta\lambda_k\g{k}}\dbeta{k}^i &&\text{in~}L^p\trkla{\Omega;L^2\trkla{0,T;\mathds{L}^2}}\,.
\end{align}

\end{theorem}
An alternative formulation for \eqref{eq:limitht:u} reads
\begin{align}
\rkla{\mathrm{d}u_\delta,\psi}_{-1}+\rkla{w_\delta,\psi}\dt=\rkla{\div\tgkla{\mysigma_\delta\dW},\psi}_{-1}
\end{align}
for all $\psi\in \mathds{L}^2$.

\begin{proof}
The weak and weak* convergence expressed in \eqref{eq:weakhtau} is a direct consequence of the uniform bounds \eqref{eq:boundhtau:independent} and \eqref{eq:boundhtau:deltadependent}.
We shall now establish \eqref{eq:weakhtau:noise}. 
First, the bound from Lemma \ref{lem:linearnoise} guarantees the existence of subsequence converging weakly towards some limit $\zeta_\delta$. 
To identify this limit, we proceed as follows.
First, we compare $S\tl$ with a suitable \Ito-integral:
\begin{align}
\begin{split}
&\expected{\norm{S\tl\trkla{t}-\sum_{i=1}^d\sum_{k\in\mathds{Z}}\int_0^t\partial_i\rkla{\Ih{\mysigma\trkla{U\tm}}\lambda_k\g{k}}\dbeta{k}^i}_{-1}^2}\\
&\quad\leq  C\expected{\norm{S\tl\trkla{t}-\sum_{i=1}^d\sum_{k\in\mathds{Z}\h}\int_0^t\partial_i\Ih{\mysigma\trkla{U\tm}\lambda_k\g{k}}\dbeta{k}^i}_{-1}^2}\\
&\qquad\quad+C\expected{\norm{\sum_{i=1}^d\sum_{k\in\mathds{Z}\h}\int_0^t\partial_i\trkla{I-\Ihop}\gkla{\Ih{\mysigma\trkla{U\tm}}\Ih{\lambda_k\g{k}}}\dbeta{k}^i}_{-1}^2}\\
&\qquad\quad+C\expected{\norm{\sum_{i=1}^d\sum_{k\in\mathds{Z}\h}\int_0^t\partial_i\rkla{\Ih{\mysigma\trkla{U\tm}}\trkla{I-\Ihop}\gkla{\lambda_k\g{k}}}\dbeta{k}^i}_{-1}^2}\\
&\qquad\quad+C\expected{\norm{\sum_{i=1}^d\sum_{k\in\mathds{Z}\setminus\mathds{Z}\h}\int_0^t\partial_i\rkla{\Ih{\mysigma\trkla{u\tm}}\lambda_k\g{k}}\dbeta{k}^i}_{-1}^2}\\
&\quad=: I+II+III+IV\,.
\end{split}
\end{align}
For $t\in[t\no,t\nn)$, we deduce using \Ito's isometry, \ref{assumption:W2}, and \ref{assumption:C}
\begin{align}
\begin{split}
I=&\,C\mathds{E}\left[\left\|\sum_{i=1}^d\sum_{k\in\mathds{Z}\h} \frac{t-t\no}{\tau}\int_{t\no}^{t\nn}\partial_i\Ih{\mysigma\trkla{U\tm}\lambda_k\g{k}}\dbeta{k}^i \right.\right.\\
&\qquad\qquad\left.\left.- \sum_{i=1}^d\sum_{k\in\mathds{Z}\h}\int_{t\no}^t\partial_i\Ih{\mysigma\trkla{U\tm}\lambda_k\g{k}}\dbeta{k}^i\right\|_{-1}^2\right]\\
=&\,C\mathds{E}\left[\left(\sup_{0\neq\psi\in\mathds{H}^1_{\per}}\left(\left[\sum_{i=1}^d\sum_{k\in\mathds{Z}\h}\frac{t-t\no}\tau\int_{t\no}^{t\nn}\Ih{\mysigma\trkla{U\tm}\lambda_k\g{k}}\dbeta{k}^i\right.\right.\right.\right.\\
&\qquad\qquad \left.\left.\left.\left.-\int_{t\no}^t\Ih{\mysigma\trkla{U\tm}\lambda_k\g{k}}\dbeta{k}^i\right],\partial_i\psi\right)\norm{\psi}_{\mathds{H}^1}^{-1}\right)^2\right]\\
\leq&\,C\mathds{E}\left[\left\|\sum_{i=1}^d\sum_{k\in\mathds{Z}\h}\left(\frac{t-t\no}\tau\int_{t\no}^{t\nn}\Ih{\mysigma\trkla{U\tm}\lambda_k\g{k}}\dbeta{k}^i\right.\right.\right. \\
&\qquad\qquad \left.\left.\left.-\int_{t\no}^t\Ih{\mysigma\trkla{U\tm}\lambda_k\g{k}}\dbeta{k}^i\right)\right\|_{\mathds{L}^2}^2\right]\\
\leq&\,C\expected{\norm{\sum_{i=1}^d\sum_{k\in\mathds{Z}\h}\int_{t\no}^{t\nn}\Ih{\mysigma\trkla{U\tm}\lambda_k\g{k}}\dbeta{k}^i}_{\mathds{L}^2}^2} \\
&\,+C\expected{\norm{\sum_{i=1}^d\sum_{k\in\mathds{Z}\h}\int_{t\no}^{t}\Ih{\mysigma\trkla{U\tm}\lambda_k\g{k}}\dbeta{k}^i}_{\mathds{L}^2}^2}\\
\leq&\,C\expected{\sum_{i=1}^d\sum_{k\in\mathds{Z}\h} \tau\norm{\Ih{\mysigma\trkla{U\tm}\lambda_k\g{k}}}_{\mathds{L}^2}^2}\leq C\tau\,.
\end{split}
\end{align}
For the second term, we obtain by using \Ito's isometry, the definition of the $\mathds{H}^{-1}_{\per}$-norm, Lemma \ref{lem:interpolation}, and inverse estimate (cf.~\cite[Theorem 4.5.11]{BS}), \ref{assumption:W2}, and \ref{assumption:C}
\begin{align}
\begin{split}
II\leq&\,C\expected{\sum_{i=1}^d\sum_{k\in\mathds{Z}\h}\int_0^t\norm{\trkla{I-\Ihop}\gkla{\Ih{\mysigma\trkla{U\tm}}\Ih{\lambda_k\g{k}}}}_{\mathds{L}^2}^2\ds}\\
\leq&\,C h^2\expected{\sum_{i=1}^d\sum_{k\in\mathds{Z}\h}\int_0^t\norm{\Ih{\mysigma\trkla{U\tm}}}_{\mathds{L}^2}^2\norm{\lambda_k\g{k}}_{\mathds{W}^{1,\infty}}^2\ds}\leq Ch^2\,.
\end{split}
\end{align}
Similarly, we compute
\begin{align}
\begin{split}
III\leq&\,C\expected{\sum_{i=1}^d\sum_{k\in\mathds{Z}\h}\int_0^t\norm{\Ih{\mysigma\trkla{U\tm}}\trkla{I-\Ihop}\gkla{\lambda_k\g{k}}}_{\mathds{L}^2}^2\ds}\\
\leq&\,C\expected{\sum_{i=1}^d\sum_{k\in\mathds{Z}\h}\int_0^t\norm{\Ih{\mysigma\trkla{U\tm}}}_{\mathds{L}^2}^2\norm{\trkla{I-\Ihop}\gkla{\lambda_k\g{k}}}_{\mathds{L}^\infty}^2\ds}\leq Ch^2\,.
\end{split}
\end{align}
For $IV$, we again use \Ito's isometry and the definition of the $\mathds{H}^{-1}_{\per}$-norm to deduce
\begin{align}
\begin{split}
IV\leq&\,C\expected{\sum_{i=1}^d\sum_{k\in\mathds{Z}\setminus\mathds{Z}\h}\int_0^t\norm{\Ih{\mysigma\trkla{U\tm}}}_{\mathds{L}^2}^2\norm{\lambda_k\g{k}}_{\mathds{L}^\infty}^2\ds}\\
\leq&\,C\sum_{k\in\mathds{Z}\setminus\mathds{Z}\h}\abs{\lambda_k}^2\norm{\g{k}}_{\mathds{L}^\infty}^2\,.
\end{split}
\end{align}
Recalling \ref{assumption:W2} and \ref{assumption:Z}, we obtain $IV\rightarrow0$ for $h\searrow0$. Thus, we have
\begin{align}
\expected{\norm{S\tl\trkla{t}-\sum_{i=1}^d\sum_{k\in\mathds{Z}}\int_0^t\partial_i\rkla{\Ih{\mysigma\trkla{U\tm}}\lambda_k\g{k}}\dbeta{}^i}_{-1}^2}\rightarrow 0
\end{align}
for $\trkla{h,\tau}\rightarrow0$.
Further, as $\Ih{\mysigma\trkla{U\tm}}\rightharpoonup\mysigma_\delta$ in $L^p\trkla{\Omega;L^2\trkla{0,T;\mathds{H}^1_{\per}}}$, we obtain
\begin{align}
\sum_{i=1}^d\sum_{k\in\mathds{Z}}\int_0^{\cdot}\partial_i\rkla{\Ih{\mysigma\trkla{U\tm}}\lambda_k\g{k}}\dbeta{k}^i\rightharpoonup \sum_{i=1}^d\sum_{k\in\mathds{Z}}\int_0^{\cdot}\partial_i\rkla{\mysigma_\delta\lambda_k\g{k}}\dbeta{k}^i
\end{align}
in $C\trkla{0,T;L^2\trkla{\Omega;\mathds{L}^2}}$. %{\color{red} needs to be double checked. I just took the argument from Bauzet et al}
Thus, we can conclude
\begin{align}
S\tl\rightharpoonup \sum_{i=1}^d\sum_{k\in\mathds{Z}}\int_0^{\cdot}\partial_i\rkla{\mysigma_\delta\lambda_k\g{k}}\dbeta{k}^i
\end{align}
in $L^2\trkla{\Omega;C\trkla{0,T;\mathds{H}^{-1}_{\per}}}$, which allows to identify the limit $\zeta_\delta$ and therefore establish \eqref{eq:weakhtau:noise}.\\
To establish \eqref{eq:limitht}, we interpolate \eqref{eq:model:disc:regularised:1} piecewise linearly in time and obtain
\begin{multline}
\rkla{U\tl\trkla{t}-S\tl\trkla{t},\varphi\h}-\iD\trkla{\ids-\Ihop}\gkla{U\tl\trkla{t}\varphi\h}\dx +\trkla{\int_0^t\nabla W\tp\ds,\nabla\varphi\h}\\
= \trkla{U^0,\varphi\h}-\iD\trkla{\ids-\Ihop}\gkla{U^0\varphi\h}\dx\,.
\end{multline}
Choosing $\varphi\h=\psi\mathcal{R}\eta\in L^2\trkla{\Omega\times\trkla{0,T};\mathds{V}\h}$ with $\psi\in L^2\trkla{\Omega\times\trkla{0,T}}$ and $\eta\in\mathds{H}^2_{\per}$ and integrating with respect to time and probability, we obtain
\begin{align}
\begin{split}
&\expected{\int_0^T\trkla{U\tl\trkla{t}-S\tl\trkla{t},\psi\mathcal{R}\eta}\dt} +\expected{\int_0^T\rkla{\int_0^tW\tp\trkla{s}\ds,-\Delta\eta}\psi\dt}\\
&~=\expected{\int_0^T\rkla{u_0,\mathcal{R}\eta}\psi\dt} +\expected{\int_0^T\iD\trkla{\ids-\Ihop}\gkla{U\tl\trkla{t}\mathcal{R}\eta}\dx\,\psi\dt}\\
&\qquad-\expected{\int_0^T\iD\trkla{\ids-\Ihop}\gkla{U^0\mathcal{R}\eta}\dx\,\psi\dt}\,,
\end{split}
\end{align}
due to Assumption \ref{assumption:I}.
Recalling the strong convergence of the Ritz projection $\mathcal{R}\eta\rightarrow\eta$ in $\mathds{H}^1_{\per}$ (cf.~\eqref{eq:convergenceRitz}) and the error estimate
\begin{align}
\iD\trkla{\ids-\Ihop}\gkla{U\tl\trkla{t}\mathcal{R}\eta}\dx\leq Ch^2\norm{\nabla U\tl\trkla{t}}_{\mathds{L}^2}\norm{\nabla \mathcal{R}\eta}_{\mathds{L}^2}\,,
\end{align}
we can pass to the limit and obtain
\begin{multline}
\expected{\int_0^T\rkla{u_\delta\trkla{t}-\sum_{i=1}^d\sum_{k\in\mathds{Z}}\int_0^t\partial_i\rkla{\mysigma_\delta\lambda_k\g{k}}\dbeta{k}^i,\eta}\psi\dt} \\
+\expected{\int_0^T\rkla{\int_0^t w_\delta\trkla{s}\ds,-\Delta\eta}\psi\dt}=\expected{\int_0^T\rkla{u_0,\eta}\psi\dt}\,
\end{multline}
for all $\eta\in\mathds{H}^2_{\per}$ and $\psi\in L^2\trkla{\Omega\times\trkla{0,T}}$.
This provides \eqref{eq:limitht:u}.
\end{proof}
The convergence results presented in Theorem \ref{thm:weakht} relied solely on weak and weak* convergence and thus prevented an identification of all terms occurring in the limit equation \eqref{eq:limitht}.
To identify these terms, we apply a monotonicity argument to deduce strong convergence of $U\tpm$ in $L^2\trkla{\Omega;L^2\trkla{0,T;\mathds{L}^2}}$:

\begin{lemma}\label{lem:strongConvergenceht}
Let the assumptions of Theorem \ref{thm:weakht} hold true.
For subsequences, we have the following (strong) convergence properties:
\begin{subequations}
\begin{align}
U\tpm&\rightarrow u_\delta&&\text{in~}L^2\trkla{\Omega;L^2\trkla{0,T;\mathds{L}^2}}\,,\label{eq:strongconvud}\\
\mysigma\trkla{U\tpm}&\rightarrow \mysigma\trkla{u_\delta}&&\text{in~}L^p\trkla{\Omega;L^p\trkla{0,T;\mathds{L}^p}}\text{~for~}p\in[1,\infty)\,,\\
\Ih{f_\delta\trkla{U\tpm}}&\rightarrow f_\delta\trkla{u_\delta}&&\text{in~}L^2\trkla{\Omega;L^2\trkla{0,T;\mathds{L}^2}}\,.\label{eq:strongconvfd}
\end{align}
\end{subequations}
\end{lemma}

\begin{proof}
To establish the strong convergence \eqref{eq:strongconvud}, we shall first prove 
\begin{align}\label{eq:convnormh-1}
\norm{U\tp-\mathfrak{m}}_{L^2\trkla{\Omega;L^2\trkla{0,T;\mathds{H}^{-1}_\per}}}\rightarrow \norm{u_\delta-\mathfrak{m}}_{L^2\trkla{\Omega;L^2\trkla{0,T;\mathds{H}^{-1}_\per}}}\,.
\end{align}
In combination with the already established weak convergence this provides strong convergence in $L^2\trkla{\Omega;L^2\trkla{0,T;\mathds{H}^{-1}_\per}}$, which we shall then improve to \eqref{eq:strongconvud}.
As the weak convergence already provides
\begin{align}
\norm{u_\delta-\mathfrak{m}}_{L^2\trkla{\Omega;L^2\trkla{0,T;\mathds{H}^{-1}_\per}}}\leq \liminf_{\trkla{h,\tau}\rightarrow0} \norm{U\tp-\mathfrak{m}}_{L^2\trkla{\Omega;L^2\trkla{0,T;\mathds{H}^{-1}_\per}}}\,,
\end{align}
it remains to verify the upper bound for the $\limsup$.

Recalling $\trkla{U\nn}_\D=\mathfrak{m}$ for all $n$, we use $\varphi\h=\trkla{-\Delta\h^{-1}}\trkla{U\nn-\mathfrak{m}}$ and $\psi\h=U\nn-\mathfrak{m}$ as test functions in \eqref{eq:model:disc:regularised} and obtain
\begin{align}
\begin{split}
&\tfrac12 \norm{U\nn-\mathfrak{m}}_*^2+\tfrac12\norm{U\nn-U\no}_*^2-\tfrac12\norm{U\no-\mathfrak{m}}_*^2+\tau\varepsilon \norm{\nabla U\nn}_{L^2\trkla{\D}}^2\\
&\qquad\quad+\tau\frac1\varepsilon\rkla{f_\delta\trkla{U\nn},U\nn-\mathfrak{m}}\h -\tau\frac1\varepsilon\rkla{U\no,U\nn-\mathfrak{m}}\h\\
&\quad\leq \rkla{\Phi\h\trkla{U\no}\sinc{n},-\Delta\h^{-1}\trkla{U\no-\mathfrak{m}}} +\tfrac12\norm{U\nn-U\no}_*^2\\
&\qquad\quad+\tfrac12\norm{\sum_{i=1}^d\sum_{k\in\Zh}\nabla\trkla{-\Delta^{-1}} \partial_i\Ih{\mysigma\trkla{U\no}\g{k}\lambda_k}\sincb{n}{k}{i}}_{\mathds{L}^2}^2\,.
\end{split}
\end{align}
Multiplying by $2e^{-\hat{c}t\no}$ with sufficiently large $\hat{c}$ which will be specified below, taking the expectation, and summing from $n=1$ to $m$, we obtain
\begin{align}
\begin{split}\label{eq:strongconv:discIto}
\mathds{E}&\ekla{e^{-\hat{c}t^{m-1}} \norm{U^m-\mathfrak{m}}_*^2}\\
&+\mathds{E}\left[2\sum_{n=1}^m\tau e^{-\hat{c}t\no}\rkla{\varepsilon\norm{\nabla U\nn}_{\mathds{L}^2}^2+\frac1\varepsilon\rkla{f_\delta\trkla{U\nn}, U\nn-\mathfrak{m}}\h  -\frac1\varepsilon \rkla{U\no,U\nn-\mathfrak{m} }\h}\right]\\
\leq&\,\mathds{E}\left[\sum_{n=2}^m \rkla{e^{-\hat{c}t\no}-e^{-\hat{c}t^{n-2}}}\norm{U\no-\mathfrak{m}}_*^2\right]+\expected{\norm{U^0-\mathfrak{m}}_*^2}\\
&+\expected{\sum_{n=1}^m\tau e^{-\hat{c}t\no} \sum_{i=1}^d\sum_{k\in\mathds{Z}\h}\norm{\nabla\trkla{-\Delta^{-1}} \partial_i\Ih{\mysigma\trkla{U\no}\g{k}\lambda_k}}_{\mathds{L}^2}^2}\,.
\end{split}
\end{align}
We continue by estimating the terms in \eqref{eq:strongconv:discIto}:
To deal with the first term on the left-hand side, we need to replace the $\norm{\cdot}_*$-norm with the $\norm{\cdot}_{-1}$-norm.
Hence, we compute
\begin{align}
\begin{split}
&\norm{U^m-\mathfrak{m}}_{-1} =\norm{\nabla\trkla{-\Delta^{-1}}\rkla{U^m-\mathfrak{m}}}_{\mathds{L}^2}\\
&\quad\leq \norm{\nabla\trkla{-\Delta\h^{-1}}\rkla{ U^m-\mathfrak{m}}}_{\mathds{L}^2} +\norm{\nabla\trkla{-\Delta\h^{-1}}\rkla{ U^m-\mathfrak{m}}-\nabla\trkla{-\Delta^{-1}}\rkla{ U^m-\mathfrak{m}}}_{\mathds{L}^2}\\
&\quad\leq \norm{U^m-\mathfrak{m}}_{*} +Ch^{1/2}\norm{U^m-\mathfrak{m}}_{\mathds{L}^2}\leq \norm{U^m-\mathfrak{m}}_{*} +Ch^{1/2}\norm{U^m}_{\mathds{L}^2}\,.
\end{split}
\end{align}
Taking squares on both sides and dividing by $\trkla{1+h^{1/2}}$ provides
\begin{align}\label{eq:tmp:*-1ineq}
\trkla{1+h^{1/2}}^{-1}\norm{U^m-\mathfrak{m}}_{-1}^2\leq \norm{U^m-\mathfrak{m}}_*^2+Ch^{1/2}\norm{U^m}_{\mathds{L}^2}^2\,.
\end{align}
Using the monotonicity of the exponential function and Lemma \ref{lem:l2estimate}, we obtain for $t\in[t^{m-1},t^m)$
\begin{align}
\begin{split}
\frac1{1+h^{1/2}}&\expected{e^{-\hat{c}t}\norm{U\tp\trkla{t}-\mathfrak{m}}_{-1}^2}\leq \frac1{1+h^{1/2}}\expected{e^{-\hat{c}t^{m-1}}\norm{U^m-\mathfrak{m}}_{-1}^2} \\
&\leq \expected{e^{-\hat{c}t^{m-1}}\norm{U^m-\mathfrak{m}}_*^2} +Ch^{1/2}\expected{e^{-\hat{c}t^{m-1}}\norm{U^m}_{\mathds{L}^2}^2}\\
&\leq \expected{e^{-\hat{c}t^{m-1}}\norm{U^m-\mathfrak{m}}_*^2} +Ch^{1/2}\,.
\end{split}
\end{align}
Similar arguments provide
\begin{align}
\norm{U^0-\mathfrak{m}}_*^2\leq \trkla{1+h^{1/2}}\norm{U^0-\mathfrak{m}}_{-1}^2+Ch^{1/2}\norm{U^0}_{\mathds{L}^2}^2\,.
\end{align}
We continue by discussing the second term on the left-hand side of \eqref{eq:strongconv:discIto}.
Using again the monotonicity of the exponential function, we obtain for $t\in[t^{m-1},t^m)$
\begin{multline}
\expected{2\varepsilon\sum_{n=1}^m\tau e^{-\hat{c}t\no} \norm{\nabla U\nn}_{\mathds{L}^2}^2}\geq \expected{2\varepsilon \int_0^t  e^{-\hat{c}s}\norm{\nabla U\tp}_{\mathds{L}^2}^2\ds}\\
= 2\varepsilon\expected{\int_0^te^{-\hat{c}s}\rkla{\norm{\nabla U\tp-\nabla u_\delta}_{\mathds{L}^2}^2 +2\rkla{\nabla U\tp,\nabla u_\delta}-\norm{\nabla u_\delta}_{\mathds{L}^2}^2}\ds}\,.
\end{multline}
To deal with the next term, we first note that for fixed $\delta>0$ the mapping $f_\delta$ is Lipschitz continuous.
Thus, the estimate
\begin{align}\label{eq:strongconv:Lipschitz}
\abs{\Ih{f_\delta\trkla{U\tpm}}-f_\delta\trkla{U\tpm}}^2\leq C_\delta h^2\abs{\nabla U\tpm}^2
\end{align}
holds true pointwise.
We continue with the decomposition
\begin{align}
\begin{split}
\trkla{f_\delta\trkla{U\tp},U\tp}\h=&\, \trkla{\Ih{f_\delta\trkla{U\tp}},U\tp}- \iD\trkla{1-\Ihop}\tgkla{\Ih{f_\delta\trkla{U\tp}}U\tp}\dx\\
=&\,\trkla{f_\delta\trkla{U\tp},U\tp} - \trkla{\trkla{1-\Ihop}\tgkla{f_\delta\trkla{U\tp}},U\tp} \\
&\,- \iD\trkla{1-\Ihop}\tgkla{\Ih{f_\delta\trkla{U\tp}}U\tp}\dx\\
=&\,\trkla{f_\delta\trkla{U\tp}-f_\delta\trkla{u_\delta},U\tp-u_\delta} + \trkla{f_\delta\trkla{u_\delta},U\tp-u_\delta}\\
&\,+\trkla{\Ih{f_\delta\trkla{U\tp}},u_\delta} +\trkla{\trkla{1-\Ihop}\tgkla{f_\delta\trkla{U\tp}},u_\delta}\\
&\,- \trkla{\trkla{1-\Ihop}\tgkla{f_\delta\trkla{U\tp}},U\tp}\\
&\,- \iD\trkla{1-\Ihop}\tgkla{\Ih{f_\delta\trkla{U\tp}}U\tp}\dx\,.
\end{split}
\end{align}
Here, the first term on the right-hand side is nonnegative.
Using Hölder's inequality, \eqref{eq:strongconv:Lipschitz}, and Lemma \ref{lem:interpolation}, we estimate the fourth, fifth, and sixth term on the right-hand side via
\begin{align}
\begin{split}
\abs{\trkla{\trkla{1-\Ihop}\tgkla{f_\delta\trkla{U\tp}},u_\delta}} \leq&\,\norm{\trkla{1-\Ihop}\tgkla{f_\delta\trkla{U\tp}}}_{\mathds{L}^2}\norm{u_\delta}_{\mathds{L}^2}\\
\leq&\, C_\delta h \norm{\nabla U\tp}_{\mathds{L}^2}\norm{u_\delta}_{\mathds{L}^2} \,,
\end{split}\\
\abs{\trkla{\Ih{f_\delta\trkla{U\tp}}-f_\delta\trkla{U\tp},U\tp}}\leq&\, C_\delta h \norm{\nabla U\tp}_{\mathds{L}^2}\norm{U\tp}_{\mathds{L}^2}\,,\\
\abs{\iD\trkla{1-\Ihop}\tgkla{\Ih{f_\delta\trkla{U\tp}}U\tp}\dx} \leq&\,C h \norm{\Ih{f_\delta\trkla{U\tp}}}_{\mathds{L}^2}\norm{\nabla U\tp}_{\mathds{L}^2}\,.
\end{align}
Hence, due to the already established uniform bounds, we deduce for $t\in[t^{m-1},t^m)$
\begin{multline}
\expected{\tau\sum_{n=1}^me^{-\hat{c}t\no}\frac1\varepsilon\trkla{f_\delta\trkla{U\nn}, U\nn}\h}\geq \expected{\int_0^t e^{-\hat{c}s}\frac1\varepsilon\trkla{f_\delta\trkla{U\tp},U\tp}\h\ds}\\
\geq \expected{\int_0^t e^{-\hat{c}s} \frac1\varepsilon\rkla{\trkla{f_\delta\trkla{u_\delta},U\tp-u_\delta} + \trkla{\Ih{f_\delta\trkla{U\tp}},u_\delta}}\ds} -Ch\,.
\end{multline}
To deal with the next part, we compute for $t\in[t^{m-1},t^m)$
\begin{align}
\begin{split}
&\expected{\sum_{n=1}^m\tau\frac1\varepsilon e^{-\hat{c}t\no} \trkla{f_\delta\trkla{U\nn},\mathfrak{m}}\h} = \expected{\int_0^{t^m}\frac1\varepsilon e^{-\hat{c}t\tm} \trkla{f_\delta\trkla{U\tp},\mathfrak{m}}\h\ds}\\
&\quad= \expected{\int_0^t \frac1\varepsilon e^{-\hat{c}s} \trkla{f_\delta\trkla{U\tp},\mathfrak{m}}\h\ds} +\expected{\int_t^{t^m}\frac1\varepsilon e^{-\hat{c}s}\rkla{f_\delta\trkla{U\tp},\mathfrak{m}}\h\ds}\\
&\qquad\quad +\expected{\int_0^{t^m}\frac1\varepsilon\rkla{e^{-\hat{c}t\tm}-e^{-\hat{c}s}}\rkla{f_\delta\trkla{U\tp},\mathfrak{m}}\h\ds}\,.
\end{split}
\end{align}
Here, the second term on the right-hand side is bounded by $C\tau$ and the last term vanishes for $\tau\searrow0$ due to the continuity of the exponential function.\\
For the last term on the left-hand side of \eqref{eq:strongconv:discIto}, we write
\begin{align}
\begin{split}
&\expected{2\sum_{n=1}^m\tau\frac1\varepsilon e^{-\hat{c}t\no}\rkla{-U\no,U\nn-\mathfrak{m}}\h}\\
&\quad\geq \expected{2\sum_{n=1}^m\tau\frac1\varepsilon e^{-\hat{c}t\no}\rkla{U\nn-U\no, U\nn-\mathfrak{m}}\h}\\
&\qquad\quad-2\expected{\int_0^{t^m}\frac1\varepsilon\rkla{e^{-\hat{c}t\tm}-e^{-\hat{c}s}}\norm{U\tp-\mathfrak{m}}\h^2\ds}\\
&\qquad\quad-2\expected{\int_0^{t}\frac{1}\varepsilon e^{-\hat{c}s}\norm{U\tp-\mathfrak{m}}_{\mathds{L}^2}^2\ds} -2\expected{\int_t^{t^m}\frac{1}\varepsilon e^{-\hat{c}s}\norm{U\tp-\mathfrak{m}}_{\mathds{L}^2}^2\ds}\\
&\qquad\quad-Ch^2\expected{\int_0^{t^m}e^{-\hat{c}s}\norm{\nabla U\tp}_{\mathds{L}^2}^2\ds}\,,
\end{split}
\end{align}
where we used Lemma \ref{lem:interpolation}.
Here, the first, second, and fourth term on the right-hand side will vanish for $\tau\searrow0$, while the last term vanishes for $h\searrow0$.
Concerning the first term on the right-hand side of \eqref{eq:strongconv:discIto}, we use
\begin{align}
e^{-\hat{c}t\no}-e^{-\hat{c}t^{n-2}}=\int_{t^{n-2}}^{t\no}  -\hat{c}e^{-\hat{c}s}\ds
\end{align}
and \eqref{eq:tmp:*-1ineq} to derive
\begin{align}
\begin{split}
&\expected{\sum_{n=2}^m\rkla{e^{-\hat{c}t\no}-e^{-\hat{c}t^{n-2}}}\norm{U\no-\mathfrak{m}}_*^2}\\
&\quad=\expected{\sum_{n=1}^{m-1}\int_{t\no}^{t\nn}-\hat{c}e^{-\hat{c}s}\norm{U\tp-\mathfrak{m}}_*^2\ds}=\expected{\int_0^{t^{m-1}}-\hat{c}e^{-\hat{c}s}\norm{U\tp-\mathfrak{m}}_*^2\ds}\\
&\quad\leq \trkla{1+h^{1/2}}^{-1}\expected{\int_0^{t}-\hat{c}e^{-\hat{c}s}\norm{U\tp-\mathfrak{m}}_{-1}^2\ds}\\
&\qquad\quad+Ch^{1/2}\expected{\int_0^{t^{m-1}}\hat{c}e^{-\hat{c}s}\norm{U\tp}_{\mathds{L}^2}^2\ds}+C\tau\,.
\end{split}
\end{align}
Here, we again estimated the integral from $t$ to $t^m$ by $C\tau$.
Due to the already established regularity results, the second term vanishes for $h\rightarrow0$.

To deal with the last term in \eqref{eq:strongconv:discIto}, we integrate by parts multiple times to derive the identity 
\begin{align}\label{eq:strongconv:hminus1}
\begin{split}
\sum_{i=1}^d&\, \trkla{\nabla\trkla{-\Delta^{-1}}\partial_i v,\nabla\trkla{-\Delta^{-1}}\partial_i v}=-\sum_{i=1}^d\trkla{ v-\trkla{v}_{\D},\partial_i\trkla{-\Delta^{-1}}\partial_i v}\\
&=-\sum_{i=1}^d\rkla{-\Delta\trkla{-\Delta^{-1}}\trkla{v-\trkla{v}_{\D}},\partial_i\trkla{-\Delta^{-1}}\partial_i v} \\
&=-\sum_{i=1}^d\trkla{\nabla\trkla{-\Delta^{-1}}\trkla{v-\trkla{v}_{\D}},\nabla\partial_i\trkla{-\Delta^{-1}}\partial_i v} \\
&=\sum_{i=1}^d\rkla{\nabla\partial_i\trkla{-\Delta^{-1}}\trkla{v-\trkla{v}_{\D}},\nabla\trkla{-\Delta^{-1}}\partial_i v}\\
&=\sum_{i=1}^d\trkla{\partial_i\trkla{-\Delta^{-1}}\trkla{v-\trkla{v}_{\D}},\partial_i v}
=\norm{v-\trkla{v}_{\D}}_{\mathds{L}^2}^2\,,
\end{split}
\end{align}
which holds true for all $v\in \mathds{H}^1$.
Here we used that $\trkla{-\Delta^{-1}}$ maps to $\mathds{H}^2_{\per}$ allowing us to interchange derivatives.
Hence, we have
\begin{multline}
\sum_{i=1}^d\sum_{k\in\mathds{Z}\h}\norm{\nabla\trkla{-\Delta^{-1}} \partial_i\Ih{\mysigma\trkla{U\no}\g{k}\lambda_k}}_{\mathds{L}^2}^2 \\
= \sum_{k\in\Zh} \norm{\Ih{\mysigma\trkla{U\no}\g{k}\lambda_k} -\trkla{\Ih{\mysigma\trkla{U\no}\g{k}\lambda_k}}_{\D}}_{\mathds{L}^2}^2\,.
\end{multline}
Furthermore, we use the decomposition
\begin{align}
\begin{split}
\Ih{\mysigma\trkla{U\no}\lambda_k\g{k}}&-\trkla{\Ih{\mysigma\trkla{U\no}\lambda_k\g{k}}}_{\D}\\
=&\,\mysigma\trkla{U\no}\lambda_k\g{k}-\trkla{\mysigma\trkla{U\no}\lambda_k\g{k}}_{\D}\\
&-\lambda_k\trkla{1-\Ihop}\gkla{\g{k}}\mysigma\trkla{U\no} -\lambda_k\Ih{\g{k}}\trkla{1-\Ihop}\gkla{\mysigma\trkla{U\no}}\\
&-\lambda_k\trkla{1-\Ihop}\gkla{\Ih{\g{k}}\Ih{\mysigma\trkla{U\no}}}\\
&+\rkla{\lambda_k\trkla{1-\Ihop}\gkla{\g{k}}\mysigma\trkla{U\no}}_{\D} +\rkla{\lambda_k\Ih{\g{k}}\trkla{1-\Ihop}\gkla{\mysigma\trkla{U\no}}}_{\D}\\
&+\rkla{\lambda_k\trkla{1-\Ihop}\gkla{\Ih{\g{k}}\Ih{\mysigma\trkla{U\no}}}}_{\D}\,.
\end{split}
\end{align}
Hence, using Hölder's inequality, standard error estimates for $\Ihop$, the assumption $\mysigma\in C^{0,1}\trkla{\mathds{R}}\cap L^\infty\trkla{\mathds{R}}$, Lemma \ref{lem:interpolation}, and Young's inequality, we deduce
\begin{align}
\begin{split}
&\norm{\Ih{\mysigma\trkla{U\no}\g{k}\lambda_k} -\trkla{\Ih{\mysigma\trkla{U\no}\g{k}\lambda_k}}_{\D}}_{\mathds{L}^2}^2\\
&\quad\leq\norm{\mysigma\trkla{U\no}\g{k}\lambda_k -\trkla{\mysigma\trkla{U\no}\lambda_k\g{k}}_{\D}}_{\mathds{L}^2}^2 \\
&\qquad+\norm{\mysigma\trkla{U\no}\g{k}\lambda_k -\trkla{\mysigma\trkla{U\no}\lambda_k\g{k}}_{\D}}_{\mathds{L}^2}C h\tabs{\lambda_k}\rkla{\norm{\g{k}}_{\mathds{W}^{1,\infty}} +\norm{\g{k}}_{\mathds{L}^\infty}\norm{\nabla U\no}_{\mathds{L}^2}}\\
&\qquad+{Ch^2\rkla{\tabs{\lambda_k}^2\norm{\g{k}}_{\mathds{W}^{1,\infty}}^2+ \tabs{\lambda_k}^2\norm{\g{k}}_{\mathds{L}^\infty}^2\norm{\nabla U\no}_{\mathds{L}^2}^2}}\\
&\quad\leq\trkla{1+h^{1/2}}\norm{\mysigma\trkla{U\no}\g{k}\lambda_k -\trkla{\mysigma\trkla{U\no}\lambda_k\g{k}}_{\D}}_{\mathds{L}^2}^2  \\
&\qquad+Ch^{1/2}\tabs{\lambda_k}^2\norm{\g{k}}_{\mathds{W}^{1,\infty}}^2 \rkla{1+\norm{\nabla U\no}_{\mathds{L}^2}^2}\,.
\end{split}
\end{align}
Thus,
\begin{align}
\begin{split}
&\expected{\sum_{n=1}^m \tau e^{-\hat{c}t\no}\sum_{i=1}^d\sum_{k\in\Zh} \norm{\nabla\trkla{-\Delta^{-1}}\partial_i\Ih{\mysigma\trkla{U\no}\g{k}\lambda_k}}_{\mathds{L}^2}^2}\\
&\quad\leq \trkla{1+h^{1/2}}\expected{\sum_{n=1}^m\tau e^{-\hat{c}t\no}\sum_{k\in\Zh}\norm{\mysigma\trkla{U\no}\g{k}\lambda_k-\trkla{\mysigma\trkla{U\no}\lambda_k\g{k}}_{\D}}_{\mathds{L}^2}^2}\\
&\quad\qquad+Ch^{1/2}\expected{\sum_{n=1}^m\tau e^{-\hat{c}t\no}\rkla{1+\norm{\nabla U\no}_{\mathds{L}^2}^2}}\,.
\end{split}
\end{align}

Due to the uniform bounds on $\nabla U\no$ established in Lemma \ref{lem:l2estimate}, the second term will vanish for $h\searrow0$.
We rewrite the first term using
\begin{align}
\begin{split}
&\expected{\sum_{n=1}^m\tau e^{-\hat{c}t\no}\sum_{k\in\Zh}\norm{\mysigma\trkla{U\no}\g{k}\lambda_k-\trkla{\mysigma\trkla{U\no}\lambda_k\g{k}}_{\D}}_{\mathds{L}^2}^2} \\
&\quad=\expected{\tau\sum_{k\in\mathds{Z}\h}\norm{\mysigma\trkla{U^0}\g{k}\lambda_k-\trkla{\mysigma\trkla{U^0}\lambda_k\g{k}}_{\D}}_{\mathds{L}^2}^2} \\
&\qquad\quad+ \expected{\sum_{n=1}^{m-1}\tau e^{-\hat{c}t\nn}\sum_{k\in\Zh}\norm{\mysigma\trkla{U\nn}\g{k}\lambda_k-\trkla{\mysigma\trkla{U\nn}\lambda_k\g{k}}_{\D}}_{\mathds{L}^2}^2}\\
&\quad\leq C \tau + \expected{\int_0^t e^{-\hat{c}s}\sum_{k\in\mathds{Z}\h}\norm{\mysigma\trkla{U\tp}\g{k}\lambda_k-\trkla{\mysigma\trkla{U\tp}\lambda_k\g{k}}_{\D}}_{\mathds{L}^2}^2\ds}
\end{split}
\end{align}
for $t\in[t^{m-1},t^m)$.
Collecting the above results, we obtain
\begin{align}
\begin{split}\label{eq:strongconvproof:*semi}
&\frac{1}{1+h^{1/2}}\expected{e^{-\hat{c}t}\norm{U\tp-\mathfrak{m}}_{-1}^2} \\
&\quad+2\varepsilon\expected{\int_0^te^{-\hat{c}s}\rkla{\norm{\nabla \rkla{U\tp-u_\delta}}_{\mathds{L}^2}^2 +2\rkla{\nabla U\tp,\nabla u_\delta}-\norm{\nabla u_\delta}_{\mathds{L}^2}^2}\ds}\\
&\quad+2\frac1\varepsilon\expected{\int_0^te^{-\hat{c}s}\rkla{\rkla{f_\delta\trkla{u_\delta},U\tp-u_\delta}+\rkla{\Ih{f_\delta\trkla{U\tp}},u_\delta-\mathfrak{m}}}\ds}\\
&\leq \trkla{1+h^{1/2}}\expected{\norm{U^0-\mathfrak{m}}_{-1}^2 }+ 2\frac1\varepsilon\expected{\int_0^t e^{-\hat{c}s}\norm{U\tp-\mathfrak{m}}_{\mathds{L}^2}^2\ds} \\
&\quad+\trkla{1+h^{1/2}}^{-1}\expected{\int_0^{t}-\hat{c}e^{-\hat{c}s}\norm{U\tp-\mathfrak{m}}_{-1}^2\ds}\\
&\quad +\trkla{1+h^{1/2}}\expected{\int_0^te^{-\hat{c}s}\sum_{k\in\mathds{Z}\h}\norm{\mysigma\trkla{U\tp}\lambda_k\g{k}-\trkla{\mysigma\trkla{U\tp}\lambda_k\g{k}}_{\D}}_{\mathds{L}^2}^2} +\mathfrak{R}\trkla{h+\tau}
\end{split}
\end{align}
with $\mathfrak{R}\trkla{s}\rightarrow 0$ for $s\rightarrow 0$.\\
Next, we want to apply \Ito's formula to $F\trkla{s,u_\delta\trkla{s}}:=e^{-\hat{c}s}\norm{u_\delta\trkla{s}-\mathfrak{m}}_{-1}^2$ to derive a continuous counterpart of \eqref{eq:strongconvproof:*semi}.
However, as $-\Delta w_\delta$ is only in $\mathds{H}^{-2}_\per$, a direct application of \Ito's formula to the $\mathds{H}^{-1}_\per$-norm is not possible. 
Thus, we introduce a family of regularised versions of the $\mathds{H}^{-1}_\per$-norm:
Starting from a nonnegative, rotationally symmetric function $\zeta\in C^\infty\trkla{\mathds{R}^d}$ with $\mathrm{supp}\,\zeta\subset\overline{B_1\trkla{0}}$, we define $\zeta_\varepsilon\trkla{x}:=\varepsilon^{-d}\zeta\trkla{\varepsilon^{-1}x}$ for $\varepsilon>0$. 
As the considered functions are periodic, they are well-defined on each $\widehat{\D}\supset\D$ with $\mathrm{dist}\trkla{\partial\widehat{\D},\D}\geq 1$.
For each $\varepsilon<1$, we can now define a mollifier
\begin{align}
\mathfrak{J}_\varepsilon\tgkla{f}\trkla{x}:=\int_{\widehat{\D}}\zeta_\varepsilon\trkla{x-y}f\trkla{y}\dy
\end{align}
for all $x\in\D$ and $f\in \mathds{L}^1$ that are extended periodically to $\widehat{\D}$.
This mollifier satisfies $\norm{\trkla{1-\mathfrak{J}_\varepsilon}\tgkla{f}}_{\mathds{L}^2}\rightarrow 0$ as $\varepsilon\searrow0$ for $f\in\mathds{L}^2$ and $\partial_i\mathfrak{J}_\varepsilon\tgkla{g}=\mathfrak{J}_\varepsilon\tgkla{\partial_i g}$ for all $g\in \mathds{H}^1_\per$.
Using $\mathfrak{J}_\varepsilon$, we define a regularised version of the $\mathds{H}^{-1}_\per$-norm as
\begin{align}\label{eq:jnorm}
\jnorm{a}^2:=\norm{\nabla\mathfrak{J}_\varepsilon\tgkla{\trkla{-\Delta^{-1}} a}}_{\mathds{L}^2}^2
\end{align}
for all $a\in\mathds{H}^{-2}_\per$.
Applying \Ito's formula to $F\trkla{s,u_\delta\trkla{s}}:=e^{-\hat{c}\trkla{s}}\jnorm{u_\delta\trkla{s}-\mathfrak{m}}^2$, we derive
\begin{align}\label{eq:itojnorm}
\begin{split}
\expected{e^{-\hat{c}t}\jnorm{u_\delta\trkla{t}-\mathfrak{m}}^2}=&\,\expected{\jnorm{u_\delta\trkla{0}-\mathfrak{m}}^2} +\expected{\int_0^t-\hat{c}e^{-\hat{c}s}\jnorm{u_\delta\trkla{s}-\mathfrak{m}}^2\ds}\\
&\,-2\expected{\int_0^te^{-\hat{c}s}\rkla{\nabla\mathfrak{J}_\varepsilon\gkla{\trkla{-\Delta^{-1}}u_\delta\trkla{s}} ,\nabla\mathfrak{J}_\varepsilon\tgkla{w_\delta\trkla{s}}}\ds}\\
&\,+\expected{\int_0^t e^{-\hat{c}s}\sum_{i=1}^d\sum_{k\in\mathds{Z}}\jnorm{\partial_i\tgkla{\mysigma_\delta\lambda_k\g{k}}}^2\ds}\,.
\end{split}
\end{align}
Using the properties of $\mathfrak{J}_\varepsilon$, we can pass to the limit $\varepsilon\searrow0$ and obtain after inserting \eqref{eq:limitht:w}
\begin{align}\label{eq:strongconvproof:ito}
\begin{split}
\expected{e^{-\hat{c}t}\norm{u_\delta\trkla{t}-\mathfrak{m}}_{-1}^2} =&\, \expected{\norm{u_\delta\trkla{0}-\mathfrak{m}}_{-1}^2 }\\
&+\expected{\int_0^t-\hat{c}e^{-\hat{c}s}\norm{u_\delta\trkla{s}-\mathfrak{m}}_{-1}^2\ds}\\
&+2\expected{\int_0^te^{-\hat{c}s}\iD-\varepsilon\tabs{\nabla u_\delta\trkla{s}}^2-\frac1\varepsilon p_\delta\trkla{s}\trkla{u_\delta\trkla{s}-\mathfrak{m}}\dx\ds}\\
&+2\frac1\varepsilon\expected{\int_0^te^{-\hat{c}s}\norm{u_\delta\trkla{s}-\mathfrak{m}}_{\mathds{L}^2}^2\ds}\\
&+\expected{\int_0^te^{-\hat{c}s}\sum_{i=1}^d\sum_{k\in\mathds{Z}}\norm{\partial_i\gkla{\mysigma_\delta\lambda_k\g{k}}}_{-1}^2\ds}\,.
\end{split}
\end{align}
Recalling \eqref{eq:strongconv:hminus1}, we can rewrite the last term as
\begin{align}
\expected{\int_0^te^{-\hat{c}s}\sum_{i=1}^d\sum_{k\in\mathds{Z}}\norm{\partial_i\gkla{\mysigma_\delta\lambda_k\g{k}}}_{-1}^2\ds}=\expected{\int_0^te^{-\hat{c}s}\sum_{k\in\mathds{Z}}\norm{\mysigma_\delta\lambda_k\g{k} -\trkla{\mysigma_\delta\lambda_k\g{k}}_{\D}}_{\mathds{L}^2}^2\ds}\,.
\end{align}
Hence, multiplying \eqref{eq:strongconvproof:*semi} by $\trkla{1+h^{1/2}}$, recalling the convergence 
\begin{align}
\norm{U^0-\mathfrak{m}}_{-1}\rightarrow \norm{u_\delta\trkla{0}-\mathfrak{m}}_{-1}=\norm{u\trkla{0}-\mathfrak{m}}_{-1} 
\end{align}
for $h\searrow0$ due to assumption \ref{assumption:I} and standard error estimates for the orthogonal $\mathds{L}^2$-projection $\mathcal{P}$, and combining the above estimates \eqref{eq:strongconvproof:*semi}--\eqref{eq:strongconvproof:ito}, integrating with respect to $t$ and redefining $\mathfrak{R}$ provides
\begin{align}\label{eq:tmp:intermediateestimate}
\begin{split}
\int_0^T&\,\expected{e^{-\hat{c}t}\norm{U\tp-\mathfrak{m}}_{-1}^2}\dt \\
&\quad+2\varepsilon\int_0^T\expected{\int_0^te^{-\hat{c}s}\rkla{\norm{\nabla U\tp-\nabla u_\delta}_{\mathds{L}^2}^2+2\trkla{\nabla U\tp,\nabla u_\delta} -\norm{\nabla u_\delta}_{\mathds{L}^2}^2}\ds}\dt\\
&\quad+2\frac{1}\varepsilon\int_0^T\expected{\int_0^te^{-\hat{c}s}\rkla{\rkla{f_\delta\trkla{u_\delta}, U\tp-u_\delta} +\rkla{\Ih{f_\delta\trkla{U\tp}},u_\delta-\mathfrak{m}}}\ds}\dt\\
&\leq \int_0^T\expected{e^{-\hat{c}t}\norm{u_\delta-\mathfrak{m}}_{-1}^2}\dt-\int_0^T\expected{\int_0^t-\hat{c}e^{-\hat{c}s}\norm{u_\delta-\mathfrak{m}}_{-1}^2\ds}\dt \\
&\quad+2\varepsilon \int_0^T\expected{\int_0^te^{-\hat{c}s}\norm{\nabla u_\delta}_{\mathds{L}^2}^2\ds}\dt\\
&\quad +2\frac1\varepsilon\int_0^T \expected{\int_0^te^{-\hat{c}s}\iD p_\delta\trkla{u_\delta-\mathfrak{m}}\dx\ds}\dt\\
&\quad-2\frac{1}\varepsilon\int_0^T\expected{\int_0^te^{-\hat{c}s}\norm{u_\delta-\mathfrak{m}}_{\mathds{L}^2}^2\ds}\dt \\
&\quad-\int_0^T\expected{\int_0^te^{-\hat{c}s}\sum_{k\in\mathds{Z}}\norm{\mysigma_\delta\lambda_k\g{k} -\trkla{\mysigma_\delta\lambda_k\g{k}}_{\D}}_{\mathds{L}^2}^2\ds}\dt\\
&\quad+2\frac1\varepsilon\int_0^T\expected{\int_0^te^{-\hat{c}s}\norm{U\tp-\mathfrak{m}}_{\mathds{L}^2}^2\ds}\dt\\
&\quad+\int_0^T\expected{\int_0^t-\hat{c}e^{-\hat{c}s}\norm{U\tp-\mathfrak{m}}_{-1}^2\ds}\dt\\
&\quad+\int_0^T\expected{\int_0^te^{-\hat{c}s}\sum_{k\in\mathds{Z}\h}\norm{\mysigma_\trkla{U\tp}\lambda_k\g{k}-\trkla{\mysigma\trkla{U\tp}\lambda_k\g{k}}_{\D}}_{\mathds{L}^2}^2\ds}\dt  +\mathfrak{R}\trkla{h+\tau}\,.
\end{split}
\end{align}
We continue by estimating
\begin{align}
\begin{split}
\norm{U\tp-\mathfrak{m}}_{-1}^2=&\,\norm{U\tp-u_\delta }_{-1}^2+2\rkla{U\tp-\mathfrak{m},u_\delta-\mathfrak{m}}_{-1} -\norm{u_\delta-\mathfrak{m}}_{-1}^2\,,
\end{split}
\end{align}
\begin{align}
\begin{split}
\norm{U\tp-\mathfrak{m}}_{\mathds{L}^2}^2=&\,\norm{U\tp -u_\delta }_{\mathds{L}^2}^2+2\rkla{U\tp-\mathfrak{m},u_\delta-\mathfrak{m}} -\norm{u_\delta-\mathfrak{m}}_{\mathds{L}^2}^2\\
\leq &\,\tilde{\alpha} \norm{\nabla U\tp -\nabla u_\delta}_{\mathds{L}^2}^2 +C_{\tilde{\alpha}}\norm{U\tp-u_\delta}_{-1}^2\\
&+2\rkla{U\tp-\mathfrak{m},u_\delta-\mathfrak{m}} -\norm{u_\delta-\mathfrak{m}}_{\mathds{L}^2}^2\,,
\end{split}
\end{align}
\begin{align}\label{eq:noise}
\begin{split}
&\norm{\mysigma\trkla{U\tp}\lambda_k\g{k}-\trkla{\mysigma\trkla{U\tp}\lambda_k\g{k}}_{\D}}_{\mathds{L}^2}^2=-\norm{\mysigma\trkla{u_\delta}\lambda_k\g{k}-\trkla{\mysigma\trkla{u_\delta}\lambda_k\g{k}}_{\D}}_{\mathds{L}^2}^2\\
&\qquad+2\trkla{\mysigma\trkla{U\tp}\lambda_k\g{k}-\trkla{\mysigma\trkla{U\tp}\lambda_k\g{k}}_{\D},\mysigma\trkla{u_\delta}\lambda_k\g{k}-\trkla{\mysigma\trkla{u_\delta}\lambda_k\g{k}}_{\D}}\\
&\qquad+\norm{\tekla{\mysigma\trkla{U\tp}\lambda_k\g{k}-\trkla{\mysigma\trkla{U\tp}\lambda_k\g{k}}_{\D}}-\tekla{\mysigma\trkla{u_\delta}\lambda_k\g{k}-\trkla{\mysigma\trkla{u_\delta}\lambda_k\g{k}}_{\D}}}_{\mathds{L}^2}^2\\
&\quad\leq \norm{\mysigma_\delta\lambda_k\g{k}-\trkla{\mysigma_\delta\lambda_k\g{k}}_{\D}}_{\mathds{L}^2}^2- 2\trkla{\mysigma_\delta\lambda_k\g{k}-\trkla{\mysigma_\delta\lambda_k\g{k}}_{\D},\mysigma\trkla{u_\delta}\lambda_k\g{k}-\trkla{\mysigma\trkla{u_\delta}\lambda_k\g{k}}_{\D}}\\
&\qquad+2\trkla{\mysigma\trkla{U\tp}\lambda_k\g{k}-\trkla{\mysigma\trkla{U\tp}\lambda_k\g{k}}_{\D},\mysigma\trkla{u_\delta}\lambda_k\g{k}-\trkla{\mysigma\trkla{u_\delta}\lambda_k\g{k}}_{\D}}\\
&\qquad+\norm{\tekla{\mysigma\trkla{U\tp}\lambda_k\g{k}-\trkla{\mysigma\trkla{U\tp}\lambda_k\g{k}}_{\D}}-\tekla{\mysigma\trkla{u_\delta}\lambda_k\g{k}-\trkla{\mysigma\trkla{u_\delta}\lambda_k\g{k}}_{\D}}}_{\mathds{L}^2}^2\,.
\end{split}
\end{align}

To estimate the last term on the right-hand side of \eqref{eq:noise}, we use Hölder's inequality, the Lipschitz continuity of $\mysigma$, and Young's inequality to compute
\begin{align}
\begin{split}
&\norm{\tekla{\mysigma\trkla{U\tp}\lambda_k\g{k}-\trkla{\mysigma\trkla{U\tp}\lambda_k\g{k}}_{\D}}-\tekla{\mysigma\trkla{u_\delta}\lambda_k\g{k}-\trkla{\mysigma\trkla{u_\delta}\lambda_k\g{k}}_{\D}}}_{\mathds{L}^2}^2\\
&\quad\leq C\norm{\trkla{\mysigma\trkla{U\tp}-\mysigma\trkla{u_\delta}}\lambda_k\g{k}}_{\mathds{L}^2}^2 \leq C\norm{U\tp- u_\delta}_{L^2\trkla{\D}}^2\tabs{\lambda_k}^2\norm{\g{k}}_{\mathds{L}^\infty}^2\\
&\quad\leq C\norm{\nabla U\tp-\nabla u_\delta}_{\mathds{L}^2}\norm{U\tp-u_\delta }_{-1} \tabs{\lambda_k}^2\norm{\g{k}}_{\mathds{L}^\infty}^2\\
&\quad\leq \rkla{\alpha\norm{\nabla U\tp-\nabla u_\delta}_{\mathds{L}^2}^2+C_\alpha\norm{U\tp-u_\delta }_{-1}^2}\tabs{\lambda_k}^2\norm{\g{k}}_{\mathds{L}^\infty}^2\,,
\end{split}
\end{align}
as $\trkla{U\tp-u_\delta}_{\D}=0$.
Therefore, we can write \eqref{eq:tmp:intermediateestimate} as
\begin{align}
\begin{split}
&\int_0^T\expected{e^{-\hat{c}t}\rkla{\norm{U\tp-\mathfrak{m}}_{-1}^2-\norm{u_\delta-\mathfrak{m}}_{-1}^2}}\dt\\
&\quad+\trkla{2\varepsilon-\frac{2\tilde{\alpha}}{\varepsilon}-\alpha\sum_{k\in\mathds{Z}}\tabs{\lambda_k}^2\norm{\g{k}}_{\mathds{L}^\infty}^2 }\int_0^T\expected{\int_0^t e^{-\hat{c}s}\norm{\nabla U\tp-\nabla u_\delta}_{\mathds{L}^2}^2\ds}\dt\\
&\quad+4\varepsilon\int_0^T\expected{\int_0^te^{-\hat{c}s}\rkla{\trkla{\nabla U\tp,\nabla u_\delta}-\norm{\nabla u_\delta}_{\mathds{L}^2}^2}\ds}\dt\\
&\quad+\frac2\varepsilon\int_0^T\expected{\int_0^t e^{-\hat{c}s}\rkla{\trkla{f_\delta\trkla{u_\delta},U\tp-u_\delta}+\trkla{\Ih{f_\delta\trkla{U\tp}}-p_\delta,u_\delta-\mathfrak{m}}  }\ds}\dt\\
&\quad+\trkla{\hat{c}-\frac2\varepsilon C_{\tilde{\alpha}} - C_\alpha\sum_{k\in\mathds{Z}}\tabs{\lambda_k}^2\norm{\g{k}}_{\mathds{L}^\infty}^2}\int_0^T\expected{\int_0^te^{-\hat{c}s}\norm{U\tp-u_\delta}_{-1}^2\ds}\dt\\
&\quad+2\hat{c}\int_0^T\expected{\int_0^te^{-\hat{c}s}\rkla{\trkla{U\tp-\mathfrak{m},u_\delta-\mathfrak{m}}_{-1}-\norm{u_\delta-\mathfrak{m}}_{-1}^2}\ds}\dt\\
&\leq \frac4\varepsilon \int_0^T\expected{\int_0^t e^{-\hat{c}s}\rkla{ \trkla{U\tp-\mathfrak{m},u_\delta-\mathfrak{m}}-\norm{u_\delta-\mathfrak{m}}_{\mathds{L}^2}^2}\ds}\dt\\
&\quad+2\int_0^T\mathds{E}\left[\int_0^te^{-\hat{c}s}\sum_{k\in\Zh}\iD\rkla{\tekla{\mysigma\trkla{U\tp}-\mysigma_\delta}\lambda_k\g{k}-\trkla{\tekla{\mysigma\trkla{U\tp}-\mysigma_\delta}\lambda_k\g{k}}_{\D}}\right. \\
&\qquad\qquad\qquad\left.\phantom{\int_0^T}\times\rkla{\mysigma\trkla{u_\delta}\lambda_k\g{k}-\trkla{\mysigma\trkla{u_\delta}\lambda_k\g{k}}_{\D}}\dx\ds\right]\dt\\
&\quad+\int_0^T\expected{\int_0^te^{-\hat{c}s}\sum_{k\in\mathds{Z}\setminus\Zh}\norm{\mysigma_\delta\lambda_k\g{k}-\trkla{\mysigma_\delta\lambda_k\g{k}}_{\D}}_{\mathds{L}^2}^2\ds}\dt
 +\mathfrak{R}\trkla{h+\tau}\,.
\end{split}
\end{align}
Choosing $\alpha$ and $\tilde{\alpha}$ sufficiently small and $\hat{c}$ large enough we obtain by passing to the superior limit
\begin{align}
\limsup_{\trkla{h,\tau}\rightarrow 0} \int_0^T\expected{e^{-\hat{c}t}\norm{U\tp-\mathfrak{m}}_{-1}^2}\dt\leq \int_0^T\expected{ e^{-\hat{c}t}\norm{u_\delta-\mathfrak{m}}_{-1}^2}\dt\,,
\end{align}
i.e., we have $U\tp\rightarrow u_\delta$ in $L^2\trkla{\Omega;L^2\trkla{0,T;\mathds{H}^{-1}_{\per}}}$.
Due to Lemma \ref{lem:h-1estimate}, this result is also valid for $U\tl$ and $U\tm$.
Applying Ehrling's lemma, we can improve this result to strong convergence in $L^2\trkla{\Omega;L^2\trkla{0,T;\mathds{L}^{2}}}$. 
As $\mysigma\in L^\infty\trkla{\mathds{R}}\cap C^{0,1}\trkla{\mathds{R}}$ (cf.~assumption \ref{assumption:C}), an application of Vitali's convergence theorem provides the strong convergence of $\mysigma\trkla{U\tpm}$.\\
Recalling \eqref{eq:strongconv:Lipschitz}, the strong convergence \eqref{eq:strongconvfd} of $\Ih{f_\delta\trkla{U\tpm}}$ follows from the uniform $L^p\trkla{\Omega;L^\infty\trkla{0,T;\mathds{H}^1_\per}}$-bound of $U\tpm$ in \eqref{eq:boundhtau:independent}, the pointwise almost everywhere convergence of $U\tpm$, and Vitali's theorem.
\end{proof}
Hence, our $\delta$-regularised problem reads as follows:
\begin{subequations}\label{eq:problemdelta}
\begin{align}\label{eq:problemdelta:u}
\mathrm{d} u_\delta-\Delta w_\delta\dt= \div\gkla{\mysigma\trkla{u_\delta}\dW}&&\text{in~} \mathds{H}^{-2}_\per
\end{align}
$\Prob$-almost surely and for almost all $t\in\trkla{0,T}$, where $w_\delta$ is given as
\begin{align}\label{eq:problemdelta:w}
w_\delta = -\varepsilon \Delta u_\delta +\tfrac1\varepsilon f_\delta\trkla{u_\delta} -\tfrac1\varepsilon u_\delta&&\text{in~}\mathds{L}^2\,.
\end{align}
\end{subequations}

\section{Pathwise uniqueness of the $\delta$-problem}\label{sec:uniquenessdelta}
In this section, we show pathwise uniqueness of the solution $u_\delta$ of \eqref{eq:problemdelta} obtained in the last section as the limit of the regularised numerical approximation \eqref{eq:model:disc:regularised}.
\begin{lemma}\label{lem:delta:unique}
The solution to problem \eqref{eq:problemdelta} is pathwise unique.
\end{lemma}
\begin{proof}
We assume that $u_\delta$ and $\widetilde{u}_\delta$ are two solutions to problem \eqref{eq:problemdelta} with the same initial condition (i.e., $\trkla{u_\delta-\widetilde{u}_\delta}_{\D}\equiv0$) and driven by the same $\mathcal{Q}$-Wiener process $\bs{\W}$.
For each of these solutions, there exists a unique chemical potential denoted by $w_\delta$ and $\widetilde{w}_\delta$ defined via \eqref{eq:limitht:w}. 
Hence, the difference $u_\delta-\widetilde{u}_\delta$ satisfies
\begin{align}
\mathrm{d}\trkla{u_\delta-\widetilde{u}_\delta}-\Delta\trkla{w_\delta-\widetilde{w}_\delta}\dt=\div\gkla{\rkla{\mysigma\trkla{u_\delta}-\mysigma\trkla{\widetilde{u}_\delta}}\dW} &&\text{in~}\mathds{H}^{-2}_\per\,.
\end{align}
For a sufficiently large constant $\hat{c}$, which will be specified below, we define $F\trkla{t,u}:=e^{-\hat{c}t}\norm{u}_{-1}^2$ and obtain via \Ito's formula (using the approximation \eqref{eq:jnorm} first and passing to the limit $\varepsilon\searrow0$), the monotonicity of $f_\delta$, and \eqref{eq:strongconv:hminus1}
\begin{align}
\begin{split}
&\expected{e^{-\hat{c}t}\norm{u_\delta-\widetilde{u}_\delta}_{-1}^2}=\expected{\int_0^t -\hat{c}e^{-\hat{c}s}\norm{u_\delta-\widetilde{u}_\delta}_{-1}^2\ds}\\
&\qquad-2\expected{\int_0^t e^{-\hat{c}s}\rkla{u_\delta-\widetilde{u}_\delta, -\varepsilon\Delta\trkla{u_\delta-\widetilde{u}_\delta} + \tfrac1\varepsilon\rkla{f_\delta\trkla{u_\delta}-f_\delta\trkla{\widetilde{u}_\delta}}-\tfrac1\varepsilon\trkla{u_\delta-\widetilde{u}_\delta}}\ds}\\
&\qquad+\expected{\int_0^te^{-\hat{c}s}\sum_{i=1}^d\sum_{k\in\mathds{Z}}\norm{\partial_i\ekla{\trkla{\mysigma\trkla{u_\delta}-\mysigma\trkla{\widetilde{u}_\delta}}\lambda_k\g{k}}}_{-1}^2\ds}\\
&\quad\leq\expected{\int_0^t -\hat{c}e^{-\hat{c}s}\norm{u_\delta-\widetilde{u}_\delta}_{-1}^2\ds} -2\varepsilon\expected{\int_0^te^{-\hat{c}s}\norm{\nabla\trkla{u_\delta-\widetilde{u}_\delta}}_{\mathds{L}^2}^2\ds}\\
&\qquad+\frac2\varepsilon\expected{\int_0^t e^{-\hat{c}s}\norm{u_\delta-\widetilde{u}_\delta}_{\mathds{L}^2}^2\ds} + C\expected{\int_0^t e^{-\hat{c}s}\norm{\mysigma\trkla{u_\delta}-\mysigma\trkla{\widetilde{u}_\delta}}_{\mathds{L}^2}^2\ds}\,.
\end{split}
\end{align}
Using the Lipschitz continuity of $\mysigma$ and
\begin{align}
\norm{u_\delta-\widetilde{u}_\delta}_{\mathds{L}^2}^2\leq \alpha\norm{\nabla\trkla{u_\delta-\widetilde{u}_\delta}}_{\mathds{L}^2}^2 +C_\alpha\norm{u_\delta-\widetilde{u}_\delta}_{-1}^2\,,
\end{align}
we deduce
\begin{align}
\begin{split}
\expected{e^{-\hat{c}t}\norm{u_\delta-\widetilde{u}_\delta}_{-1}^2}\leq&\,-\rkla{2\varepsilon-\frac{2\alpha}{\varepsilon}-\alpha}\expected{\int_0^te^{-\hat{c}s}\norm{\nabla\trkla{u_\delta-\widetilde{u}_\delta}}_{\mathds{L}^2}^2\ds}\\
&\,-\rkla{\hat{c}-C_\alpha\rkla{\frac2\varepsilon+1}}\expected{\int_0^te^{-\hat{c}s}\norm{u_\delta-\widetilde{u}_\delta}_{-1}^2\ds}\,.
\end{split}
\end{align}
Hence, choosing $\alpha$ sufficiently small and $\hat{c}$ sufficiently large provides the result.
\end{proof}

\section{$\delta$-independent estimates}\label{sec:est}

In this section we derive uniform estimates that are necessary to pass to the limit with respect to the regularisation parameter $\delta$.
We first establish a $\mathds{H}^1$-estimate in Section~\ref{sec:H1} and additional estimates in Section~\ref{sec:further} as its consequence.

\subsection{$\mathds{H}^1$-estimate}\label{sec:H1}
In this section we establish the following $\mathds{H}^1$-estimate:
\begin{lemma}\label{lem:H1}
Let $\tgkla{u_\delta}_{\delta}$ be the family of solutions obtained in Theorem \ref{thm:weakht}, and let the assumptions \ref{assumption:W1}, \ref{assumption:W2}, \ref{assumption:C}, and \ref{assumption:I} hold true.
Then, there exists a constant $C$ which may depend on $s$ but not on $\delta$ such that
\begin{align}
\expected{\esssup_{t\in\tekla{0,T}}\norm{u_\delta}_{\mathds{H}^1}^s}+\expected{\esssup_{t\in\tekla{0,T}}\norm{F_\delta\trkla{u_\delta}}_{\mathds{L}^1}^{s/2}} +\expected{\rkla{\int_0^T\norm{w_\delta}_{\mathds{H}^1}^2\dt}^{s/2}}\leq C
\end{align}
for $s\in[1,2)$.
\end{lemma}
To show Lemma~\ref{lem:H1} we adapt the ideas of \cite{DiPrimioGrasselliScarpa2024}. First we obtain a weighted version of the $\mathds{H}^1$-estimate in Lemma~\ref{lem:H1weighted} below,
where we establish control over the quantity
\begin{align}\label{eq:def:weightedEnergy}
H\trkla{t,u_\delta}:=e^{-\Lambda\trkla{t}}\rkla{\frac{\varepsilon}2\norm{\nabla u_\delta}_{\mathds{L}^2}^2+\frac1\varepsilon\iD F_\delta\trkla{u_\delta}\dx-\frac1{2\varepsilon}\norm{u_\delta}_{\mathds{L}^2}^2}\,.
\end{align}
Here, the weight $\Lambda$ is defined as
\begin{align}\label{eq:defLambda}
\Lambda\trkla{t}:=\gamma\int_0^t\norm{u_\delta\trkla{s}}_{\mathds{H}^2}^2\ds
\end{align}
with a suitably large positive constant $\gamma$ which will be specified later.
\begin{lemma}\label{lem:H1weighted}
Let $u_\delta$ be the the family of solutions obtained in Theorem \ref{thm:weakht}, and let the assumptions \ref{assumption:W1}, \ref{assumption:W2}, \ref{assumption:C}, and \ref{assumption:I} hold true. Then, the estimate
\begin{multline}
\esssup_{t\in\trkla{0,T}}\expected{H\trkla{t,u_\delta\trkla{t}}}+\expected{\int_0^Te^{-\Lambda\trkla{s}}\norm{w_\delta}_{\mathds{H}^1}^2\ds} +\expected{\int_0^Te^{-\Lambda\trkla{s}}\norm{u_\delta}_{\mathds{W}^{2,6}}^2\ds}\\
+\expected{\int_0^Te^{-\Lambda\trkla{s}}\norm{f_\delta\trkla{u_\delta}}_{\mathds{L}^6}^2\ds}+\expected{\int_0^Te^{-\Lambda\trkla{s}}\norm{u_\delta}_{\mathds{H}^2}^2\norm{\nabla u_\delta}_{\mathds{L}^2}^2\ds}\leq C\,
\end{multline}
holds true for $\gamma$ sufficiently large.
\end{lemma}
To prove Lemma \ref{lem:H1weighted}, we need additional tools and auxiliary results.
First, we approximate $u_\delta -\delta f_\delta\trkla{u_\delta}$ by a family of smooth, monotonously increasing functions
\begin{align}\label{eq:defJ}
\mathcal{J}_\delta\,:\, \mathds{R}\rightarrow \tekla{-1-\delta/2,1+\delta/2}
\end{align}
satisfying $\mathcal{J}_\delta\trkla{\varphi}=\varphi$ for $\varphi\in\tekla{-1,1}$, $\abs{\mathcal{J}_\delta\trkla{\varphi}}\leq \abs{\varphi}$ for all $\varphi\in\mathds{R}$, $0\leq \mathcal{J}_\delta^\prime\trkla{\varphi}\leq 1$, and $\mathcal{J}_\delta\trkla{1+\delta}=-\mathcal{J}_\delta\trkla{-1-\delta}=1+\delta/2$.
This implies
\begin{subequations}\label{eq:propJ}
\begin{align}
\mathcal{J}_\delta\trkla{\varphi}\varphi \geq 0\,,\\
\nabla\mathcal{J}_\delta\trkla{\varphi}\cdot\nabla\varphi\geq 0\,,\label{eq:propJ:b}\\
%0\leq \mathcal{J}_\delta^\prime\trkla{\varphi}\leq 1\,,\\
\mysigma^\prime\trkla{\varphi}=\mysigma^\prime\trkla{\mathcal{J}_\delta\trkla{\varphi}} 
\end{align}
\end{subequations}
$\forall\varphi\in\mathds{R}$ due to assumption \ref{assumption:C}.
As $\rkla{u_\delta-\mathcal{J}_\delta\trkla{u_\delta}}\equiv \delta f_\delta\trkla{u_\delta}$ in $\tgkla{\tabs{u_\delta}\leq 1}\cup \tgkla{\tabs{u_\delta}\geq 1+\delta}$ and a straightforward estimate provides
\begin{align}
\abs{u_\delta-\mathcal{J}_\delta\trkla{u_\delta}-\delta f_\delta\trkla{u_\delta}}\leq C\delta
\end{align}
for $u_\delta\in\tgkla{1<u_\delta<1+\delta}\cup\tgkla{-1-\delta<u_\delta<-1}$, we obtain the estimate
\begin{align}\label{eq:J-f}
\abs{\trkla{u_\delta-\mathcal{J}_\delta\trkla{u_\delta}}_{\D} -\delta\trkla{f_\delta\trkla{u_\delta}}_{\D}}\leq C\delta\,.
\end{align}
Next, we derive an estimate for the mean-value of $w_\delta$.
\begin{lemma}\label{lem:meanw}
Let $u_\delta$ be the the family of solutions obtained in Theorem \ref{thm:weakht}, and let the assumptions \ref{assumption:W1}, \ref{assumption:W2}, \ref{assumption:C}, and \ref{assumption:I} hold true. Then, 
\begin{align}
\abs{\trkla{w_\delta}_{\D}}\leq \widehat{C}\rkla{\norm{\nabla w_\delta}_{\mathds{L}^2} +\norm{u_\delta}_{\mathds{L}^2} +1}
\end{align}
holds true with a constant $\widehat{C}$ depending on $\mathfrak{m}$ and $\varepsilon$ but not on $\delta$.
\end{lemma}
\begin{proof}
Using $\rkla{f_\delta\trkla{u_\delta},\mathcal{J}_\delta\trkla{u_\delta}}\geq \norm{f_\delta\trkla{u_\delta}}_{\mathds{L}^1}$ and the the estimate \eqref{eq:J-f}, we deduce
\begin{align}
\begin{split}
&\rkla{f_\delta\trkla{u_\delta}, \mathcal{J}_\delta\trkla{u_\delta}-\rkla{\mathcal{J}_\delta\trkla{u_\delta}}_{\D}}= \rkla{f_\delta\trkla{u_\delta}, \mathcal{J}_\delta\trkla{u_\delta}-\mathfrak{m}} +\rkla{f_\delta\trkla{u_\delta}, \trkla{u_\delta-\mathcal{J}_\delta\trkla{u_\delta}}_{\D}}\\
&\geq\trkla{1-\tabs{\mathfrak{m}}}\norm{f_\delta\trkla{u_\delta}}_{\mathds{L}^1}+ \trkla{f_\delta\trkla{u_\delta},\trkla{\delta f_\delta\trkla{u_\delta}}_\D} +\rkla{f_\delta\trkla{u_\delta},\trkla{u_\delta-\mathcal{J}_\delta\trkla{u_\delta}-\delta f_\delta\trkla{u_\delta}}_\D} \\
&\geq \trkla{1-\tabs{\mathfrak{m}}}\norm{f_\delta\trkla{u_\delta}}_{\mathds{L}^1} +\tfrac\delta2\abs{\D}\trkla{f_\delta\trkla{u_\delta}}_{\D}^2-C\delta\,.
\end{split}
\end{align}
Due to Assumption \ref{assumption:I} the coefficient $\trkla{1-\tabs{\mathfrak{m}}}=:M$ is a given positive constant.
Hence, recalling \eqref{eq:problemdelta:w} and \eqref{eq:propJ:b}, we obtain using Poincar\'e's inequality and the boundedness of $\mathcal{J}_\delta$
\begin{align}
\begin{split}
\tfrac1\varepsilon\norm{f_\delta\trkla{u_\delta}}_{\mathds{L}^1}\leq&\,M^{-1}\tfrac1\varepsilon\ekla{\rkla{f_\delta\trkla{u_\delta},\mathcal{J}_\delta\trkla{u_\delta}-\rkla{\mathcal{J}_\delta\trkla{u_\delta}}_{\D}}+C\delta}\\
=&\,M^{-1}\rkla{\rkla{w_\delta,\mathcal{J}_\delta\trkla{u_\delta}-\rkla{\mathcal{J}_\delta\trkla{u_\delta}}_{\D}}+\tfrac1\varepsilon\rkla{u_\delta,\mathcal{J}_\delta\trkla{u_\delta}-\rkla{\mathcal{J}_\delta\trkla{u_\delta}}_{\D}}+C\tfrac\delta\varepsilon}\\
\leq&\, C\norm{\nabla w_\delta}_{\mathds{L}^2} +C\norm{u_\delta}_{\mathds{L}^2}+C\delta\,.
\end{split}
\end{align}
Using \eqref{eq:problemdelta:w} again, we deduce
\begin{align}
\begin{split}
\abs{\rkla{w_\delta}_{\D}}\leq&\, \tfrac1\varepsilon\abs{\rkla{f_\delta\trkla{u_\delta}}_{\D}}+\tfrac1\varepsilon\abs{\trkla{u_\delta}_{\D}}\leq \tfrac1\varepsilon \norm{f_\delta\trkla{u_\delta}}_{\mathds{L}^1} + C\\
\leq&\, C\norm{\nabla w_\delta}_{\mathds{L}^2} +C\norm{u_\delta}_{\mathds{L}^2} +C\,.
\end{split}
\end{align}
\end{proof}
As a final tool for the proof of Lemma \ref{lem:H1weighted}, we need to show that  control over $\norm{w_\delta}_{\mathds{H}^1}$ provides improved bounds for $u_\delta$:
\begin{lemma}\label{lem:W26}
Let $u_\delta$ be the the family of solutions obtained in Theorem \ref{thm:weakht}, and let the assumptions \ref{assumption:W1}, \ref{assumption:W2}, \ref{assumption:C}, and \ref{assumption:I} hold true.
Then, the estimates
\begin{subequations}
\begin{align}
\norm{f_\delta}_{\mathds{L}^6}\leq&\, C\trkla{\norm{w_\delta}_{\mathds{H}^1}+\norm{u_\delta}_{\mathds{H}^1}}\,,\label{eq:W26:1}\\
\norm{u_\delta}_{\mathds{W}^{2,6}}\leq&\, C\trkla{\norm{w_\delta}_{\mathds{H}^1}+\norm{u_\delta}_{\mathds{H}^1}}\label{eq:W26:2}
\end{align}
hold true almost everywhere in $\Omega\times\trkla{0,T}$ with a constant $C$ independent of $\delta$.
\end{subequations}
\end{lemma}
\begin{proof}
Adapting the lines of \cite[Corollary 4.1]{GiorginiGrasselliMiranville2017}, we test \eqref{eq:problemdelta:w} for fixed $\delta$ by the function $\abs{f_\delta\trkla{u_\delta}}^4f_\delta\trkla{u_\delta}$, which is well-defined as $u_\delta\in\mathds{H}^2$ almost everywhere in $\Omega\times\trkla{0,T}$.
This provides after integrating by parts and applying Young's inequality and a standard Sobolev embedding
\begin{align}
\begin{split}
5\varepsilon &\rkla{\nabla u_\delta,\abs{f_\delta\trkla{u_\delta}}^4 f_\delta^\prime\trkla{u_\delta}\nabla u_\delta}+\frac1\varepsilon\norm{f_\delta\trkla{u_\delta}}_{\mathds{L}^6}^6\\
=&\,\trkla{w_\delta,\abs{f_\delta\trkla{u_\delta}}^4f_\delta\trkla{u_\delta}} +\frac1\varepsilon \rkla{u_\delta,\abs{f_\delta\trkla{u_\delta}}^4 f_\delta\trkla{u_\delta}}\\
\leq&\, \frac{1}{2\varepsilon}\norm{f_\delta\trkla{u_\delta}}_{\mathds{L}^6}^6 +C\norm{w_\delta}_{\mathds{H}^1}^6 +C\norm{u_\delta}_{\mathds{H}^1}^6\,.
\end{split}
\end{align}
Noting that the first term on the left-hand side is nonnegative, an absorption argument provides \eqref{eq:W26:1}.
The estimate \eqref{eq:W26:2} then follows from elliptic regularity theory and \eqref{eq:W26:1}.
\end{proof}
Having these tools at hand, we are now in the position to establish a weighted version of the $\mathds{H}^1$-estimate.
\begin{proof}[Proof of Lemma \ref{lem:H1weighted}]
The arguments below are formal but can be made rigorous by using a suitable approximation procedure, cf., for instance \cite{scarpa18}, \cite{DiPrimioGrasselliScarpa2024}.

After applying \Ito's formula for $H\trkla{t,u_\delta}$ defined in \eqref{eq:def:weightedEnergy}, we obtain
\begin{align}\label{eq:tmp:Ito}
\begin{split}
&\expected{H\trkla{t,u_\delta\trkla{t}}} +\gamma \expected{\int_0^t \norm{u_\delta\trkla{s}}_{\mathds{H}^2}^2 H\trkla{s,u_\delta\trkla{s}}\ds} +\expected{\int_0^t e^{-\Lambda\trkla{s}}\norm{\nabla w_\delta\trkla{s}}_{\mathds{L}^2}^2\ds}\\
&=\expected{H\trkla{0,u_\delta\trkla{0}}} +\expected{\frac\varepsilon2\sum_{i=1}^d\sum_{k\in\mathds{Z}} \int_0^te^{-\Lambda\trkla{s}}\norm{\nabla\partial_i\trkla{\lambda_k\g{k}\mysigma\trkla{u_\delta\trkla{s}}}}_{\mathds{L}^2}^2\ds}\\
&\quad+\expected{ \frac1{2\varepsilon} \sum_{i=1}^d\sum_{k\in\mathds{Z}}\int_0^t e^{-\Lambda\trkla{s}} \iD f_\delta^{\prime}\trkla{u_\delta\trkla{s}}\abs{\partial_i\trkla{\lambda_k\g{k}\mysigma\trkla{u_\delta\trkla{s}}}}^2\dx\ds}\\
&\quad-\expected{\frac1{2\varepsilon}\sum_{i=1}^d\sum_{k\in\mathds{Z}}\int_0^t e^{-\Lambda\trkla{s}}\norm{\partial_i\trkla{\lambda_k\g{k}\sigma\trkla{u_\delta\trkla{s}}}}_{\mathds{L}^2}^2\ds}\\
&=:\expected{H\trkla{0,u_\delta\trkla{0}}} + R_1 + R_2 +R_3 \,.
\end{split}
\end{align}
As $\sigma^\prime$ vanishes outside of $\tekla{-1,+1}$ and $f_\delta$ vanishes inside of $\tekla{-1,+1}$, $R_2$ vanishes. 
Noting $R_3\leq0$, it remains to control $R_1$.
Using $\mathcal{J}_\delta$, we deduce the estimate%\todo{Here, we use $\mysigma\in W^{2,\infty}$}
\begin{align}
\begin{split}
&\norm{\partial_j\partial_i\trkla{\lambda_k\g{k}\mysigma\trkla{u_\delta}}}_{\mathds{L}^2}^2= \norm{\partial_j\rkla{\lambda_k\partial_i\g{k}\mysigma\trkla{u_\delta}+\lambda_k\g{k}\sigma^\prime\trkla{u_\delta}\partial_i u_\delta}}_{\mathds{L}^2}^2\\
&\leq\tabs{\lambda_k}^2\left(\norm{\partial_j\partial_i\g{k}\mysigma\trkla{u_\delta}}_{\mathds{L}^2} +\norm{\partial_i\g{k}\mysigma^\prime\trkla{u_\delta}\partial_j u_\delta}_{\mathds{L}^2} +\norm{\partial_j\g{k}\mysigma^\prime\trkla{u_\delta}\partial_i u_\delta}_{\mathds{L}^2}\right.\\
&\qquad\left. +\norm{\g{k}\sigma^{\prime\prime}\trkla{\mathcal{J}_\delta\trkla{u_\delta}}\partial_j\mathcal{J}_\delta\trkla{u_\delta}\partial_i u_\delta }_{\mathds{L}^2} + \norm{\g{k}\mysigma^\prime\trkla{u_\delta}\partial_j\partial_i u_\delta}_{\mathds{L}^2}\right)^2\,.
\end{split}
\end{align}
Hence, due to \ref{assumption:C} and \ref{assumption:W2}, we deduce
\begin{align}
\begin{split}
R_1\leq&\, C\expected{ \int_0^t e^{-\Lambda\trkla{s}}\rkla{1+\norm{\nabla u_\delta}_{\mathds{L}^2}^2+\norm{u_\delta}_{\mathds{H}^2}^2}\ds} \\
&\,+ C\expected{\int_0^te^{-\Lambda\trkla{s}}\norm{\nabla\mathcal{J}_\delta\trkla{u_\delta}\otimes\nabla u_\delta}_{\mathds{L}^2}^2\ds}\\ 
\end{split}
\end{align}
While the first term on the right-hand side is uniformly bounded, the treatment of the second term requires special care.
Following the lines of \cite{DiPrimioGrasselliScarpa2024}, we compute using \eqref{eq:propJ}, the standard Sobolev embedding $\mathds{W}^{2,6}\hookrightarrow\mathds{W}^{1,\infty}$, and Young's inequality with $\alpha<\!\!<1$
\begin{align}
\begin{split}
\norm{\nabla\mathcal{J}_\delta\trkla{u_\delta}\otimes\nabla u_\delta}_{\mathds{L}^2}^2 =&\, \iD\abs{\nabla\mathcal{J}_\delta\trkla{u_\delta}}^2\abs{\nabla u_\delta}^2\dx \leq \iD\nabla\mathcal{J}_\delta\trkla{u_\delta}\cdot \nabla u_\delta \abs{\nabla u_\delta}^2\dx\\
=&\,-\iD\mathcal{J}_\delta\trkla{u_\delta}\Delta u_\delta\abs{\nabla u_\delta}^2\dx-\iD\mathcal{J}_\delta\trkla{u_\delta}\nabla u_\delta\cdot\nabla\abs{\nabla u_\delta}^2\dx\\
\leq&\,C \norm{u_\delta}_{\mathds{H}^2}\norm{\nabla u_\delta}_{\mathds{L}^2}\norm{\nabla u_\delta}_{\mathds{L}^\infty}\leq \alpha\norm{u_\delta}_{\mathds{W}^{2,6}}^2 +C\norm{u_\delta}_{\mathds{H}^2}^2\norm{\nabla u_\delta}_{\mathds{L}^2}^2\,.
\end{split}
\end{align}
We shall now reformulate the terms on the left-hand side of \eqref{eq:tmp:Ito}.
We start by replacing the $\mathds{H}^1$-semi-norm of $w_\delta$ by its full $\mathds{H}^1$-norm. 
For this reason, we add the expression $\frac{1}{4\widehat{C}^2}\expected{\int_0^te^{-\Lambda\trkla{s}}\trkla{w_\delta}_{\D}^2\ds}$ with $\widehat{C}$ being the constant from Lemma \ref{lem:meanw} on both sides and obtain
\begin{align}
\begin{split}
\expected{H\trkla{t,u_\delta\trkla{t}}}&\,+c_0\expected{\int_0^t e^{-\Lambda\trkla{s}}\norm{w_\delta}_{\mathds{H}^1}^2\ds} +{\frac{\gamma\varepsilon}2}\expected{\int_0^t e^{-\Lambda\trkla{s}}\norm{u_\delta}_{\mathds{H}^2}^2\norm{\nabla u_\delta}_{\mathds{L}^2}^2\ds} \\
&\,+ {\frac\gamma\varepsilon}\expected{\int_0^t e^{-\Lambda\trkla{s}}\norm{u_\delta}_{\mathds{H}^2}^2\rkla{\iD F_\delta\trkla{u_\delta}\dx-\frac12\norm{u_\delta}_{\mathds{L}^2}^2}\ds} \\
\leq&\,\expected{H\trkla{0,u_\delta\trkla{0}}} +C\expected{\int_0^t e^{-\Lambda\trkla{s}}\rkla{1+\norm{\nabla u_\delta}_{\mathds{L}^2}^2+\norm{u_\delta}_{\mathds{H}^2}^2}\ds}\\
&\, +\alpha \expected{\int_0^te^{-\Lambda\trkla{s}}\norm{u_\delta}_{\mathds{W}^{2,6}}^2\ds} +C\expected{\int_0^te^{-\Lambda\trkla{s}}\norm{u_\delta}_{\mathds{H}^2}^2\norm{\nabla u_\delta}_{\mathds{L}^2}^2\ds}\\
&\, +C \expected{\int_0^te^{-\Lambda\trkla{s}} \norm{u_\delta}_{\mathds{L}^2}^2\ds} +C
\end{split}
\end{align}
with a positive constant $c_0$ due to Lemma \ref{lem:meanw}.
Recalling Lemma \ref{lem:W26} and the estimate $\rkla{\iD F_\delta\trkla{u_\delta}\dx-\tfrac12\norm{u_\delta}_{\mathds{L}^2}^2}\geq -\mathcal{C}$, we obtain for sufficiently large $\gamma$ after rescaling
\begin{align}
\begin{split}
\esssup_{t\in\trkla{0,T}}&\,\expected{H\trkla{t,u_\delta\trkla{t}}}+\expected{\int_0^Te^{-\Lambda\trkla{s}}\norm{w_\delta}_{\mathds{H}^1}^2\ds} +\expected{\int_0^Te^{-\Lambda\trkla{s}}\norm{u_\delta}_{\mathds{W}^{2,6}}^2\ds}\\
&\,+\expected{\int_0^Te^{-\Lambda\trkla{s}}\norm{f_\delta\trkla{u_\delta}}_{\mathds{L}^6}^2\ds}+\expected{\int_0^Te^{-\Lambda\trkla{s}}\norm{u_\delta}_{\mathds{H}^2}^2\norm{\nabla u_\delta}_{\mathds{L}^2}^2\ds}\\
\leq&\,C\expected{H\trkla{0,u_\delta\trkla{0}}}+ C\expected{\int_0^T\norm{u_\delta}_{\mathds{H}^2}^2\ds} +C\leq C\,.
\end{split}
\end{align}
Here we used $e^{-\Lambda\trkla{s}}\leq 1$ and the already established $\delta$-independent $L^2\trkla{\Omega;L^2\trkla{0,T;\mathds{H}^2}}$-bound on $u_\delta$ from Lemma \ref{lem:l2estimate}.

\end{proof}

To remove the weights in the regularity result stated in Lemma \ref{lem:H1weighted}, we require an exponential version of the $\mathds{L}^2$-estimate from Lemma \ref{lem:l2estimate}:
\begin{lemma}
Let $u_\delta$ be the the family of solutions obtained in Theorem \ref{thm:weakht}, and let the assumptions \ref{assumption:W1}, \ref{assumption:W2}, \ref{assumption:C}, and \ref{assumption:I} hold true.
Then, for any $\kappa>0$ we have
\begin{multline}
\expected{\exp\rkla{\kappa\int_0^T\norm{u_\delta}_{\mathds{L}^2}^2\ds}}+\expected{\exp\rkla{\kappa\int_0^T \norm{\Delta u_\delta}_{\mathds{L}^2}^2\ds}}\\
+\expected{\exp\rkla{\kappa\int_0^T\rkla{f_\delta^\prime\trkla{u_\delta}\nabla u_\delta,\nabla u_\delta}\ds}}\leq C
\end{multline}
with a positive constant $C$ depending on $\kappa$ but not on $\delta$.
\end{lemma}
\begin{proof}

Applying \Ito's formula (\cite[Theorem 4.2.5]{LiuRoeckner}) to $H\trkla{t,u_\delta\trkla{t}}:=\norm{u_\delta}_{\mathds{L}^2}^2$, we obtain
\begin{align}
\begin{split}\label{eq:tmp:e:ito}
\norm{u_\delta\trkla{t}}_{\mathds{L}^2}^2 &\,+ 2\varepsilon\int_0^t\norm{\Delta u_\delta}_{\mathds{L}^2}^2\ds +\frac2\varepsilon\int_0^t\rkla{f_\delta^\prime\trkla{u_\delta}\nabla u_\delta,\nabla u_\delta}\ds\\
=&\,\norm{u_\delta\trkla{0}}_{\mathds{L}^2}^2+2\sum_{i=1}^d\sum_{k\in\mathds{Z}}\int_0^t\rkla{u_\delta,\partial_i\rkla{\mysigma\trkla{u_\delta}\lambda_k\g{k}}}\dbeta{k}^i+\frac2\varepsilon\int_0^t\norm{\nabla u_\delta}_{\mathds{L}^2}^2\ds\\
&\,+\sum_{i=1}^d\sum_{k\in\mathds{Z}}\int_0^t\norm{\partial_i\rkla{\mysigma\trkla{u_\delta}\lambda_k\g{k}}}_{\mathds{L}^2}^2\ds
\end{split}
\end{align}
Due to assumption \ref{assumption:W2}, an application of Young's inequality provides
\begin{multline}
\frac2\varepsilon\int_0^t\norm{\nabla u_\delta}_{\mathds{L}^2}^2\ds+\sum_{i=1}^d\sum_{k\in\mathds{Z}}\int_0^t\norm{\partial_i\rkla{\mysigma\trkla{u_\delta}\lambda_k\g{k}}}_{\mathds{L}^2}^2\ds \\
\leq \frac\varepsilon2\int_0^t\norm{\Delta u_\delta}_{\mathds{L}^2}^2\ds+C\varepsilon^{-2}\int_0^t\norm{u_\delta}_{\mathds{L}^2}^2\ds +Ct\,.
\end{multline}
Similarly, we deduce the following estimates for the quadratic variation of the martingale $\sum_{i=1}^d\sum_{k\in\mathds{Z}}\int_0^t\trkla{u_\delta,\partial_i\trkla{\mysigma\trkla{u_\delta}\lambda_k\g{k}}}\dbeta{k}^i$:
\begin{multline}
\sum_{i=1}^d\sum_{d\in\mathds{Z}}\int_0^t\rkla{u_\delta,\partial_i\trkla{\mysigma\trkla{u_\delta}\lambda_k\g{k}}}^2\ds =\sum_{i=1}^d\sum_{d\in\mathds{Z}}\int_0^t\rkla{\partial_iu_\delta,\mysigma\trkla{u_\delta}\lambda_k\g{k}}^2\ds\\
\leq C\int_0^t\norm{\nabla u_\delta}_{\mathds{L}^2}^2\ds\leq C_\alpha\int_0^t\norm{u_\delta}_{\mathds{L}^2}^2\ds + \alpha\int_0^t\norm{\Delta u_\delta}_{\mathds{L}^2}^2\ds\,.
\end{multline}
Hence, adding and subtracting $3\kappa\sum_{i=1}^d\sum_{k\in\mathds{Z}}\int_0^t \rkla{u_\delta,\partial_i\trkla{\mysigma\trkla{u_\delta}\lambda_k\g{k}} }^2\ds$ with arbitrary but fixed $\kappa>0$ to \eqref{eq:tmp:e:ito} and choosing $\alpha=\varepsilon/\trkla{6\kappa}$, we obtain
\begin{align}
\begin{split}
&\norm{u_\delta\trkla{t}}_{\mathds{L}^2}^2+\varepsilon\int_0^t\norm{\Delta u_\delta}_{\mathds{L}^2}^2\ds+\frac2\varepsilon\int_0^t\trkla{f_\delta^\prime\trkla{u_\delta}\nabla u_\delta,\nabla u_\delta}\ds\\
&\quad\leq\norm{u_\delta\trkla{0}}_{\mathds{L}^2}^2 +C\varepsilon^{-2}\trkla{1+\kappa}\int_0^t\norm{u_\delta}_{\mathds{L}^2}^2\ds +C t\\
&\qquad+2\rkla{\sum_{i=1}^d\sum_{k\in\mathds{Z}}\int_0^t\trkla{u_\delta,\partial_i\trkla{\mysigma\trkla{u_\delta}\lambda_k\g{k}}}\dbeta{k}^i -\frac{3\kappa}2 \sum_{i=1}^d\sum_{k\in\mathds{Z}}\int_0^t\trkla{u_\delta,\partial_i\trkla{\mysigma\trkla{u_\delta}\lambda_k\g{k}}}^2\ds}\,.
\end{split}
\end{align}
Multiplying the above inequality by $\kappa$ and taking the exponentials provides
\begin{align}\label{eq:tmp:exponentialestimate}
\begin{split}
&\exp\rkla{\kappa\norm{u_\delta}_{\mathds{L}^2}^2} +\exp\rkla{\varepsilon\kappa\int_0^t\norm{\Delta u_\delta}_{\mathds{L}^2}^2\ds} +\exp\rkla{\frac{2\kappa}\varepsilon \int_0^t\rkla{f_\delta^\prime\trkla{u_\delta}\nabla u_\delta,\nabla u_\delta}\ds}\\
&~\leq C\left(\exp\rkla{3\kappa\norm{u_\delta\trkla{0}}_{\mathds{L}^2}^2} +\exp\rkla{3\kappa\varepsilon^{-2}\trkla{1+\kappa} \int_0^t\norm{u_\delta}_{\mathds{L}^2}^2\ds +3\kappa t} \right.\\
&\quad\left.+\exp\rkla{\sum_{i=1}^d\sum_{k\in\mathds{Z}}\rkla{3\kappa\int_0^t\rkla{u_\delta,\partial_i\trkla{\mysigma\trkla{u_\delta}\lambda_k\g{k}}}\dbeta{k}^i -\frac{9\kappa^2}{2}\int_0^t\trkla{u_\delta,\partial_i\trkla{\mysigma\trkla{u_\delta}\lambda_k\g{k}}}^2\ds}}       \right)\,.
\end{split}
\end{align}
As the last term on the right-hand side in \eqref{eq:tmp:exponentialestimate} is a real supermartingale, its expectation is bounded by the expectation of its initial state.
Hence, taking the expectation and the suprema in time of \eqref{eq:tmp:exponentialestimate} yields
\begin{align}
\begin{split}
\sup_{t\in\tekla{0,T}}&\expected{\exp\rkla{\kappa\norm{u_\delta\trkla{t}}_{\mathds{L}^2}^2}} +\expected{\exp\rkla{\varepsilon\kappa\int_0^T\norm{\Delta u_\delta}_{\mathds{L}^2}^2\ds}} \\
& +\expected{\exp\rkla{\frac{2\kappa}\varepsilon \int_0^T\rkla{f_\delta^\prime\trkla{u_\delta}\nabla u_\delta, u_\delta}\ds}}\\
\leq&\,C+C\exp\rkla{\kappa T}+C\expected{\exp\rkla{3\kappa\norm{u_\delta\trkla{0}}_{\mathds{L}^2}^2}} \\
&\,+ C\expected{\exp\rkla{3\kappa\varepsilon^{-2}\trkla{1+\kappa}\int_0^t\norm{u_\delta}_{\mathds{L}^2}^2\ds}}\,.
\end{split}
\end{align}
Following the arguments presented in \cite[Section 4.3.6]{DiPrimioGrasselliScarpa2024} concludes the proof.
\end{proof}
Having control this exponential version of the $\mathds{L}^2$-estimate at hand, an application of Hölder's inequality provides 
\begin{align}
w_\delta\in&\,L^s\trkla{\Omega;L^2\trkla{0,T;\mathds{H}^1}}\,,\\
u_\delta\in&\,L^s\trkla{\Omega;L^2\trkla{0,T;\mathds{W}^{2,6}}}\,,\\
f_\delta\trkla{u_\delta}\in&\,L^s\trkla{\Omega;L^2\trkla{0,T;\mathds{L}^6}}\label{eq:bound:fdelta}
\end{align}
for any $s\in[1,2)$ uniformly in $\delta$.
Noting the above bounds, Lemma~\ref{lem:H1} follows along the lines of \cite[Section 4.3.8]{DiPrimioGrasselliScarpa2024}.

\subsection{Further estimates}\label{sec:further}
In the next lemma, we establish Hölder continuity of the regularised solution in time:
\begin{lemma}
Let $u_\delta$ be the the family of solutions obtained in Theorem \ref{thm:weakht}, and let the assumptions \ref{assumption:W1}, \ref{assumption:W2}, \ref{assumption:C}, and \ref{assumption:I} hold true.
Then, 
\begin{align*}
S_\delta\trkla{t}:=\sum_{i=1}^d\sum_{k\in\mathds{Z}} \int_0^t\partial_i\rkla{\mysigma\trkla{u_\delta}\lambda_k\g{k}}\dbeta{k}^i
\end{align*}
is uniformly bounded in $L^p\trkla{\Omega;C^{0,\alpha-1/p}\trkla{\tekla{0,T};\mathds{H}^{-1}_\per}}$ for all $p>2$ and $\alpha<1/2$.
\end{lemma}
\begin{proof}
We consider the operator $\Phi\trkla{u_\delta}$ which is defined via
\begin{align}
\Phi\trkla{u_\delta} \bs{g}:=\sum_{i=1}^d\sum_{k\in\mathds{Z}}\partial_i\trkla{\mysigma\trkla{u_\delta} \trkla{\bs{g}_i,\g{k}}\g{k}}
\end{align}
and show that it is uniformly bounded in $L^p\trkla{\Omega;L^p\trkla{0,T;L_2\trkla{\trkla{\mathcal{Q}\mathds{L}^2}^d;\mathds{H}^{-1}_\per}}}$.
Indeed, we have
\begin{multline}
\norm{\Phi\trkla{u_\delta}}_{L_2\trkla{\trkla{\mathcal{Q}\mathds{L}^2}^d;\mathds{H}^{-1}_\per}}^2=\sum_{i=1}^d\sum_{k\in\mathds{Z}}\rkla{\sup_{0\neq\psi\in\mathds{H}^1_\per}\frac{\abs{\iD \mysigma\trkla{u_\delta}\lambda_k\g{k}\partial_i\psi\dx}}{\norm{\psi}_{\mathds{H}^1}}}^2\\
\leq \sum_{i=1}^d\sum_{k\in\mathds{Z}}\abs{\lambda_k}^2\norm{\g{k}}_{\mathds{L}^2}^2\leq C\,.
\end{multline}
Thus, $\Phi\trkla{u_\delta}\in L^p\trkla{\Omega;L^p\trkla{0,T;L_2\trkla{\trkla{\mathcal{Q}\mathds{L}^2}^d;\mathds{H}^{-1}_\per}}}$ for all $p\geq2$. 
By \cite[Lemma 2.1]{Flandoli1995}, we have that $S_\delta$ is bounded in $L^p\trkla{\Omega;W^{\alpha,p}\trkla{0,T;\mathds{H}^{-1}_\per}}$ for any $\alpha<1/2$.
A standard embedding theorem concludes the proof.
\end{proof}
\begin{lemma}
Let $u_\delta$ be the the family of solutions obtained in Theorem \ref{thm:weakht}, and let the assumptions \ref{assumption:W1}, \ref{assumption:W2}, \ref{assumption:C}, and \ref{assumption:I} hold true.
Then, the estimate
\begin{align}
\norm{u_\delta}_{L^s\trkla{\Omega;C^{0,\widetilde{\alpha}}\trkla{\tekla{0,T};\mathds{H}^{-1}_\per}}}\leq C
\end{align}
holds true for all $s\in[1,2)$ with $\widetilde{\alpha}<1/2$.
\end{lemma}
\begin{proof}
For $\psi\in \mathds{H}^1_\per$ and $0\leq t_1<t_2\leq T$, we have
\begin{multline}
\iD\rkla{u_\delta\trkla{t_2}-u_\delta\trkla{t_1}}\psi\dx+\int_{t_1}^{t_2}\iD\nabla w_\delta\cdot\nabla\psi\dx\dt\\
=\sum_{i=1}^d\sum_{k\in\mathds{Z}}\int_{t_1}^{t_2}\iD\partial_i\rkla{\mysigma\trkla{u_\delta}\lambda_k\g{k}}\psi\dx\!\dbeta{k}^i\,.
\end{multline}
Thus,
\begin{align}
\begin{split}
\abs{\iD\rkla{u_\delta\trkla{t_2}-u_\delta\trkla{t_1}}\psi\dx}\leq&\,\trkla{t_2-t_1}^{1/2}\norm{\nabla w_\delta}_{L^2\trkla{t_1,t_2;\mathds{L}^2}}\norm{\psi}_{\mathds{H}^1}\\
&+ \trkla{t_2-t_1}^{\widetilde{\alpha}}\frac{\norm{S_\delta\trkla{t_2}-S_\delta\trkla{t_1}}_{-1}}{\trkla{t_2-t_1}^{\widetilde{\alpha}}}\norm{\psi}_{\mathds{H}^1}\,.
\end{split}
\end{align}
\end{proof}
Owing to the $\delta$-independent $\mathds{H}^1$-estimate from Lemma \ref{lem:H1}, we show that the limit of $u_\delta$ attains values in the interval $\tekla{-1,1}$ almost surely almost everywhere in $\trkla{0,T}\times\D$.
Using the notation
\begin{align}
\trkla{u}^+:=\max\tgkla{0,u}&&\text{and}&&\trkla{u}^-:=\min\tgkla{0,u}\,,
\end{align}
we deduce the following bounds:
\begin{lemma}\label{lem:bound1}
Let $u_\delta$ be the the family of solutions obtained in Theorem \ref{thm:weakht}, and let the assumptions \ref{assumption:W1}, \ref{assumption:W2}, \ref{assumption:C}, and \ref{assumption:I} hold true.
Then
\begin{subequations}\label{eq:pm1}
\begin{align}
\norm{\trkla{u_\delta-1}^+}_{L^{s}\trkla{\Omega;L^\infty\trkla{0,T;\mathds{L}^2}}}+\norm{\trkla{u_\delta+1}^-}_{L^{s}\trkla{\Omega;L^\infty\trkla{0,T;\mathds{L}^2}}}\leq&\, C\sqrt{\delta}\,,\label{eq:pm1:linfty}\\
\norm{\trkla{u_\delta-1}^+}_{L^s\trkla{\Omega;L^2\trkla{0,T;\mathds{L}^6}}}+\norm{\trkla{u_\delta+1}^-}_{L^s\trkla{\Omega;L^2\trkla{0,T;\mathds{L}^6}}}\leq&\, C\delta\,\label{eq:pm1:l2}
\end{align}
\end{subequations}
holds true for $s\in[1,2)$.
\end{lemma}
\begin{proof}
We start by establishing the first part of the estimate \eqref{eq:pm1:linfty}:
\begin{align}
\begin{split}
\norm{\trkla{u_\delta-1}^+}_{\mathds{L}^2}^2=&\,\int_{\tgkla{1<u_\delta<1+\delta}}\trkla{u_\delta-1}^2\dx+\int_{\tgkla{u_\delta\geq1+\delta}}\trkla{u_\delta-1}^2\dx\\
\leq&\,\int_{\D}\delta^2\dx+\int_{\tgkla{u_\delta\geq1+\delta}}\trkla{u_\delta-\trkla{1+\tfrac\delta2}+\tfrac\delta2}^2\dx\\
\leq&\,C\delta^2+2\int_{\tgkla{u_\delta\geq1+\delta}}\trkla{u_\delta-\trkla{1+\tfrac\delta2}}^2\dx \leq C\delta^2+C\delta\int_\D F_\delta\trkla{u_\delta}\dx\,.
\end{split}
\end{align}
The bound on $F_\delta\trkla{u_\delta}$ derived from Lemma \ref{lem:H1} concludes the argument. 
The second estimate in \eqref{eq:pm1:linfty} can be established analogously.
To establish \eqref{eq:pm1:l2}, we compute
\begin{align}
\begin{split}
\norm{\frac{\trkla{u_\delta-1}^+}{\delta}}_{\mathds{L}^6}^6 =&\, \int_{\tgkla{1<u_\delta <1 +\delta}}\frac{\trkla{u_\delta-1}^6}{\delta^6}\dx +\int_{\tgkla{u_\delta\geq1+\delta}}\frac{\trkla{u_\delta-1}^6}{\delta^6}\dx\\
\leq&\, C\int_{\tgkla{1<u_\delta <1 +\delta}}\abs{f_\delta\trkla{u_\delta}}^3\dx+ \int_{\tgkla{u_\delta\geq1+\delta}}\rkla{f_\delta\trkla{u_\delta}+\tfrac12}^6\dx\\
\leq&\,C\norm{f_\delta\trkla{u_\delta}}_{\mathds{L}^6}^6 +C\,.
\end{split}
\end{align}
Arguing in a similar manner for $\trkla{u_\delta+1}^-/\delta$ and recalling \eqref{eq:bound:fdelta} provides \eqref{eq:pm1:l2}.
\end{proof}

\begin{lemma}
Let $u_\delta$ be the the family of solutions obtained in Theorem \ref{thm:weakht}, and let the assumptions \ref{assumption:W1}, \ref{assumption:W2}, \ref{assumption:C}, and \ref{assumption:I} hold true.
Then, there exists a constant $C>0$, which is independent of $\delta$, such that
\begin{align}
\norm{\partial_t\rkla{u_\delta-S_\delta}}_{L^s\trkla{\Omega;L^2\trkla{0,T;\mathds{H}^{-1}_\per}}}\leq C\,.
\end{align}
\end{lemma}
\begin{proof}
Noting the improved regularity results established in Lemma \ref{lem:H1}, we deduce that \eqref{eq:problemdelta:u} also holds true in $\mathds{H}^{-1}_\per$. Hence, we deduce for $v\in\mathds{H}^1_\per$
\begin{align}
\abs{\skla{\partial_t\rkla{u_\delta- S_\delta},v}}=\abs{\rkla{\nabla w_\delta,\nabla v}}\leq \norm{\nabla w_\delta}_{\mathds{L}^2}\norm{\nabla v}_{\mathds{L}^2}\,,
\end{align}
which provides the result.
\end{proof}

\section{Limit $\delta\searrow0$}\label{sec:limitdelta}
In this section, we pass to the limit $\delta\searrow0$.
As in Section \ref{sec:limitht}, we start by deducing weak(-star) convergence results from the $\delta$-independent bounds established in Section \ref{sec:H1}. 
For the reader's convenience, we collect the most important $\delta$-independent results obtained so far.
In particular, we have
\begin{subequations}\label{eq:bounddelta}
\begin{multline}\label{eq:bounddelta:old}
\norm{u_\delta}_{L^p\trkla{\Omega;L^\infty\trkla{0,T;\mathds{H}^{-1}_\per}}}+\norm{u_\delta}_{L^p\trkla{\Omega;L^4\trkla{0,T;\mathds{L}^2}}} + \norm{u_\delta}_{L^3\trkla{\Omega;L^\infty\trkla{0,T;\mathds{L}^2}}}+\norm{u_\delta}_{L^p\trkla{\Omega;L^2\trkla{0,T;\mathds{H}^1}}}\\
+\norm{\Delta u_\delta}_{L^3\trkla{\Omega;L^2\trkla{0,T;\mathds{L}^2}}} +\norm{\mysigma\trkla{u_\delta}}_{L^\infty\trkla{\Omega;L^\infty\trkla{0,T;\mathds{L^\infty}}}} + \norm{\mysigma\trkla{u_\delta}}_{L^p\trkla{\Omega;L^2\trkla{0,T;\mathds{H}^1}}}\leq C\,,
\end{multline}
\begin{multline}\label{eq:bounddelta:new}
\norm{u_\delta}_{L^s\trkla{\Omega;L^\infty\trkla{0,T;\mathds{H}^{1}}}}+\norm{u_\delta}_{L^s\trkla{\Omega;L^2\trkla{0,T;\mathds{W}^{2,6}}}} +\norm{u_\delta}_{L^s\trkla{\Omega;C^{0,\widetilde{\alpha}}\trkla{\tekla{0,T};\mathds{H}^{-1}_\per}}}\\
+ \norm{f_\delta\trkla{u_\delta}}_{L^s\trkla{\Omega;L^2\trkla{0,T;\mathds{L}^6}}} +\norm{w_\delta}_{L^s\trkla{\Omega;L^2\trkla{0,T;\mathds{H}^1}}} +\norm{\partial_t\trkla{u_\delta-S_\delta}}_{L^s\trkla{\Omega;L^2\trkla{0,T;\mathds{H}^{-1}_\per}}} \\
+\delta^{-1/2}\norm{\trkla{u_\delta-1}^+}_{L^s\trkla{\Omega;L^\infty\trkla{0,T;\mathds{L}^2}}}+\delta^{-1/2}\norm{\trkla{u_\delta+1}^-}_{L^s\trkla{\Omega;L^\infty\trkla{0,T;\mathds{L}^2}}}\\
+\delta^{-1}\norm{\trkla{u_\delta-1}^+}_{L^s\trkla{\Omega;L^2\trkla{0,T;\mathds{L}^6}}}+\delta^{-1}\norm{\trkla{u_\delta+1}^-}_{L^s\trkla{\Omega;L^2\trkla{0,T;\mathds{L}^6}}}\leq C
\end{multline}
for all $p\in[1,\infty)$, $s\in[1,2)$, and $\widetilde{\alpha}<1/2$.
Here, the uniform bounds in \eqref{eq:bounddelta:old} are a result of the $\delta$-independent bounds collected in \eqref{eq:boundhtau:independent} and the estimates collected in \eqref{eq:bounddelta:new} were established in the Section \ref{sec:H1}.
To identify the limit of $f_\delta\trkla{u_\delta}$, we shall split the expression into a positive and a negative part.
Denoting the characteristic function of the set $\tgkla{u_\delta>1}\subset\D$ by $\chi_\delta^+$ and the one of set $\tgkla{u_\delta<-1}\subset\D$ by $\chi_\delta^-$, we obtain the decomposition $f_\delta\trkla{u_\delta}=f_\delta\trkla{u_\delta}\chi_\delta^++f_\delta\trkla{u_\delta}\chi_\delta^-$.
Therefore, we have
\begin{align}
\norm{f_\delta\trkla{u_\delta}\chi_\delta^+}_{L^s\trkla{\Omega;L^2\trkla{0,T;\mathds{L}^6}}}+\norm{f_\delta\trkla{u_\delta}\chi_\delta^-}_{L^s\trkla{\Omega;L^2\trkla{0,T;\mathds{L}^6}}} \leq C
\end{align}
for all $s\in[1,2)$.
\end{subequations}
\begin{theorem}\label{thm:weakdelta}
Let $u_\delta$ and $w_\delta$ be the the family of solutions obtained in Theorem \ref{thm:weakht}, and let the assumptions \ref{assumption:W1}, \ref{assumption:W2}, \ref{assumption:C}, and \ref{assumption:I} hold true.
Then, there exists
\begin{align*}
\begin{split}
u\in&\,L^p_{\operatorname{weak-(*)}}\trkla{\Omega;L^\infty\trkla{0,T;\mathds{H}^{-1}_\per}}\cap L^3_{\operatorname{weak-(*)}}\trkla{\Omega;L^\infty\trkla{0,T;\mathds{L}^{2}}}\cap L^p\trkla{\Omega;L^4\trkla{0,T;\mathds{L}^2}}\\
&\,\cap L^p\trkla{\Omega;L^2\trkla{0,T;\mathds{H}^1}}\cap L^3\trkla{\Omega;L^2\trkla{0,T;\mathds{H}^2}}\cap L^s_{\operatorname{weak-(*)}}\trkla{\Omega;L^\infty\trkla{0,T;\mathds{H}^{1}}}\\
&\,\cap L^s\trkla{\Omega;L^2\trkla{0,T;\mathds{W}^{2,6}}}\cap L^s\trkla{\Omega;C^{0,\widetilde{\alpha}}\trkla{\tekla{0,T};\mathds{H}^{-1}_\per}}\,,
\end{split}\\
w\in&\,L^s\trkla{\Omega;L^2\trkla{0,T;\mathds{H}^1}}\,,\\
\widetilde{p}\in&\,L^s\trkla{\Omega;L^2\trkla{0,T;\mathds{L}^6}}\,,\\
\widetilde{\mysigma}\in&\,L^p\trkla{\Omega;L^2\trkla{0,T;\mathds{H}^1}}
\end{align*}
for all $p\in[1,\infty)$, $s\in[1,2)$, and $\widetilde{\alpha}<1/2$ satisfying
\begin{subequations}\label{eq:limitdelta}
\begin{align}\label{eq:limitdelta:u}
\mathrm{d}u-\Delta w\dt=\div\gkla{\widetilde{\mysigma}\dW}&&&\text{in~}\mathds{H}^{-1}_\per\,,
\end{align}
$\mathds{P}$-almost surely and for all $t\in\trkla{0,T}$.
Here, $w$ satisfies
\begin{align}\label{eq:limitdelta:w}
w=-\varepsilon\Delta u+\varepsilon^{-1}\widetilde{p}-\varepsilon^{-1}u&&&\text{in~}\mathds{L}^2\,
\end{align}
$\mathds{P}$-almost surely and for almost all $t\in\trkla{0,T}$.
\end{subequations}
Further, we have $u\in\tekla{-1,+1}$ $\mathds{P}$-almost surely almost everywhere in $\trkla{0,T}\times\D$.
In addition, for $\delta\searrow0$ we have for (not relabeled) subsequences
\begin{subequations}\label{eq:weakconv:delta}
\begin{align}
u_\delta \stackrel{*}{\rightharpoonup}&\,u&&\text{in~}L^p_{\operatorname{weak-(*)}}\trkla{\Omega;L^\infty\trkla{0,T;\mathds{H}^{-1}_\per}}\cap L^3_{\operatorname{weak-(*)}}\trkla{\Omega;L^\infty\trkla{0,T;\mathds{L}^{2}}}\\
&&&\nonumber\qquad\cap L^s_{\operatorname{weak-(*)}}\trkla{\Omega;L^\infty\trkla{0,T;\mathds{H}^{1}}}\,,\\
u_\delta\rightharpoonup&\,u&&\text{in~} L^p\trkla{\Omega;L^4\trkla{0,T;\mathds{L}^2}}\cap L^p\trkla{\Omega;L^2\trkla{0,T;\mathds{H}^1}}\\
&&&\nonumber\qquad\cap L^3\trkla{\Omega;L^2\trkla{0,T;\mathds{H}^2}}\cap L^s\trkla{\Omega;L^2\trkla{0,T;\mathds{W}^{2,6}}}\,,\\
w_\delta\rightharpoonup&\, w&&\text{in~}L^s\trkla{\Omega;L^2\trkla{0,T;\mathds{H}^1}}\,,\\
f_\delta\trkla{u_\delta}\rightharpoonup&\,\widetilde{p}&&\text{in~}L^s\trkla{\Omega;L^2\trkla{0,T;\mathds{L}^6}}\,,\\
\mysigma\trkla{u_\delta}\rightharpoonup&\,\widetilde{\mysigma}&&\text{in~}L^p\trkla{\Omega;L^2\trkla{0,T;\mathds{H}^1}}\,,\\
f_\delta\trkla{u_\delta}\chi_\delta^-\rightharpoonup&\,\widetilde{p}_-&&\text{in~}L^s\trkla{\Omega;L^2\trkla{0,T;\mathds{L}^6}}\,,\quad\text{with~}\widetilde{p}_-\leq0\,,\\
f_\delta\trkla{u_\delta}\chi_\delta^+\rightharpoonup&\,\widetilde{p}_+&&\text{in~}L^s\trkla{\Omega;L^2\trkla{0,T;\mathds{L}^6}}\,,\quad\text{with~}\widetilde{p}_+\geq0\,,
\end{align}
\end{subequations}
and
\begin{align}\label{eq:weakconv:delta:noise}
u_\delta-S_\delta\rightharpoonup u-\sum_{i=1}^d\sum_{k\in\mathds{Z}}\int_0^{\cdot}\partial_i\rkla{\widetilde{\mysigma}\lambda_k\g{k}}\dbeta{k}^i&&\text{in~}L^s\trkla{\Omega;H^1\trkla{0,T;\mathds{H}^{-1}_\per}}\,.
\end{align}
In addition, we have $\widetilde{p}=\widetilde{p}_-+\widetilde{p}_+$ $\mathds{P}$-almost surely almost everywhere in $\trkla{0,T}\times\D$.
\end{theorem}
\begin{proof}
The convergence results collected in \eqref{eq:weakconv:delta} are a straightforward consequence of the uniform bounds stated in \eqref{eq:bounddelta}.
To establish \eqref{eq:weakconv:delta:noise}, we argue similarly to Theorem \ref{thm:weakht}:
From the uniform bound on $\partial_t\trkla{u_\delta-S_\delta}$ in $L^s\trkla{\Omega;L^2\trkla{0,T;\mathds{H^{-1}_\per}}}$, we deduce the weak convergence towards some limit $\zeta$. Using the linearity and continuity of the stochastic integral, we can pass to the limit in $L^s\trkla{\Omega;L^2\trkla{0,T;\mathds{H}^{-1}_\per}}$ and identify $\zeta$.
\end{proof}
In order to identify the limit process $\widetilde{\mysigma}$ and to verify that $\rkla{\widetilde{p}-u}\in\partial\Psi\trkla{u}$, we establish strong convergence of $u_\delta$ by again applying a monotonicity argument:

\begin{lemma}
Let the assumption of Theorem \ref{thm:weakdelta} hold true. Then, 
\begin{align*}
u_\delta\rightarrow&\, u&&\text{in~}L^{2q}\trkla{\Omega;L^2\trkla{0,T;\mathds{L}^2}}
\end{align*}
for $q<2$.
Furthermore, $\widetilde{p}-u\in\partial\Psi$ and $\widetilde{\mysigma}=\mysigma\trkla{u}$ hold true.
\end{lemma}

\begin{proof}
Applying \Ito's formula to $F\trkla{s,u_\delta\trkla{s}}:=e^{-\hat{c}s}\norm{u_\delta\trkla{s}-\mathfrak{m}}_{-1}^2$ with a constant $\hat{c}$ specified later (note that due to Lemma \ref{lem:H1} \eqref{eq:problemdelta:u} holds in $\mathds{H}^{-1}_\per$) and recalling \eqref{eq:strongconv:hminus1}, we obtain
\begin{align}\label{eq:tmp:itoud}
\begin{split}
\expected{e^{-\hat{c}t}\norm{u_\delta\trkla{t}-\mathfrak{m}}_{-1}^2}=&\,\expected{\norm{u_\delta\trkla{0}-\mathfrak{m}}_{-1}^2}-\expected{\int_0^t \hat{c}e^{-\hat{c}s}\norm{u_\delta\trkla{s}-\mathfrak{m}}_{-1}^2\ds}\\
&-\,2\expected{\int_0^te^{-\hat{c}s}\rkla{\varepsilon\norm{\nabla u_\delta\trkla{s}}_{\mathds{L}^2}^2+\varepsilon^{-1} \rkla{f_\delta\trkla{u_\delta\trkla{s}},u_\delta\trkla{s}-\mathfrak{m}}}\ds}\\
&\,+2\varepsilon^{-1}\expected{\int_0^te^{-\hat{c}s}\norm{u_\delta\trkla{s}-\mathfrak{m}}_{\mathds{L}^2}^2\ds}\\
&\,+\expected{\int_0^t e^{-\hat{c}s}\sum_{k\in\mathds{Z}}\norm{\mysigma\trkla{u_\delta}\lambda_k\g{k}-\trkla{\mysigma\trkla{u_\delta}\lambda_k\g{k}}_{\D}}_{\mathds{L}^2}^2\ds}\,.
\end{split}
\end{align}
Similarly, we deduce for $F\trkla{s,u\trkla{s}}$
\begin{align}\label{eq:tmp:itou}
\begin{split}
\expected{e^{-\hat{c}t}\norm{u\trkla{t}-\mathfrak{m}}_{-1}^2}=&\,\expected{\norm{u\trkla{0}-\mathfrak{m}}_{-1}^2}-\expected{\int_0^t \hat{c}e^{-\hat{c}s}\norm{u\trkla{s}-\mathfrak{m}}_{-1}^2\ds}\\
&-\,2\expected{\int_0^te^{-\hat{c}s}\rkla{\varepsilon\norm{\nabla u\trkla{s}}_{\mathds{L}^2}^2+\varepsilon^{-1}\trkla{\widetilde{p}\trkla{s},u\trkla{s}-\mathfrak{m}}}\ds}\\
&\,+2\varepsilon^{-1}\expected{\int_0^te^{-\hat{c}s}\norm{u\trkla{s}-\mathfrak{m}}_{\mathds{L}^2}^2\ds}\\
&\,+\expected{\int_0^t e^{-\hat{c}s}\sum_{k\in\mathds{Z}}\norm{\widetilde{\sigma}\lambda_k\g{k}-\trkla{\widetilde{\sigma}\lambda_k\g{k}}_{\D}}_{\mathds{L}^2}^2\ds}\,.
\end{split}
\end{align}
As in the proof of Lemma \ref{lem:strongConvergenceht}, we use
\begin{align}
\norm{u_\delta -\mathfrak{m}}_{-1}^2=&\,\norm{u_\delta-u}_{-1}^2+2\rkla{u_\delta-\mathfrak{m},u-\mathfrak{m}}_{-1}-\norm{u-\mathfrak{m}}_{-1}^2\,,\\
\norm{\nabla u_\delta}_{\mathds{L}^2}^2=&\, \norm{\nabla u_\delta-\nabla u}_{\mathds{L}^2}^2+2\rkla{\nabla u_\delta,\nabla u}-\norm{\nabla u}_{\mathds{L}^2}^2\,,\\
\begin{split}
\norm{u_\delta-\mathfrak{m}}_{\mathds{L}^2}^2\leq &\,\widetilde{\alpha}\norm{\nabla u_\delta-\nabla u}_{\mathds{L}^2}^2+C_{\widetilde{\alpha}}\norm{u_\delta-u}_{-1}^2+2\rkla{u_\delta-\mathfrak{m},u-\mathfrak{m}}-\norm{u-\mathfrak{m}}_{\mathds{L}^2}^2\,,
\end{split}
\end{align}
and
\begin{align}
\begin{split}
&\norm{\mysigma\trkla{u_\delta}\lambda_k\g{k}-\trkla{\mysigma\trkla{u_\delta}\lambda_k\g{k}}_{\D}}_{\mathds{L}^2}^2\\
&\quad=\norm{\ekla{\mysigma\trkla{u_\delta}\lambda_k\g{k}-\trkla{\mysigma\trkla{u_\delta}\lambda_k\g{k}}_{\D}}-\ekla{\mysigma\trkla{u}\lambda_k\g{k}-\trkla{\mysigma\trkla{u}\lambda_k\g{k}}_{\D}}}_{\mathds{L}^2}^2\\
&\qquad+2\rkla{\mysigma\trkla{u_\delta}\lambda_k\g{k}-\trkla{\mysigma\trkla{u_\delta}\lambda_k\g{k}}_{\D},\mysigma\trkla{u}\lambda_k\g{k}-\trkla{\mysigma\trkla{u}\lambda_k\g{k}}_{\D}}\\
&\qquad-\norm{\ekla{\mysigma\trkla{u}\lambda_k\g{k}-\trkla{\mysigma\trkla{u}\lambda_k\g{k}}_{\D}}-\ekla{\widetilde{\sigma}\lambda_k\g{k}-\trkla{\widetilde{\sigma}\lambda_k\g{k}}_{\D}}}_{\mathds{L}^2}^2\\
&\qquad-2\rkla{\mysigma\trkla{u}\lambda_k\g{k}-\trkla{\mysigma\trkla{u}\lambda_k\g{k}}_{\D},\widetilde{\sigma}\lambda_k\g{k}-\trkla{\widetilde{\sigma}\lambda_k\g{k}}_{\D}}+\norm{\widetilde{\sigma}\lambda_k\g{k}-\trkla{\widetilde{\sigma}\lambda_k\g{k}}_{\D}}_{\mathds{L}^2}^2\,.
\end{split}
\end{align}
Using the Lipschitz continuity of $\mysigma$ and Young's inequality, we estimate for the first term on the right-hand side
\begin{align}
\begin{split}
&\norm{\ekla{\mysigma\trkla{u_\delta}\lambda_k\g{k}-\trkla{\mysigma\trkla{u_\delta}\lambda_k\g{k}}_{\D}}-\ekla{\mysigma\trkla{u}\lambda_k\g{k}-\trkla{\mysigma\trkla{u}\lambda_k\g{k}}_{\D}}}_{\mathds{L}^2}^2\\
&\quad\leq C\norm{\rkla{\mysigma\trkla{u_\delta}-\mysigma\trkla{u}}\lambda_k\g{k}}_{\mathds{L}^2}^2\leq C\norm{u_\delta-u}_{\mathds{L}^2}^2\abs{\lambda}^2\norm{\g{k}}_{\mathds{L}^\infty}^2\\
&\quad\leq \rkla{\alpha\norm{\nabla u_\delta-\nabla u}_{\mathds{L}^2}^2+C_\alpha\norm{u_\delta -u}_{-1}^2}\abs{\lambda_k}^2\norm{\g{k}}_{\mathds{L}^\infty}^2\,.
\end{split}
\end{align}
Further, noting that $\abs{u}\leq1$ $\mathds{P}$-almost surely almost everywhere in $\trkla{0,T}\times\D$ due to Lemma \ref{lem:bound1}, we deduce
\begin{multline}
\rkla{f_\delta\trkla{u_\delta},u_\delta-\mathfrak{m}}=\rkla{f_\delta\trkla{u_\delta}-f_\delta\trkla{u},u_\delta-\mathfrak{m}}\\
=\rkla{f_\delta\trkla{u_\delta}-f_\delta\trkla{u},u_\delta-u}+\rkla{f_\delta\trkla{u_\delta},u-\mathfrak{m}}\,,
\end{multline}
where the first term on the right-hand side is nonnegative.
Thus, subtracting \eqref{eq:tmp:itou} from \eqref{eq:tmp:itoud}, integrating with respect to $t$, and neglecting nonnegative terms provides
\begin{align}
\begin{split}
&\int_0^T\expected{e^{-\hat{c}t}\rkla{\norm{u_\delta\trkla{t}-\mathfrak{m}}_{-1}^2-\norm{u\trkla{t}-\mathfrak{m}}_{-1}^2}}\dt\\
&\quad+\rkla{2\varepsilon-\frac{2\widetilde{\alpha}}\varepsilon-\alpha\sum_{k\in\mathds{Z}}\abs{\lambda_k}^2\norm{\g{k}}_{\mathds{L}^\infty}^2}\int_0^T\expected{\int_0^te^{-\hat{c}s}\norm{\nabla u_\delta\trkla{s}-\nabla u\trkla{s}}_{\mathds{L}^2}^2\ds}\dt\\
&\quad+4\varepsilon\int_0^T\expected{\int_0^te^{-\hat{c}s}\rkla{\trkla{\nabla u_\delta\trkla{s},\nabla u\trkla{s}}-\norm{\nabla u\trkla{s}}_{\mathds{L}^2}^2}\ds}\dt\\
&\quad+\frac2\varepsilon\int_0^T\expected{\int_0^te^{-\hat{c}s}\ekla{\rkla{f_\delta\trkla{u_\delta\trkla{s}},u\trkla{s}-\mathfrak{m}} -\rkla{\widetilde{p}\trkla{s},u\trkla{s}-\mathfrak{m}}}\ds}\dt\\
&\quad+\rkla{\hat{c}-\frac2\varepsilon C_{\widetilde{\alpha}}-C_\alpha\sum_{k\in\mathds{Z}}\abs{\lambda_k}^2\norm{\g{k}}_{\mathds{L}^\infty}^2}\int_0^T\expected{\int_0^te^{-\hat{c}s}\norm{u_\delta\trkla{s}-u\trkla{s}}_{-1}^2\ds}\dt\\
&\quad+2\hat{c}\int_0^T\expected{\int_0^te^{-\hat{c}s}\rkla{\trkla{u_\delta\trkla{s}-\mathfrak{m},u\trkla{s}-\mathfrak{m}}_{-1}-\norm{u\trkla{s}-\mathfrak{m}}_{-1}^2}\ds}\dt\\
&\leq \frac4\varepsilon \int_0^T\expected{\int_0^t e^{-\hat{c}s}\rkla{\trkla{u_\delta\trkla{s}-\mathfrak{m},u\trkla{s}-\mathfrak{m}}-\norm{u\trkla{s}-\mathfrak{m}}_{\mathds{L}^2}^2}\ds}\dt\\
&\quad+2\int_0^T\expected{\int_0^te^{-\hat{c}s}\sum_{k\in\mathds{Z}}\rkla{\ekla{\mysigma\trkla{u_\delta\trkla{s}}-\widetilde{\sigma}\trkla{s}}\lambda_k\g{k},\mysigma\trkla{u\trkla{s}}\lambda_k\g{k}-\trkla{\mysigma\trkla{u\trkla{s}}\lambda_k\g{k}}_{\D}}\ds}\dt%\\
%&\quad-2\int_0^T\expected{\int_0^te^{-\hat{c}s}\sum_{k\in\mathds{Z}}\rkla{\trkla{\tekla{\mysigma\trkla{u_\delta\trkla{s}}-\widetilde{\sigma}\trkla{s}}\lambda_k\g{k}}_{\D},\mysigma\trkla{u\trkla{s}}\lambda_k\g{k}-\trkla{\mysigma\trkla{u\trkla{s}}\lambda_k\g{k}}_{\D}}\ds}\dt
\,.
\end{split}
\end{align}
Choosing $\alpha$ and $\widetilde{\alpha}$ sufficiently small and $\hat{c}$ sufficiently large, we can pass to the superior limit and obtain
\begin{align}
\limsup_{\delta\rightarrow0}\int_0^T\expected{e^{-\hat{c}t}\norm{u_\delta-\mathfrak{m}}_{-1}^2}\dt\leq \int_0^T\expected{e^{-\hat{c}t}\norm{u-\mathfrak{m}}_{-1}^2}\dt\,.
\end{align}
Thus, $u_\delta\rightarrow u$ in $L^2\trkla{\Omega;L^2\trkla{0,T;\mathds{H}^{-1}_\per}}$.
For $q<2$ we have
\begin{align*}
\begin{split}
\norm{u_\delta-u}_{L^{2q}\trkla{\Omega;L^2\trkla{0,T;\mathds{L}^2}}}^{2q}\leq&\, \expected{\rkla{\norm{u_\delta-u}_{L^2\trkla{0,T;\mathds{H}^{-1}_\per}}\norm{u_\delta-u}_{L^2\trkla{0,T;\mathds{H}^1}}}^q}\\
\leq&\,\norm{u_\delta-u}_{L^2\trkla{\Omega;L^2\trkla{0,T;\mathds{H}^{-1}_\per}}}^{q/2} \norm{u_\delta-u}_{L^{2q/\trkla{q-2}}\trkla{\Omega;L^2\trkla{0,T;\mathds{H}^1}}}^{\trkla{2-q}/q}\rightarrow0\,.
\end{split}
\end{align*}
This allows us to deduce
\begin{multline}
\expected{\int_0^T\rkla{\widetilde{p}_-+\widetilde{p}_+,u}\dt}=\expected{\int_0^T\rkla{\widetilde{p},u}\dt}\leftarrow \expected{\int_0^T\rkla{f_\delta\trkla{u_\delta},u_\delta}\dt}\\
=\expected{\int_0^T\rkla{f_\delta\trkla{u_\delta}\chi_\delta^-,\trkla{u_\delta+1}^--1}\dt}+\expected{\int_0^T\rkla{f_\delta\trkla{u_\delta}\chi_\delta^+,\trkla{u_\delta-1}^++1}\dt}\\
\rightarrow -\expected{\int_0^T\rkla{\widetilde{p}_-,1}\dt} +\expected{\int_0^T\rkla{\widetilde{p}_+,1}\dt}\,.
\end{multline}
Thus, we have $\mathds{P}$-almost surely, almost everywhere in $\trkla{0,T}\times\D$ that
\begin{align}\label{eq:complem}
\widetilde{p}_+\trkla{u-1}+\widetilde{p}_-\trkla{u+1}=0\,,
\end{align}
i.e. $\widetilde{p}_+=\widetilde{p}_-=\widetilde{p}=0$ if $u\in\trkla{-1,+1}$, $\widetilde{p}_+=0$ and $\widetilde{p}=\widetilde{p}_-\leq0$ for $u=-1$, as well as $\widetilde{p}_-=0$ and $\widetilde{p}=\widetilde{p}_+\geq0$ for $u=+1$.
Hence, $\widetilde{p}\in\partial\Psi+u$.
Having the strong convergence of $u_\delta\rightarrow u$, we can extract a subsequence converging pointwise almost everywhere and identify $\widetilde{\mysigma}$ with $\mysigma\trkla{u}$ using a Vitali argument.

\end{proof}

\begin{lemma}\label{thm:uniqueness}
The solution to problem \eqref{eq:limitdelta} is $\mathbb{P}$-a.s. unique.
%\todo[inline]{In which sense are the Lagrange multipliers $\widetilde{p}$ unique?}
\end{lemma}
\begin{proof}
We assume that $\trkla{u_1,\widetilde{p}_1}$ and $\trkla{u_2,\widetilde{p}_2}$ are two solutions of \eqref{eq:limitdelta} with respect to the same initial and the same $\mathcal{Q}$-Wiener process $\bs{\W}$.
For each of these solutions, there exists a unique chemical potential denoted by $w_1$ and $w_2$ defined via \eqref{eq:limitdelta:w}.
For a sufficiently large constant $\widehat{c}$, we apply \Ito's formula to $F\trkla{t,u}:=e^{\widehat{c}t}\norm{u}_{-1}^2$ and obtain
\begin{align}
\begin{split}
&\expected{e^{-\widehat{c}t}\norm{u_1-u_2}_{-1}^2}=\expected{\int_0^t-\widehat{c}e^{-\widehat{c}s}\norm{u_1-u_2}_{-1}^2\ds}\\
&\quad-2\expected{\int_0^te^{-\widehat{c}s}\rkla{u_1-u_2,-\varepsilon\Delta\trkla{u_1-u_2}+\tfrac1\varepsilon\trkla{p_1-p_2}-\tfrac1\varepsilon\trkla{u_1-u_2}}\ds}\\
&\quad+\expected{\int_0^te^{-\widehat{c}s}\sum_{i=1}^d\sum_{k\in\mathds{Z}}\norm{\partial_i\ekla{\trkla{\mysigma\trkla{u_1}-\mysigma\trkla{u_2}}\lambda_k\g{k}}}_{-1}^2\ds}\,.
\end{split}
\end{align}
Noting $\trkla{u_1-u_2,p_1-p_2}\geq0$ and arguing as in Lemma \ref{lem:delta:unique}, we obtain
\begin{align}
\expected{e^{-\widehat{c}t}\norm{u_1-u_2}_{-1}^2}\leq 0\,.
\end{align}
The interpolation estimate
\begin{align*}\expected{\norm{u_1\trkla{t}-u_2\trkla{t}}^2}\leq C\expected{e^{-\widehat{c}t}\norm{u_1\trkla{t}-u_2\trkla{t}}_{-1}^2}^{1/2}\expected{e^{\widehat{c}t}\rkla{\norm{u_1\trkla{t}}_{\mathds{H}^1}^2+\norm{u_2\trkla{t}}_{\mathds{H}^1}^2}}^{1/2}
\end{align*}
then implies pathwise uniqueness of $u$. % with respect to the $L^\infty\trkla{0,T;\mathds{L}^2}$-norm.

Hence, from \eqref{eq:limitdelta} we deduce that the chemical potential is unique up to a time-dependent constant, which yields $\widetilde{p}_1(t) - \widetilde{p}_2(t) = \widetilde{C}(t)$ for a.a.~$t\in(0,T)$, $\mathbb{P}$-a.s., where $\widetilde{C}(t)$ is $\mathbb{P}$-a.s.~constant in space.
Property \eqref{eq:complem} implies that $\widetilde{p}_1(t) = \widetilde{p}_2(t) = 0$, a.e.~in $\mathcal{D}$ where $|u(t)| < 1$.
Consequently, $\widetilde{C} = 0$, i.e., $\widetilde{p}_1(t) - \widetilde{p}_2(t) = 0$, a.e. in $(0,T)\times \mathcal{D}$, $\mathbb{P}$-a.s..
\end{proof}

\section{Numerical experiments}\label{sec:numerics}

In this section we present numerical experiments using scheme (\ref{eq:model:disc:inequality})
and its regularised version (\ref{eq:model:disc:regularised}) (where we choose the regularisation parameter $\delta=10^{-3}$, unless mentioned otherwise).
The discrete variational inequality in scheme (\ref{eq:model:disc:inequality}) is solved using an active-set strategy \cite{voids3d}, \cite{gk09}.
The nonlinear system associated with the regularised scheme (\ref{eq:model:disc:regularised}) is solved using the Newton method.

For convenience of implementation we take $\mathcal{D} = (0,1)^2$ along with Neumann boundary conditions on $\partial \mathcal{D}$.
The noise is taken in the form
$$
\big(\Phi\trkla{u}\dW, \varphi\big) = \nu\sum_{k=1}^{\dim \widetilde{\mathds{V}}_h} \big(\max\{0, 1-u^2\} \widetilde{\psi}_{k}, (\dbeta{k}^1,\dbeta{k}^2)\cdot \nabla \varphi \big),
$$
where $\nu>0$ and $\widetilde{\mathds{V}}_h = \mathrm{span}\{\widetilde{\psi}_{k};\, k=1,\dots,\dim\widetilde{\mathds{V}}_h\}$
is the finite element space associated with the partition $\widetilde{\mathcal{T}}_h$ of $\mathcal{D}$ into squares with side $\tilde{h} = 1/L$
where (if not mentioned otherwise) each square is partitioned into four triangles with common vertex at the barycenter of the square,
i.e., $\dim\widetilde{\mathds{V}}_h = (L+1)^2+L^2$ and $\widetilde{\psi}_{k}$ and the nodal basis functions of $\widetilde{\mathds{V}}_h$.
In the experiments below we choose $L=16$.

\subsection{Spinodal decomposition}

We take $T=0.012$, $\varepsilon = 1/(12\pi)$, $\nu=10^{-5}$.
The computations were performed with the time step size $\tau=10^{-6}$ and $\mathbb{V}_h$ the finite element space associated with an equidistant mesh consisting of
squares with mesh size $h=1/256$ where each square is divided into four triangles along its diagonals.
The initial condition is displayed in Figures~\ref{fig_spin_u0} (left), here the coloring scale is adapted to the extremal values $-0.11$ (blue), $0.19$ (red) of the initial condition.

The evolution of the computed solution is displayed in Figures \ref{fig_spin_det} and \ref{fig_spin_sym}.
We only display the stochastic solution computed with the the unregularised scheme (\ref{eq:model:disc:inequality}).
The solution computed with regularised scheme (\ref{eq:model:disc:regularised}), for $\delta=10^{-3}$ was graphically indistinguishable
and remained bounded within the range $(-1.002, 1.002)$.

We observe that already an extremely small amount of noise leads to a completely different evolution of the solution.
In particular, the addition of noise accelerates the early stages of the decomposition process significantly.
{Furthermore, the results seem to be sensitive to the geometry of the triangulation used to construct the space $\widetilde{\mathds{V}}\h$.}
In Figure~\ref{fig_spin_nonsym1} and Figure~\ref{fig_spin_nonsym2} we display the solution computed with noise constructed on partitions consisting of halved squares along their respective diagonals
(the Brownian motions associated with the vertices of the noise partition are the same as in the ''symmetric'' case in Figure~\ref{fig_spin_sym}).
The computed results for all three noises exhibit some similarities (compare in particular the respective solution at time $0.9\times10^{-3}$).
However, the solutions computed with noises on the partitions which consist of halved squares exhibit some anisotropy, i.e., the interfaces
tend to be aligned along the diagonals of the squares employed in the respective partitions.

\begin{figure}
\includegraphics[width=0.32\textwidth]{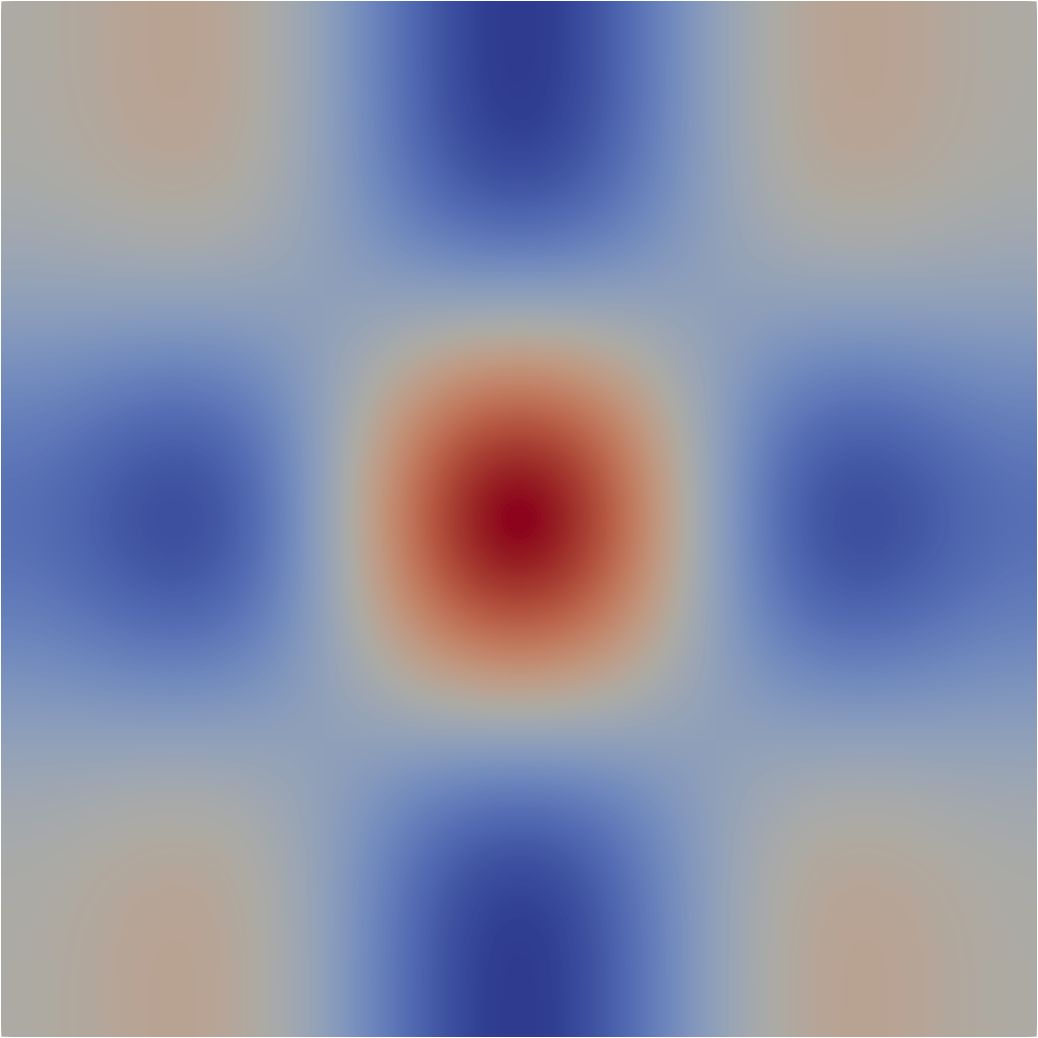}
\includegraphics[width=0.32\textwidth]{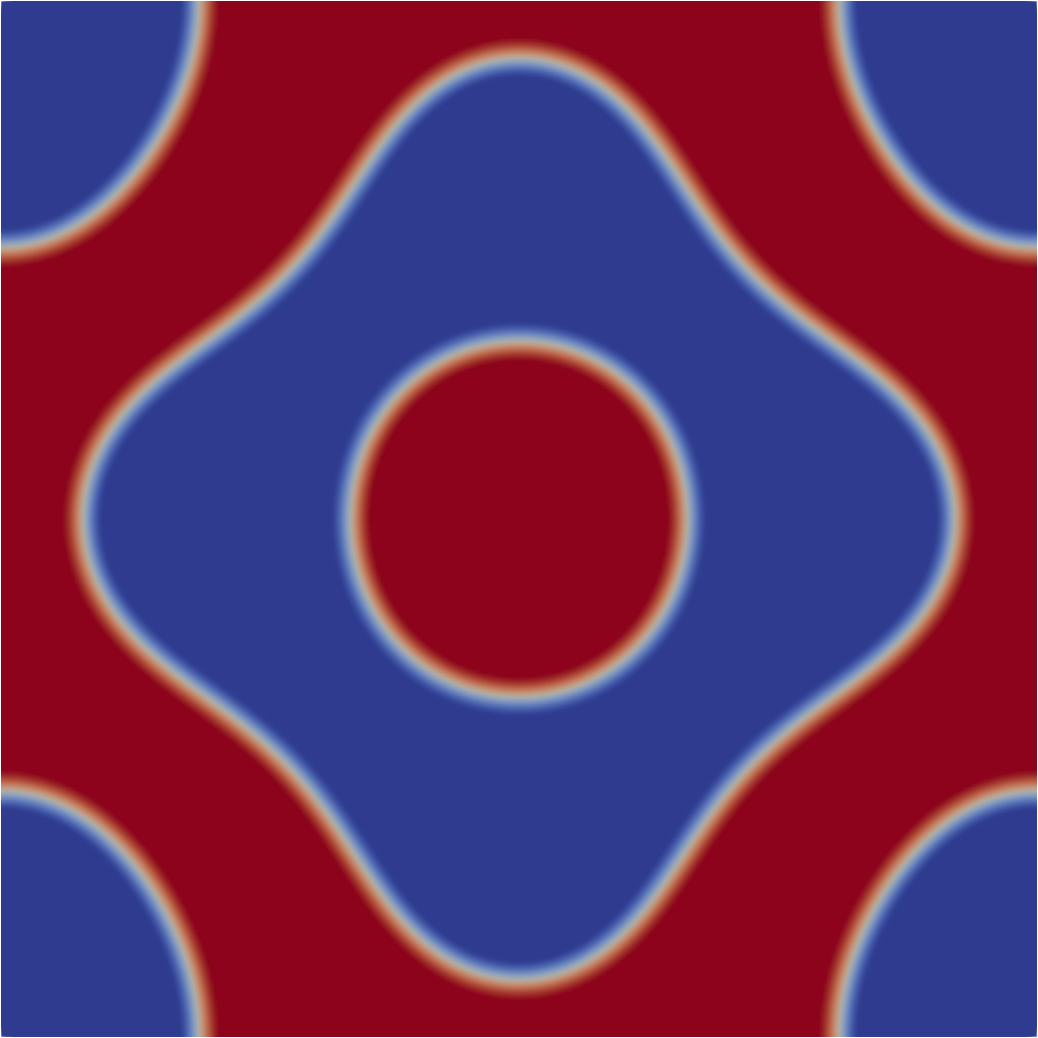}
\includegraphics[width=0.32\textwidth]{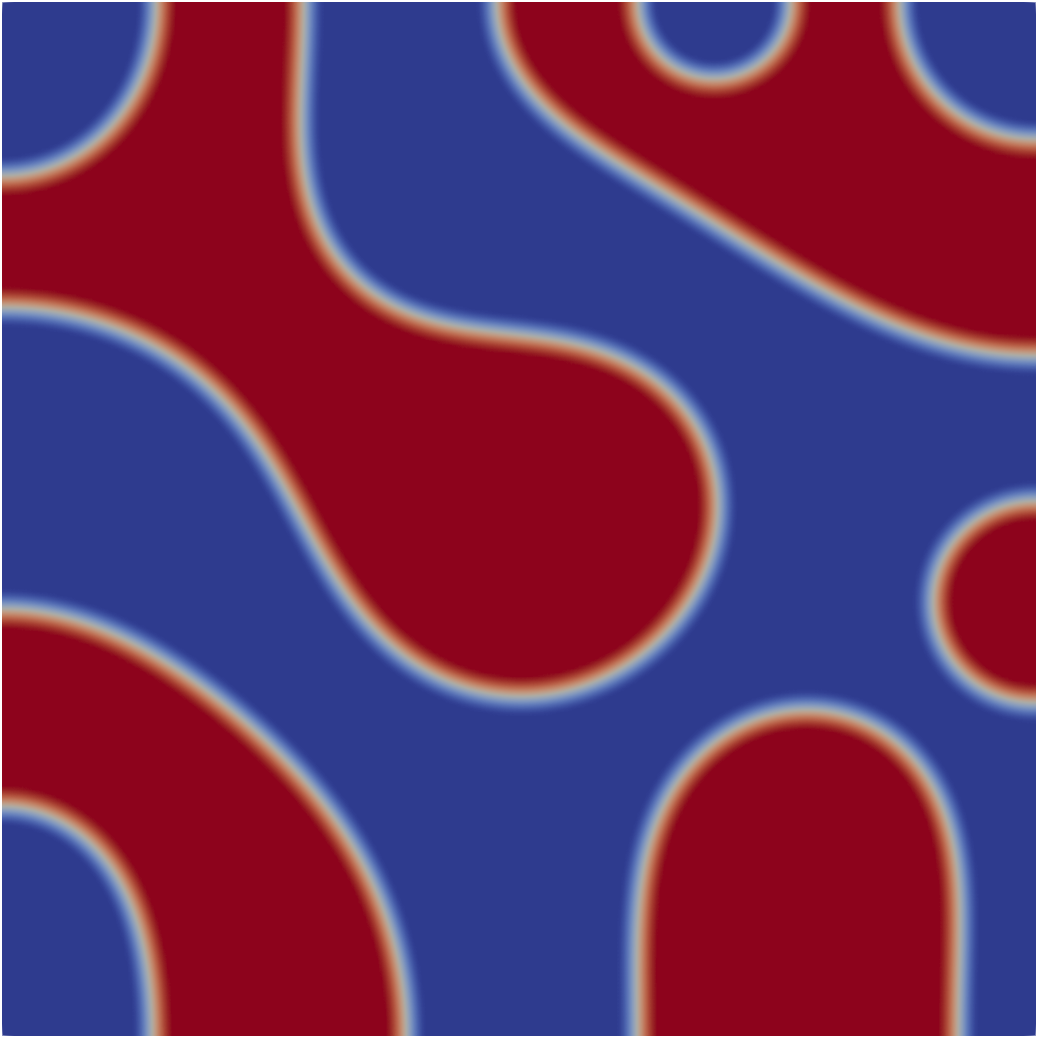}
\caption{Initial condition (left), deterministic numerical solution ($\nu=0$) (middle) and stochastic numerical solution ($\nu=1.6^{-4}$) (right) at time $t=1.2\times 10^{-2}$.}
\label{fig_spin_u0}
\end{figure}

\begin{figure}
\includegraphics[width=0.24\textwidth]{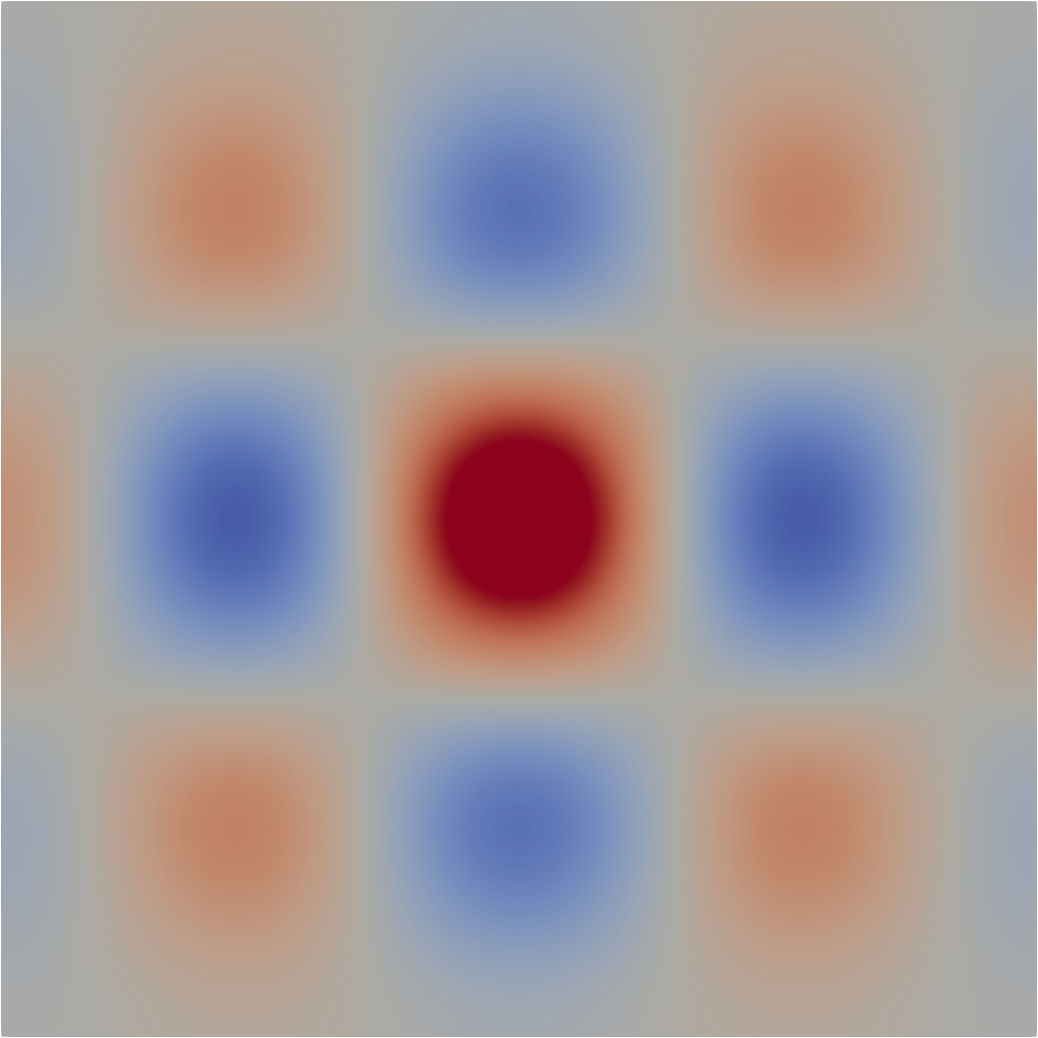}
\includegraphics[width=0.24\textwidth]{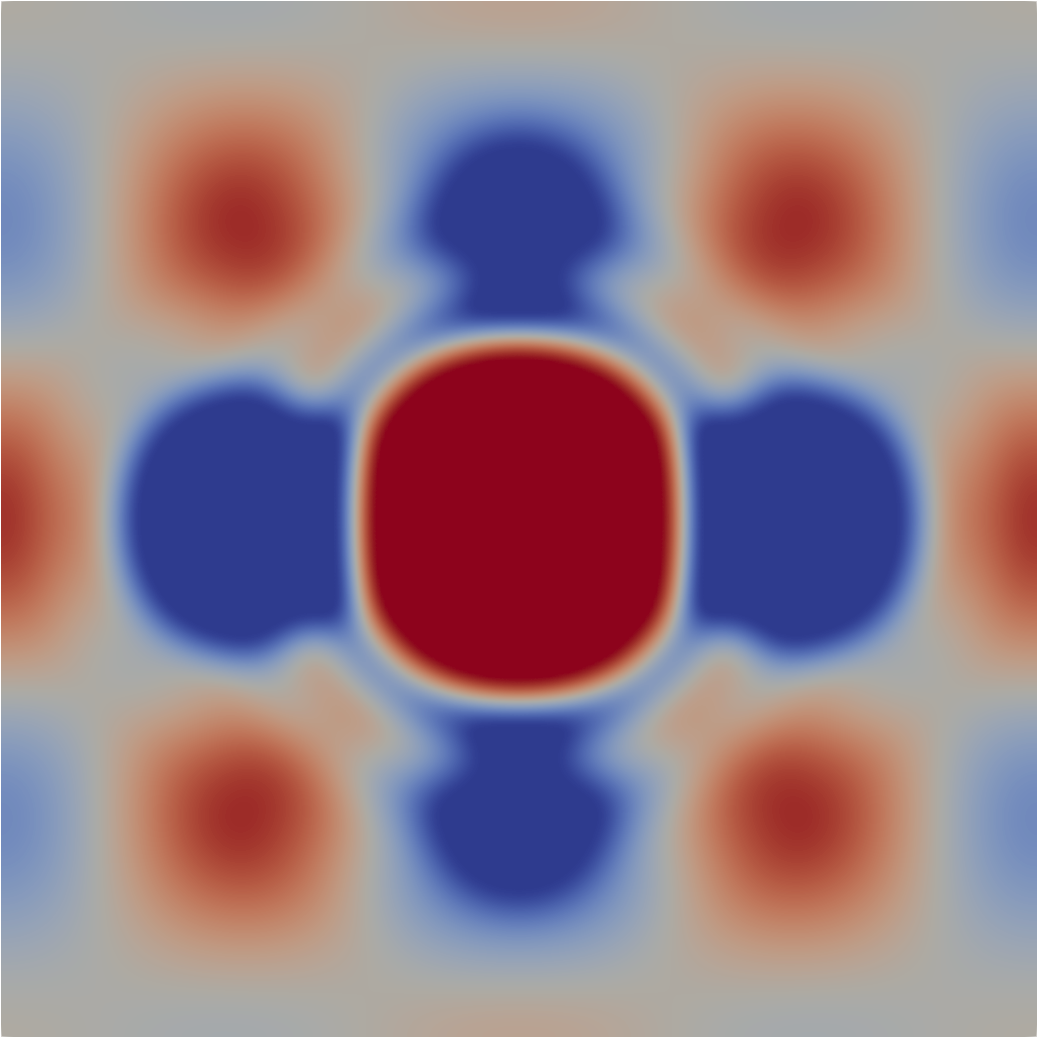}
\includegraphics[width=0.24\textwidth]{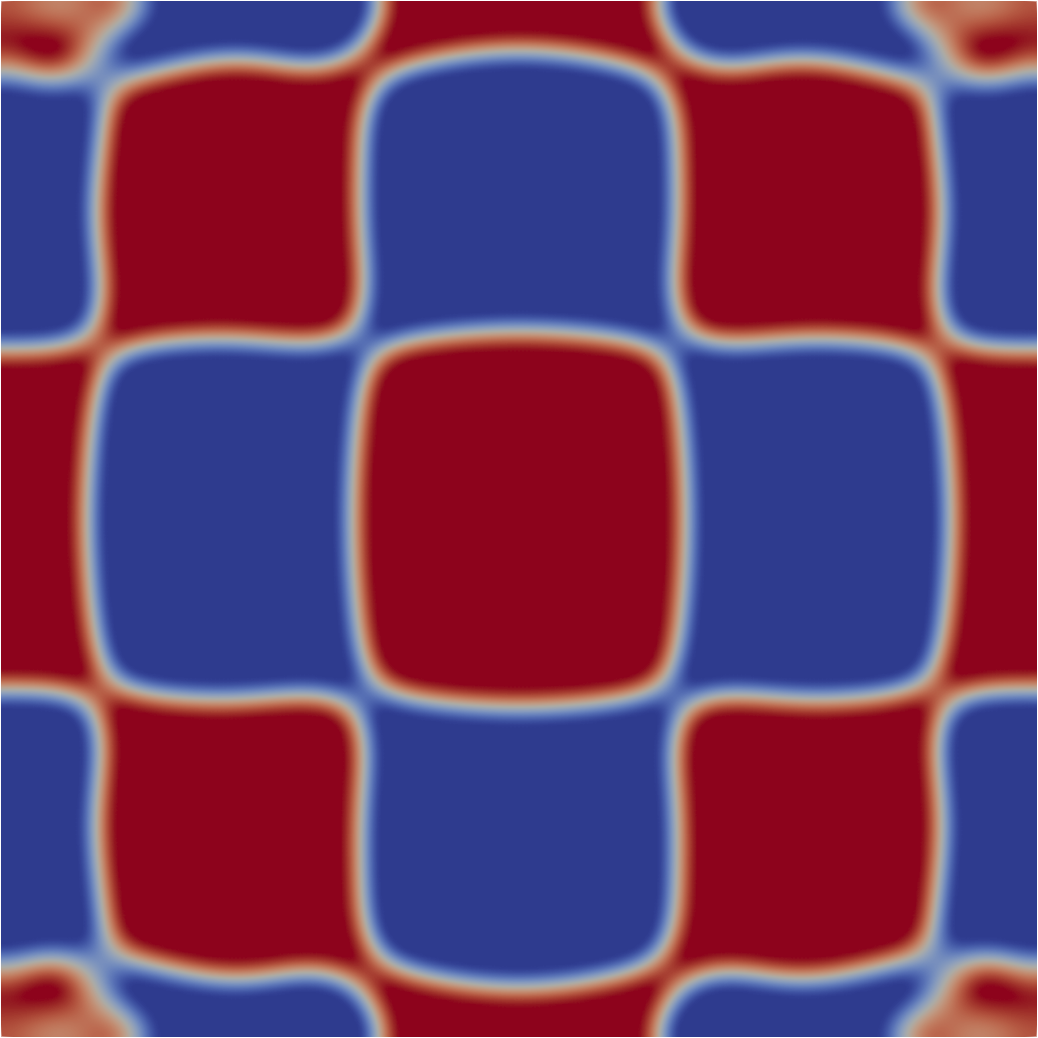}
\includegraphics[width=0.24\textwidth]{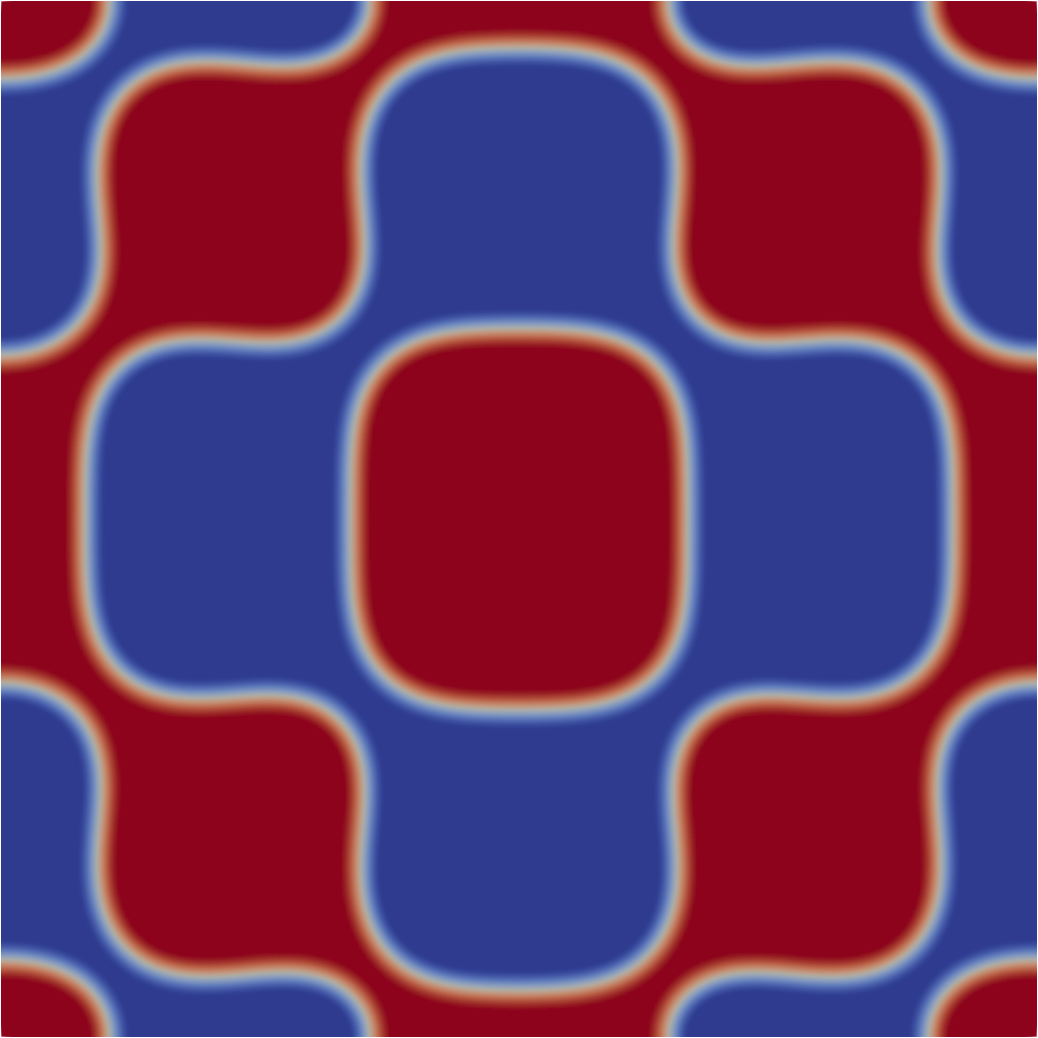}
\caption{Deterministic numerical solution ($\nu=0$) at time $t=1.2,1.5,2,2.5 \times 10^{-3}$.}
\label{fig_spin_det}
\end{figure}

\begin{figure}
\includegraphics[width=0.24\textwidth]{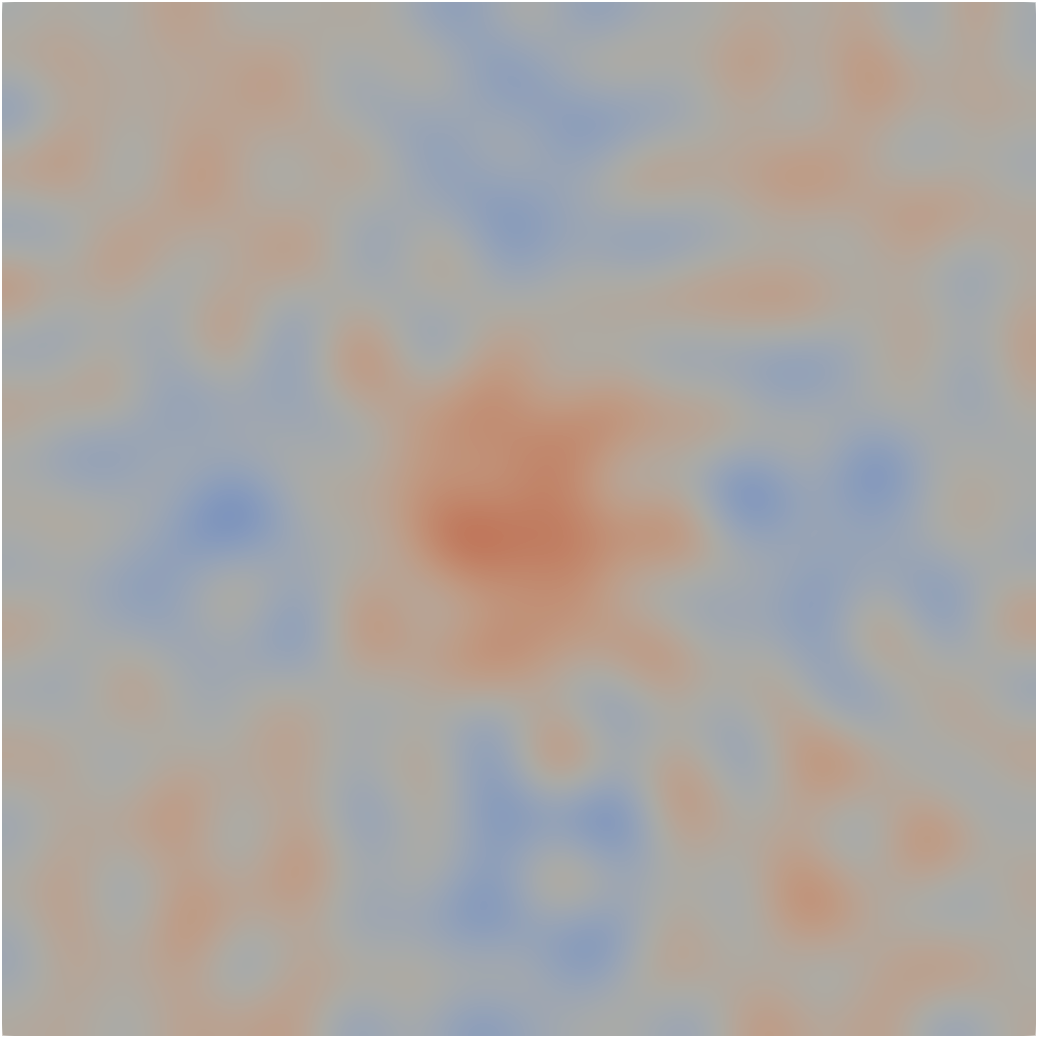}
\includegraphics[width=0.24\textwidth]{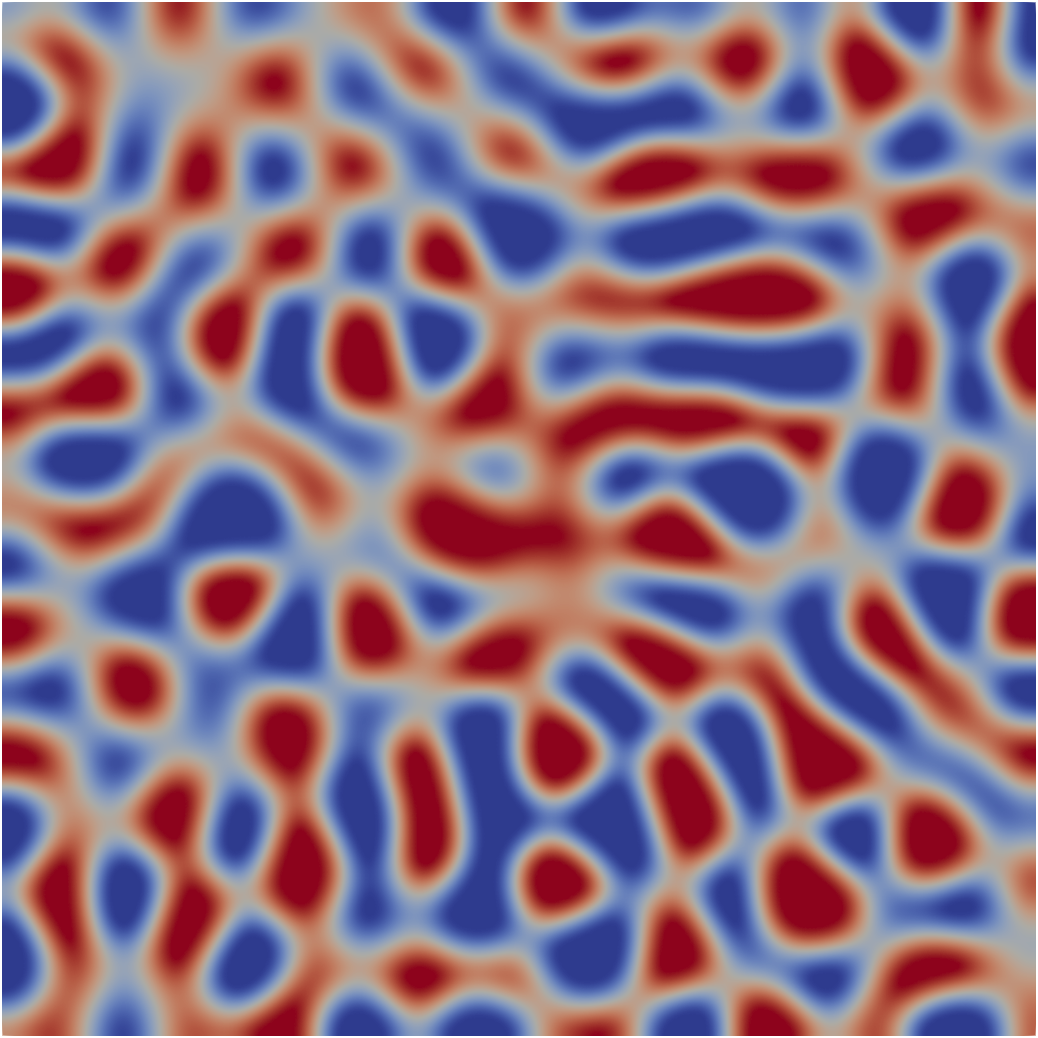}
\includegraphics[width=0.24\textwidth]{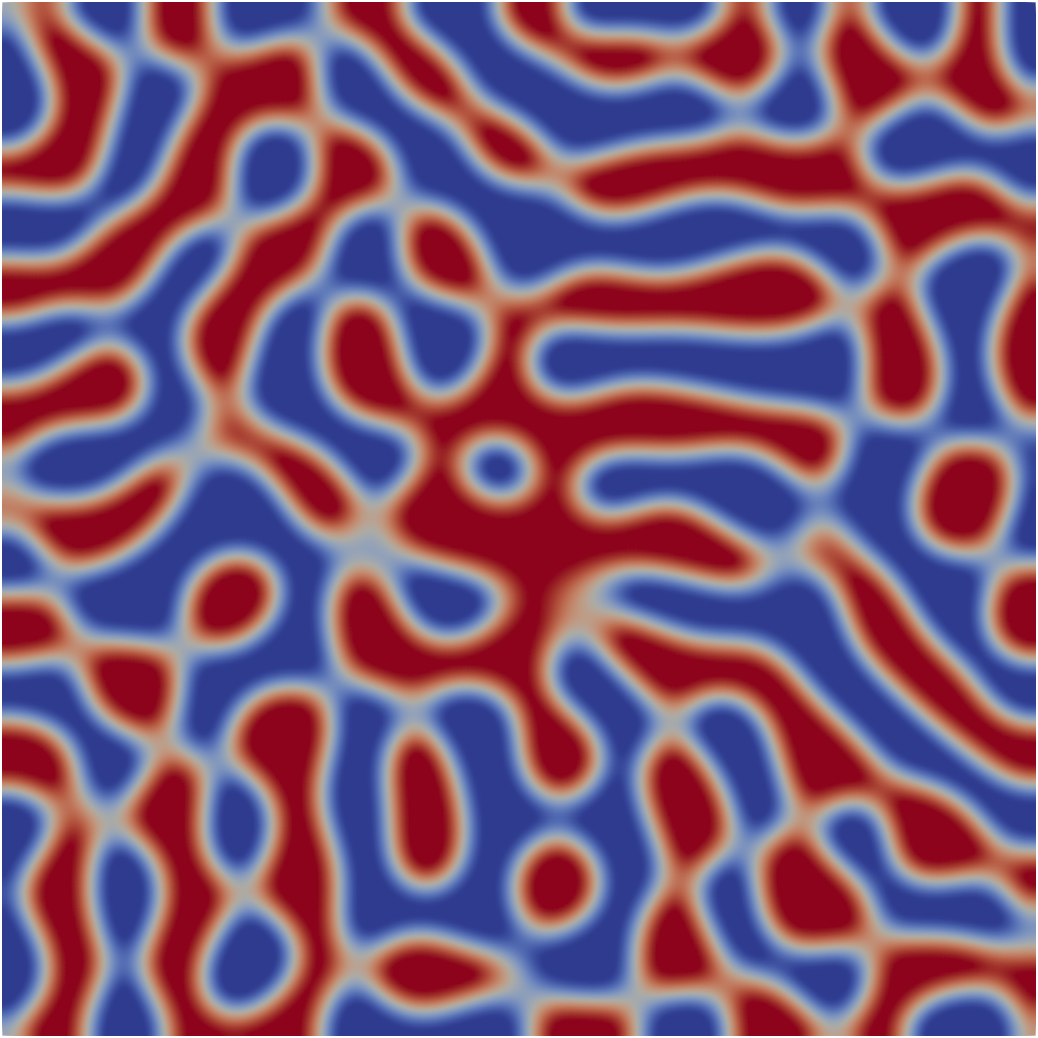}
\includegraphics[width=0.24\textwidth]{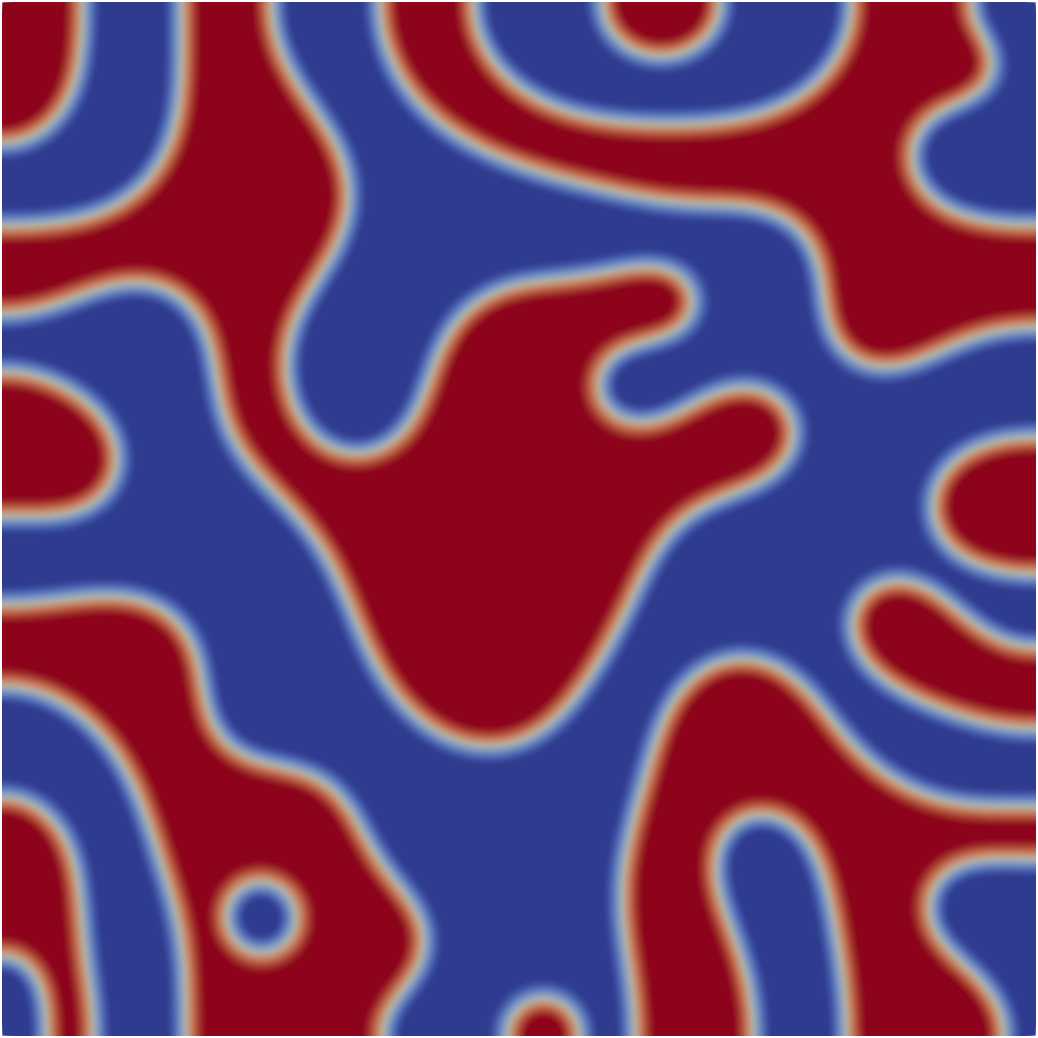}
\caption{One path of the stochastic numerical solution ($\nu=1.6\times10^{-4}$) at time $t=0.6,0.8,0.9,2.5 \times 10^{-3}$.}
\label{fig_spin_sym}
\end{figure}

\begin{figure}
\includegraphics[width=0.24\textwidth]{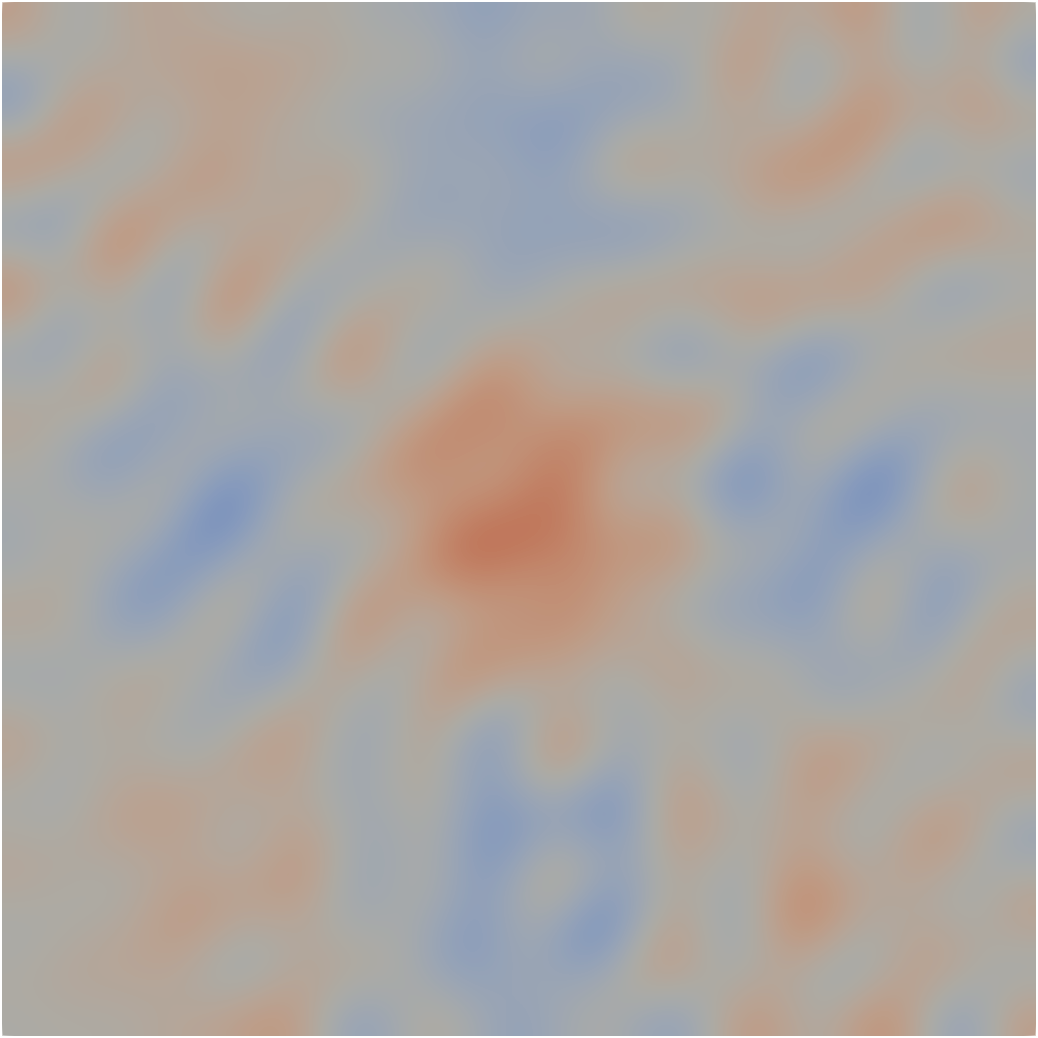}
\includegraphics[width=0.24\textwidth]{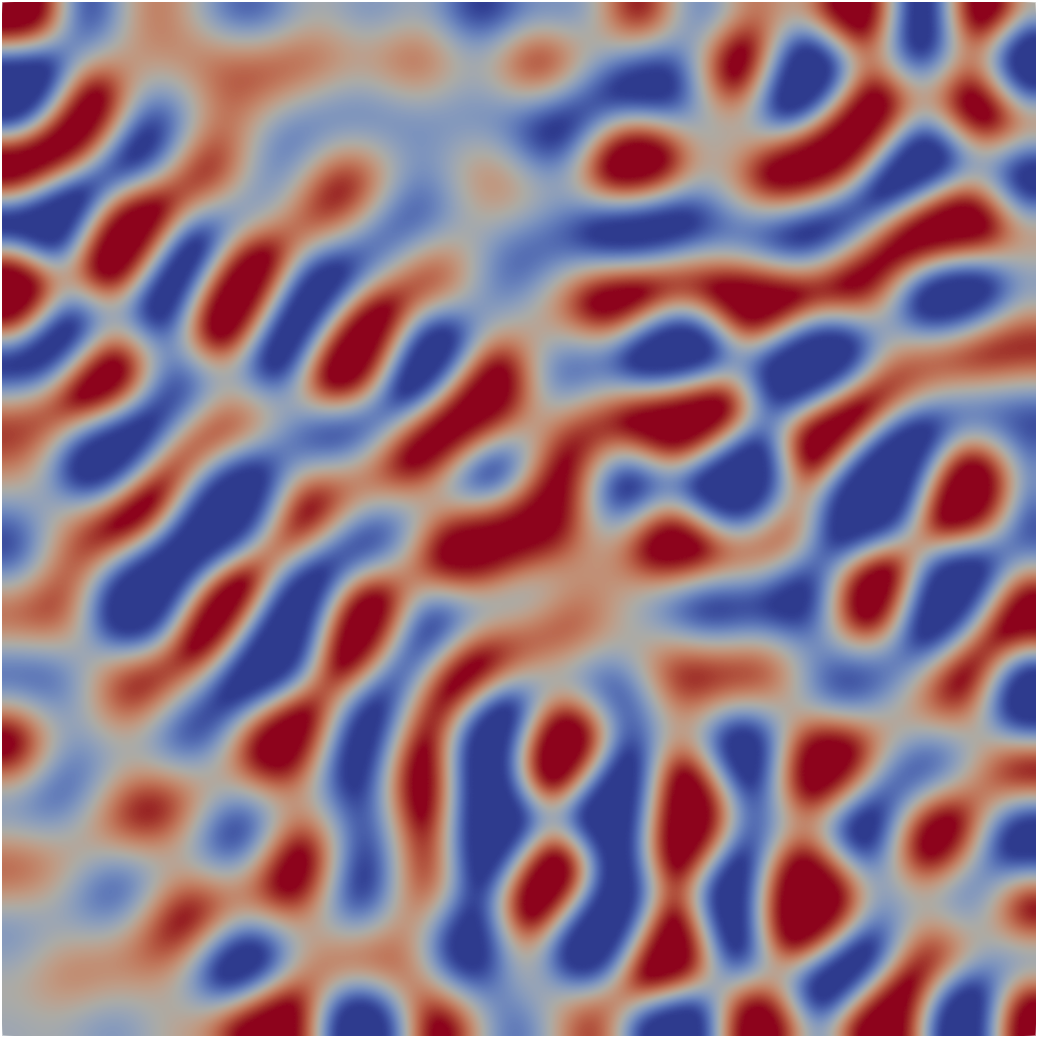}
\includegraphics[width=0.24\textwidth]{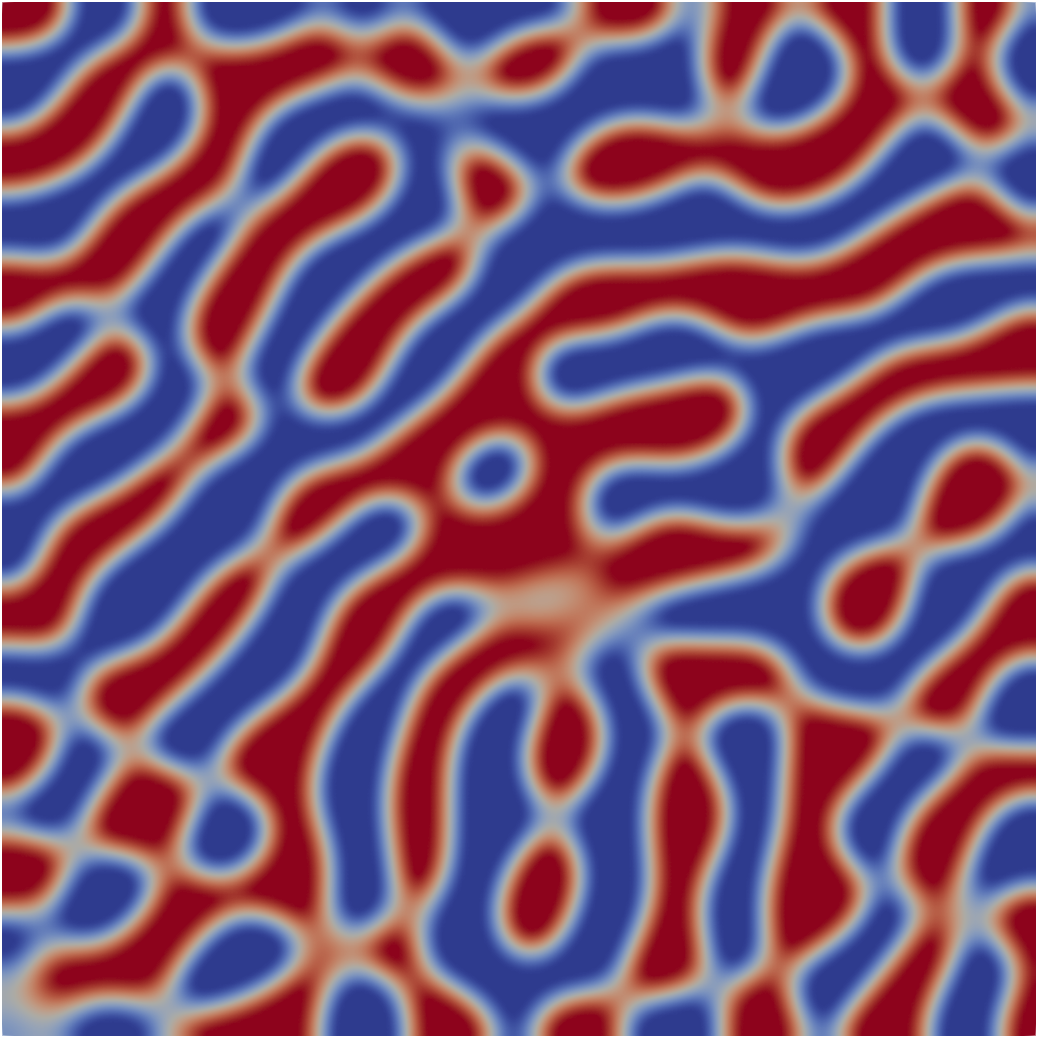}
\includegraphics[width=0.24\textwidth]{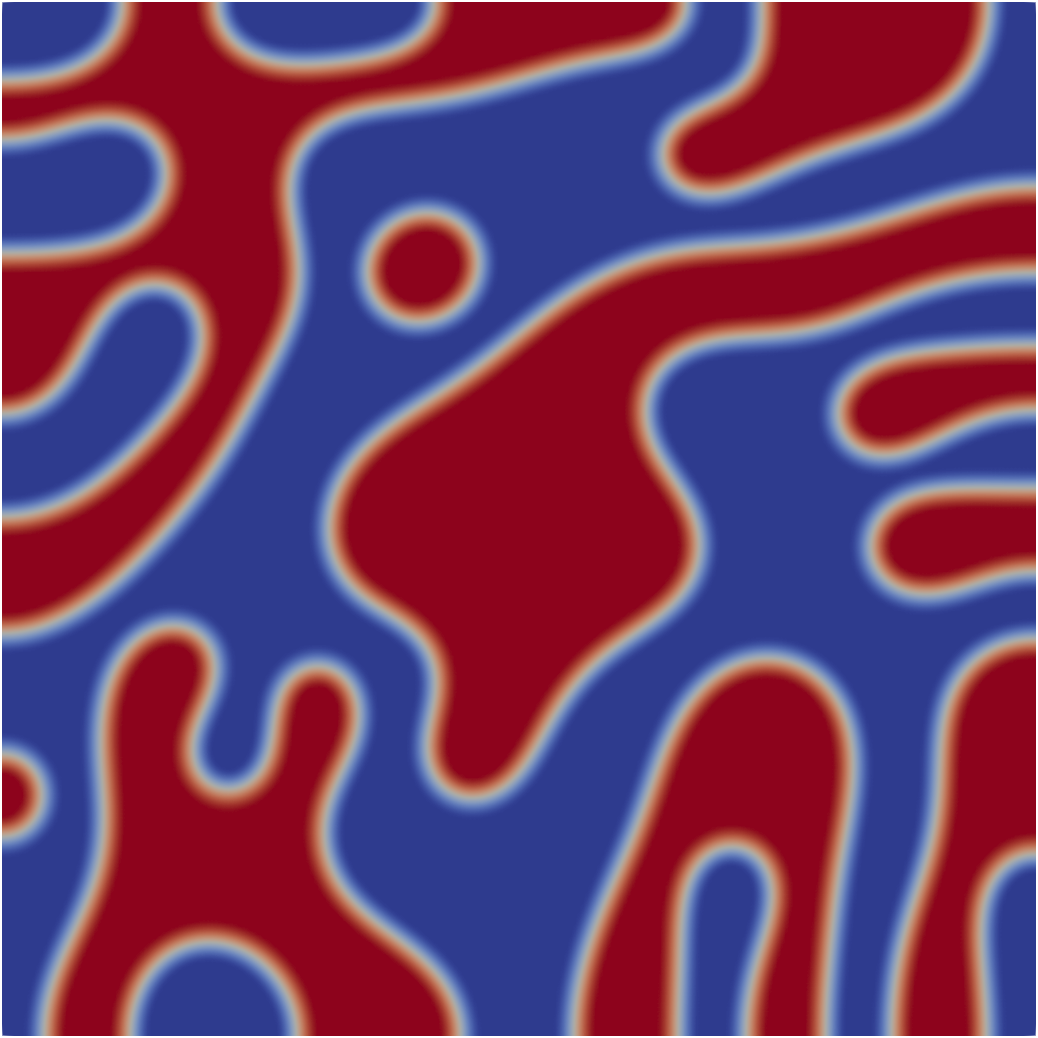}
\caption{One path of the stochastic numerical solution ($\nu=1.6\times10^{-4}$, halved squares, from lower left to upper right corner) at time $t=0.6,0.8,0.9,2.5 \times 10^{-3}$.}
\label{fig_spin_nonsym1}
\end{figure}

\begin{figure}
\includegraphics[width=0.24\textwidth]{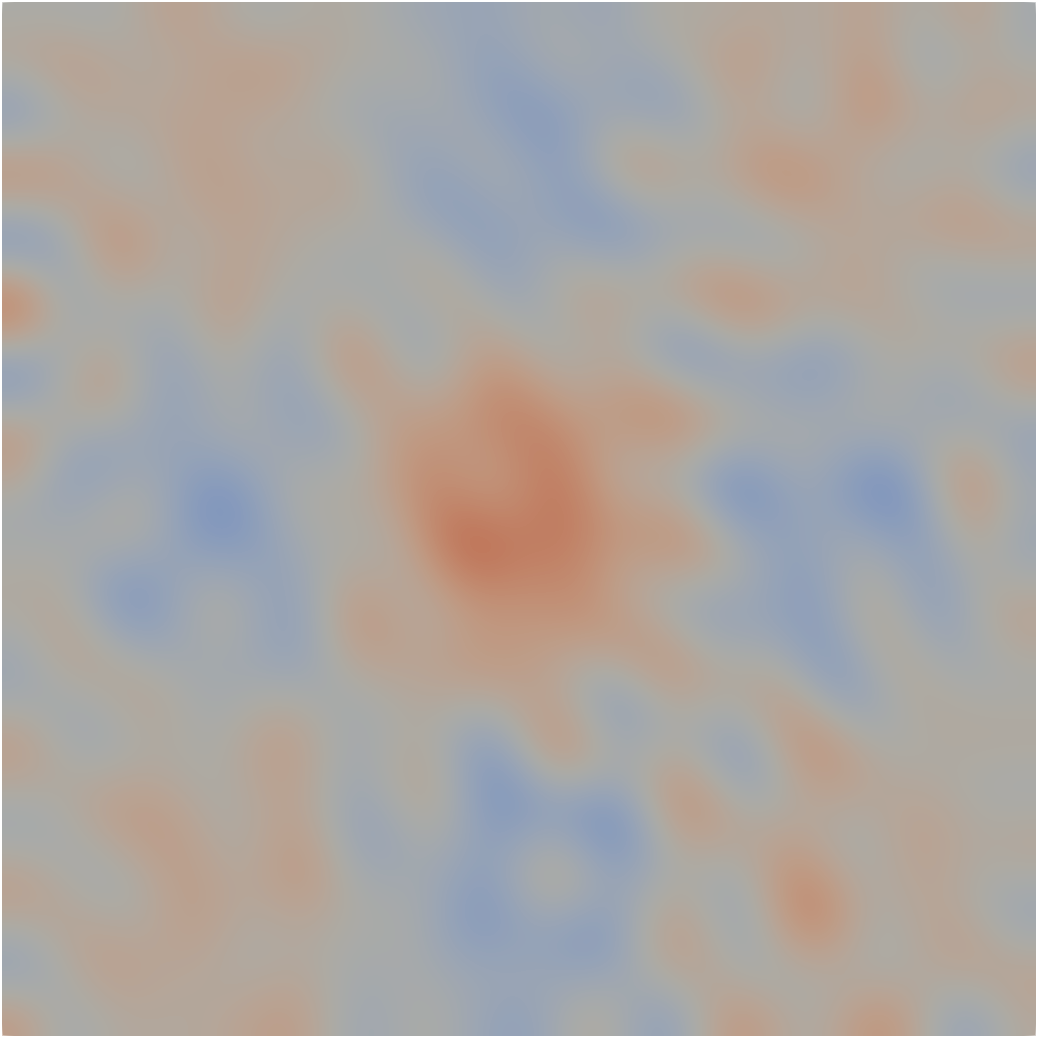}
\includegraphics[width=0.24\textwidth]{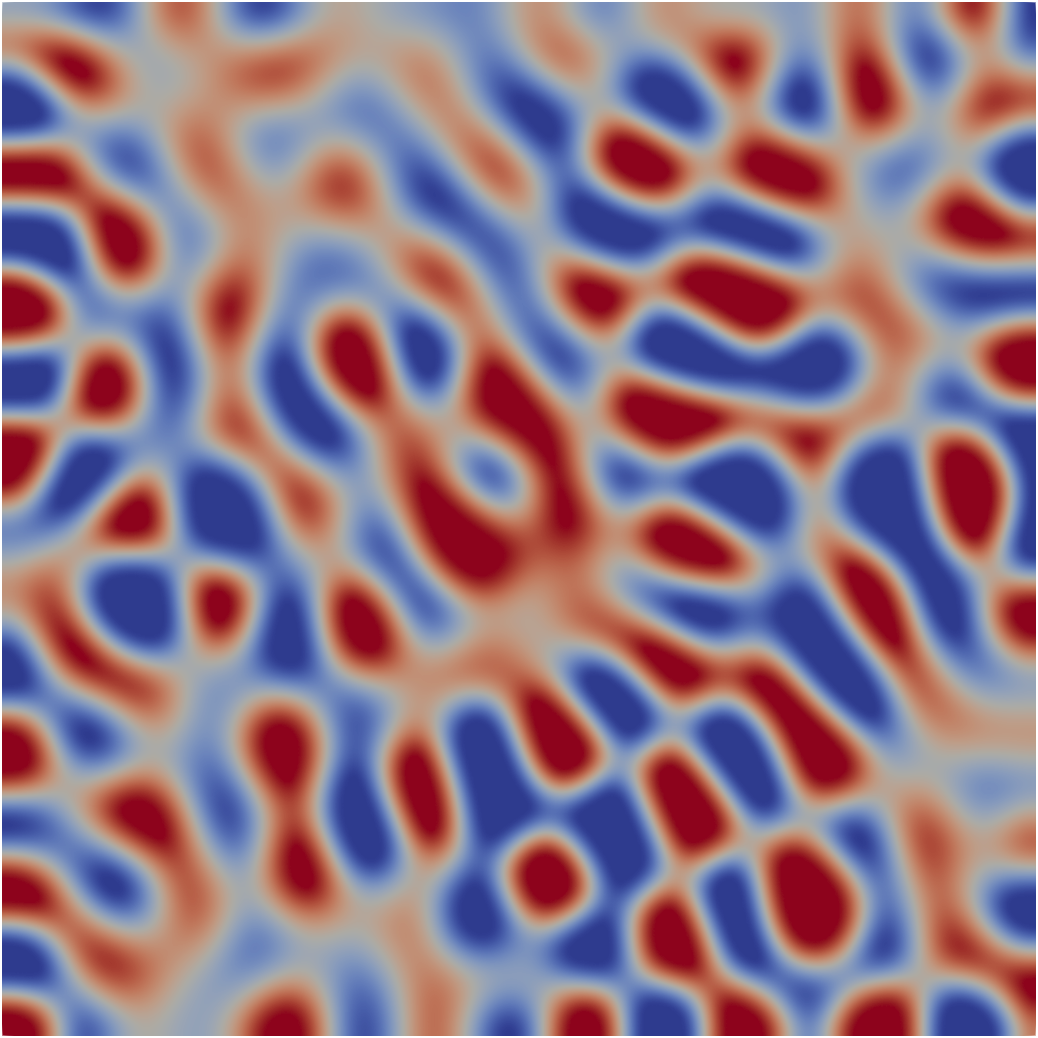}
\includegraphics[width=0.24\textwidth]{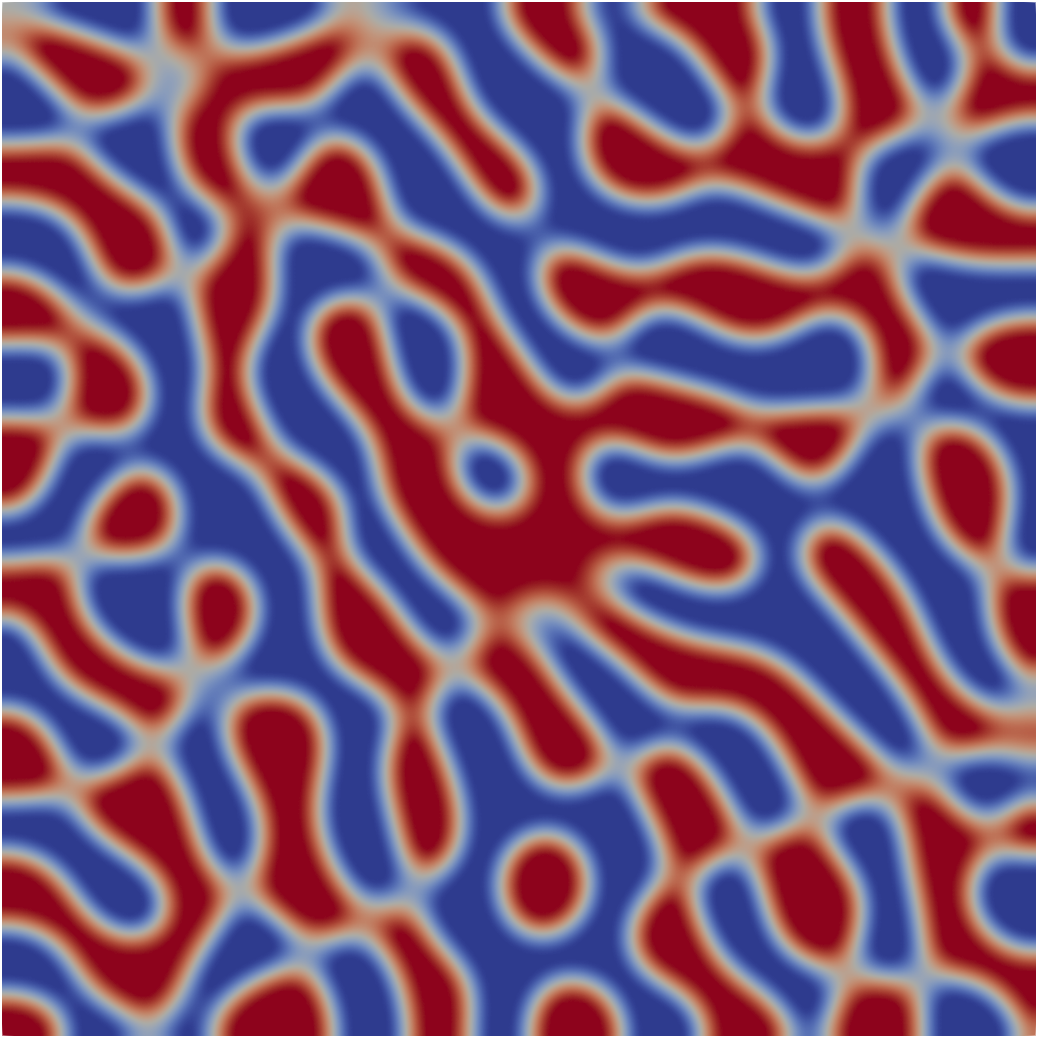}
\includegraphics[width=0.24\textwidth]{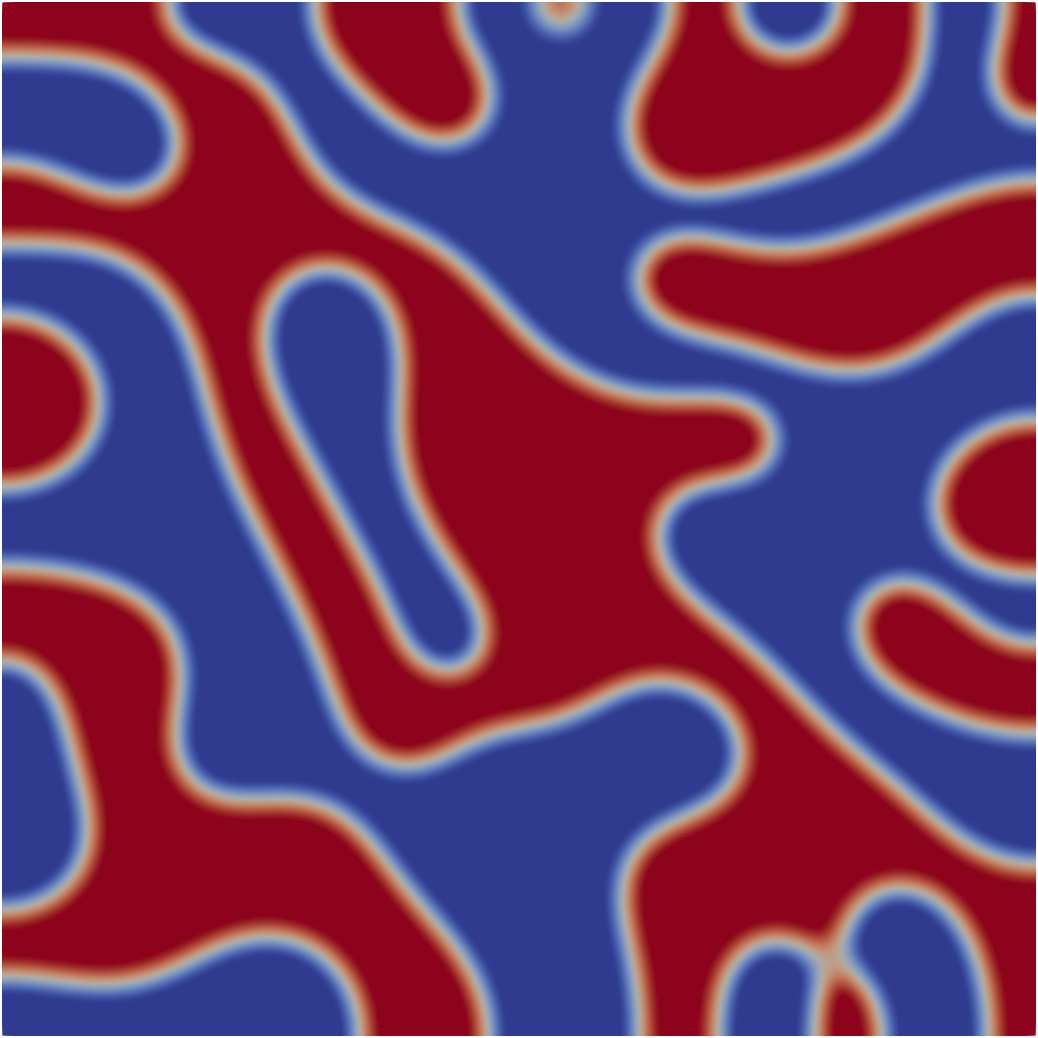}
\caption{One path of the stochastic numerical solution ($\nu=1.6\times10^{-4}$, halved squares, from lower right to upper left corner) at time $t=0.6,0.8,0.9,2.5 \times 10^{-3}$.}
\label{fig_spin_nonsym2}
\end{figure}

\subsection{Closing of a void}

The parameters for the computation were $T=0.012$, $\varepsilon = 1/(12\pi)$, $\tau=10^{-5}$.
We employ an adaptive mesh refinement algorithm from \cite{femtec}, \cite{aposter_ch} with (pathwise) refinement along the free-boundary between the subdomains where the solution $|u|=1$ and $|u|<1$, see Figure~\ref{fig_void_mesh}.

In Figure~\ref{fig_void_det} we display the evolution of the deterministic problem. In Figure~\ref{fig_void_stoch} we display
one path of the stochastic solution for $\nu=1.6$ computed with the scheme (\ref{eq:model:disc:inequality})
and the corresponding adaptive meshes are displayed in Figure~\ref{fig_void_mesh} (the corresponding results for the regularised scheme (\ref{eq:model:disc:regularised}) were similar).

\begin{figure}
\includegraphics[width=0.24\textwidth]{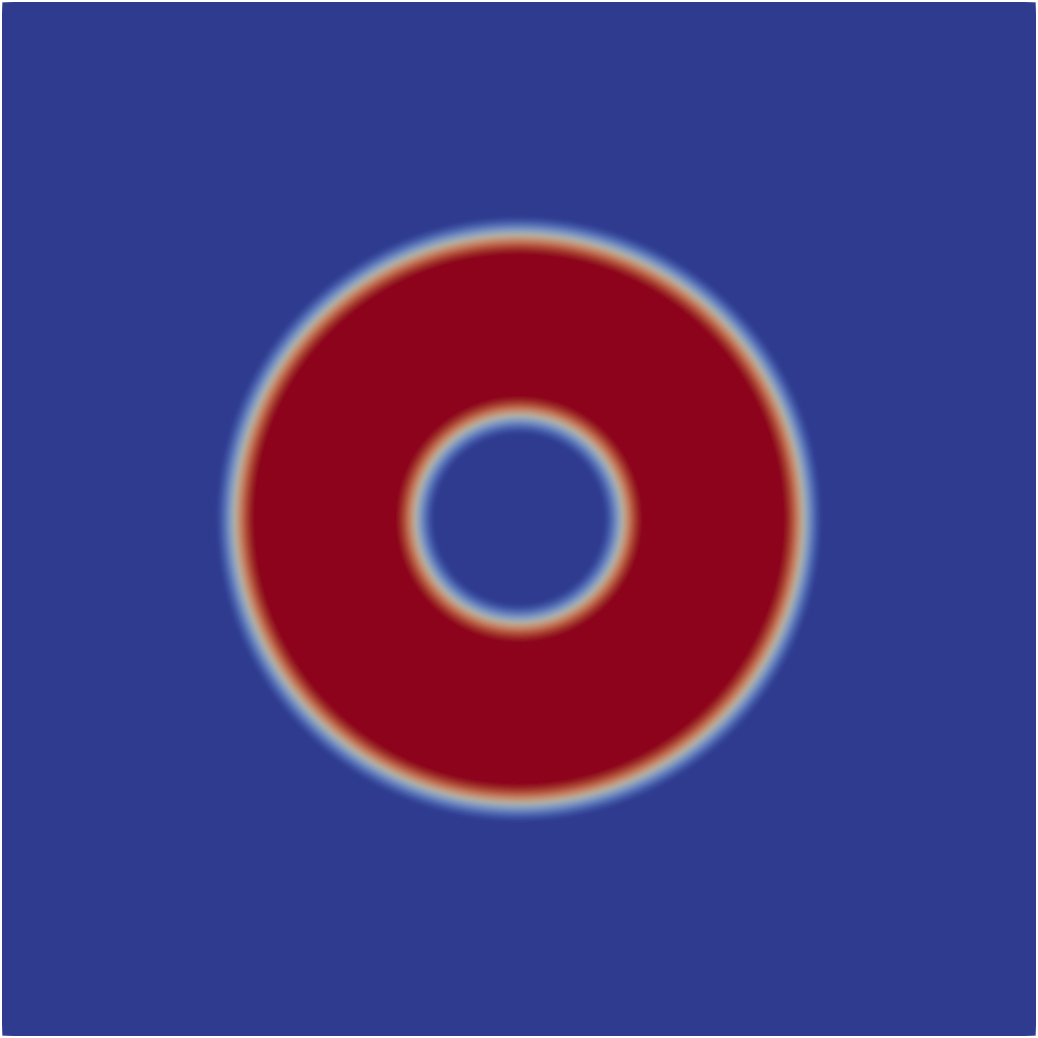}
\includegraphics[width=0.24\textwidth]{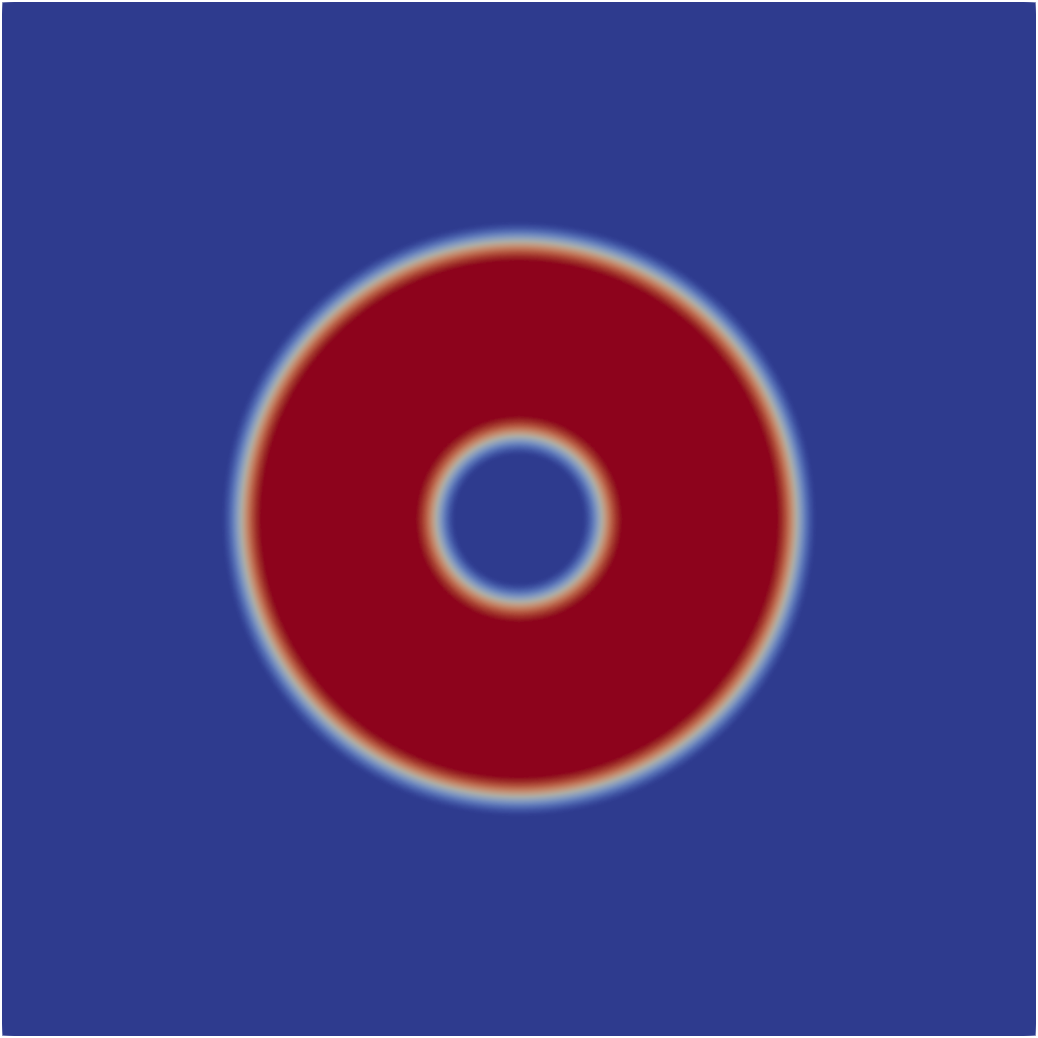}
\includegraphics[width=0.24\textwidth]{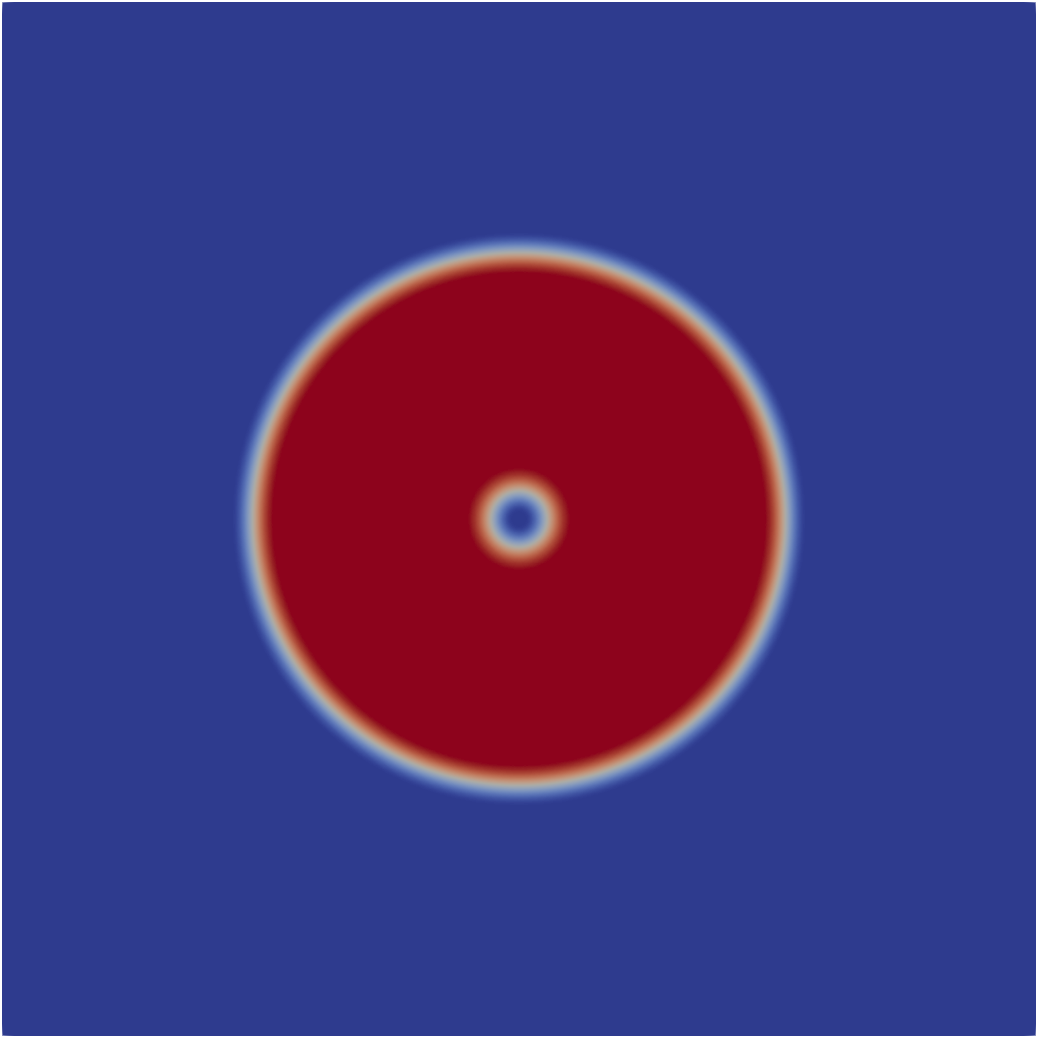}
\includegraphics[width=0.24\textwidth]{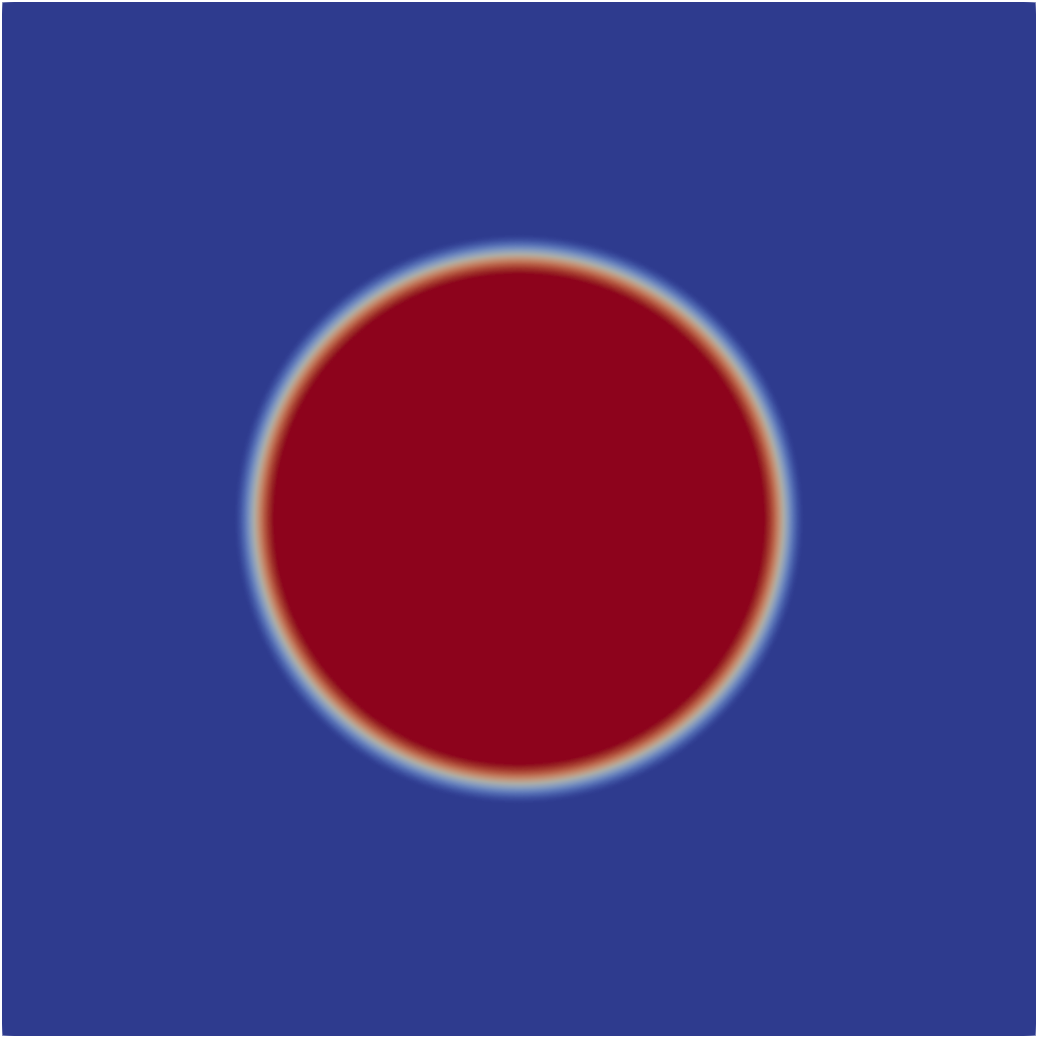}
\caption{Deterministic numerical solution ($\nu=0$) at time $t=0,2.5,6.5,8 \times 10^{-3}$.}
\label{fig_void_det}
\end{figure}

\begin{figure}
\includegraphics[width=0.24\textwidth]{u_void_000}
\includegraphics[width=0.24\textwidth]{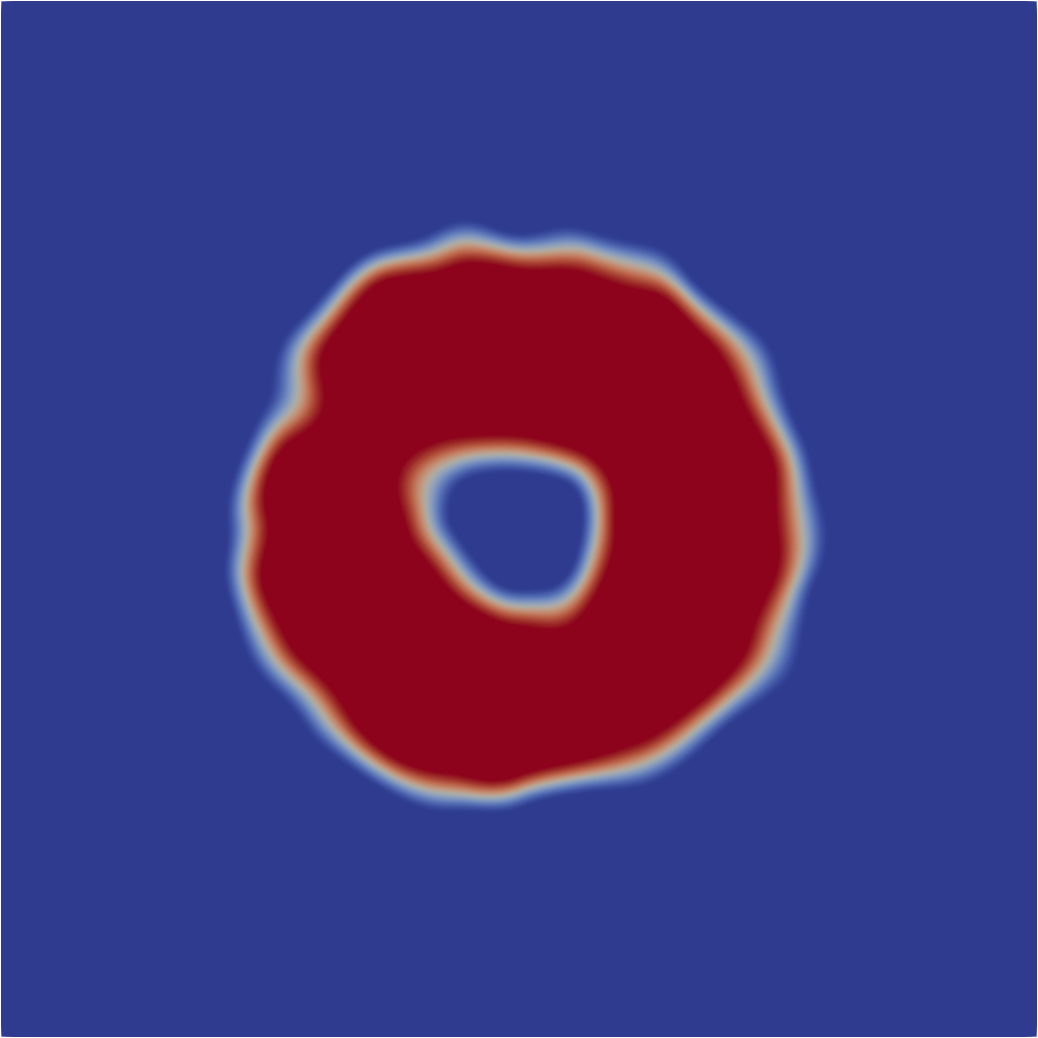}
\includegraphics[width=0.24\textwidth]{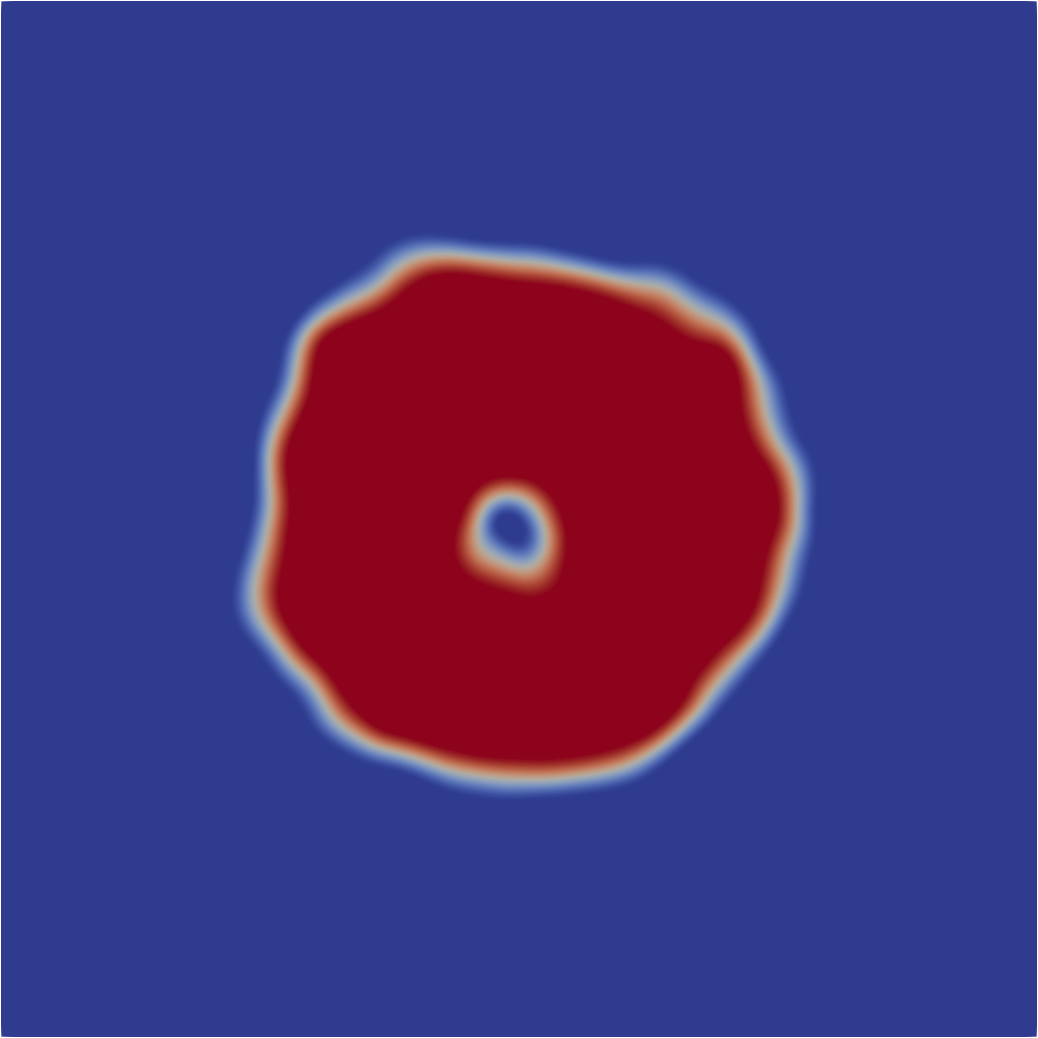}
\includegraphics[width=0.24\textwidth]{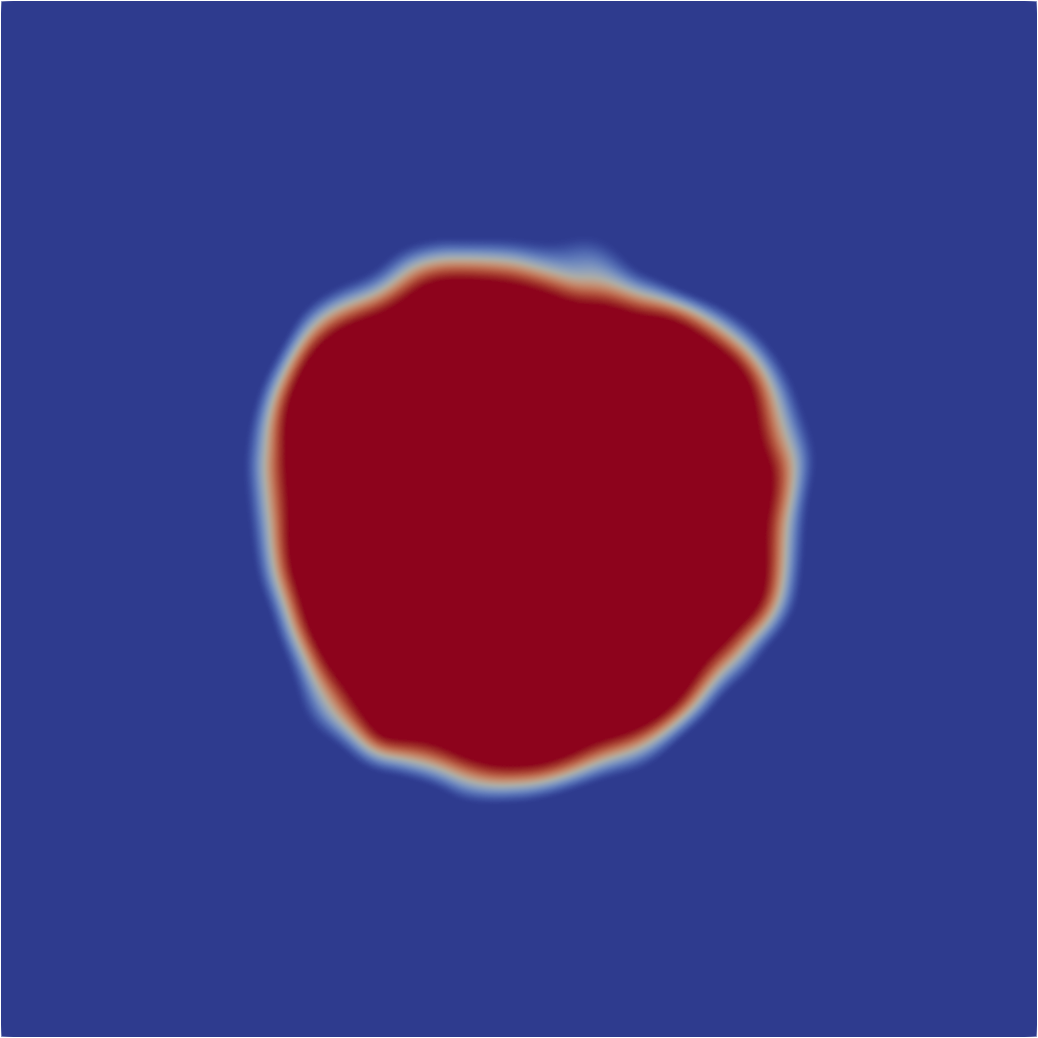}
\caption{Stochastic numerical solution ($\nu=1.6$) at time $t=0,2.5,6.5,8 \times 10^{-3}$.}
\label{fig_void_stoch}
\end{figure}

\begin{figure}
\includegraphics[width=0.24\textwidth]{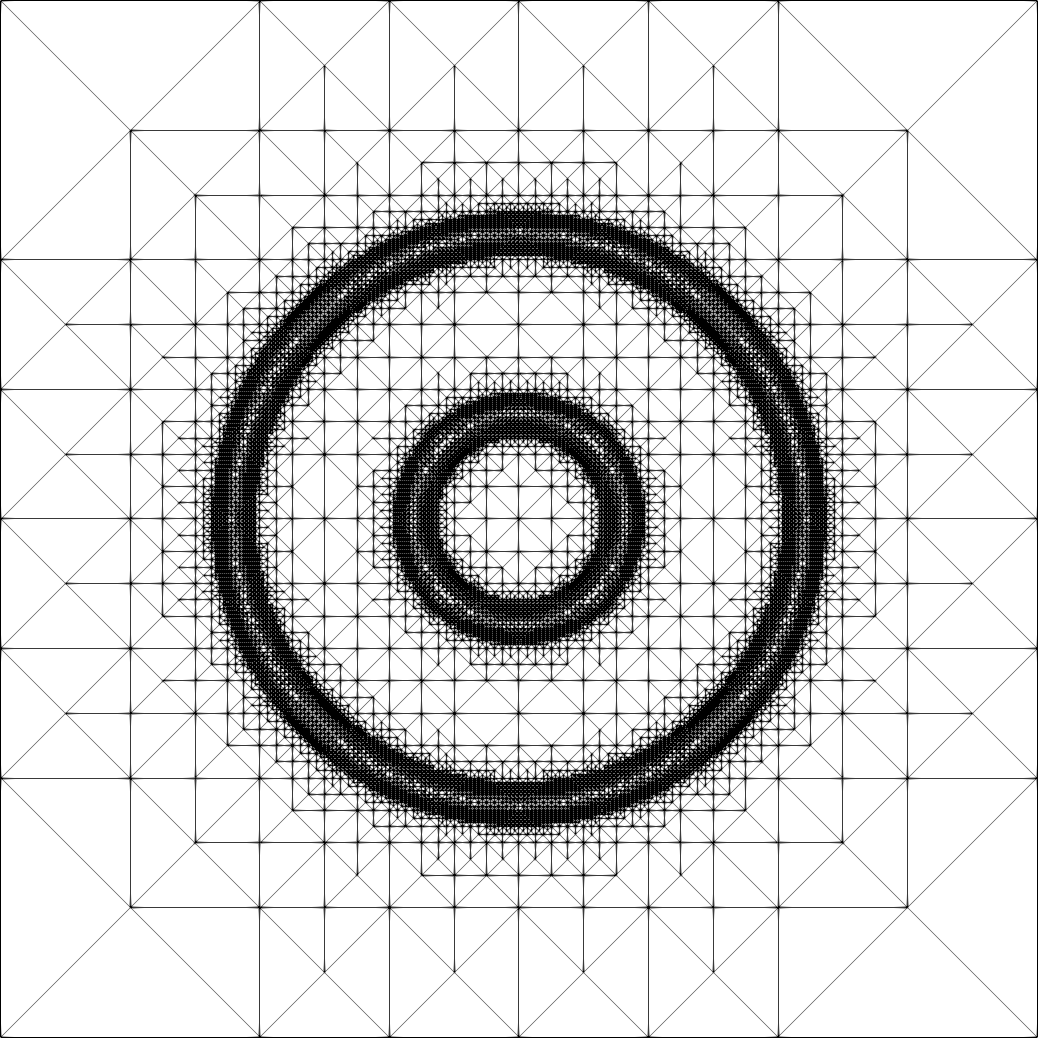}
\includegraphics[width=0.24\textwidth]{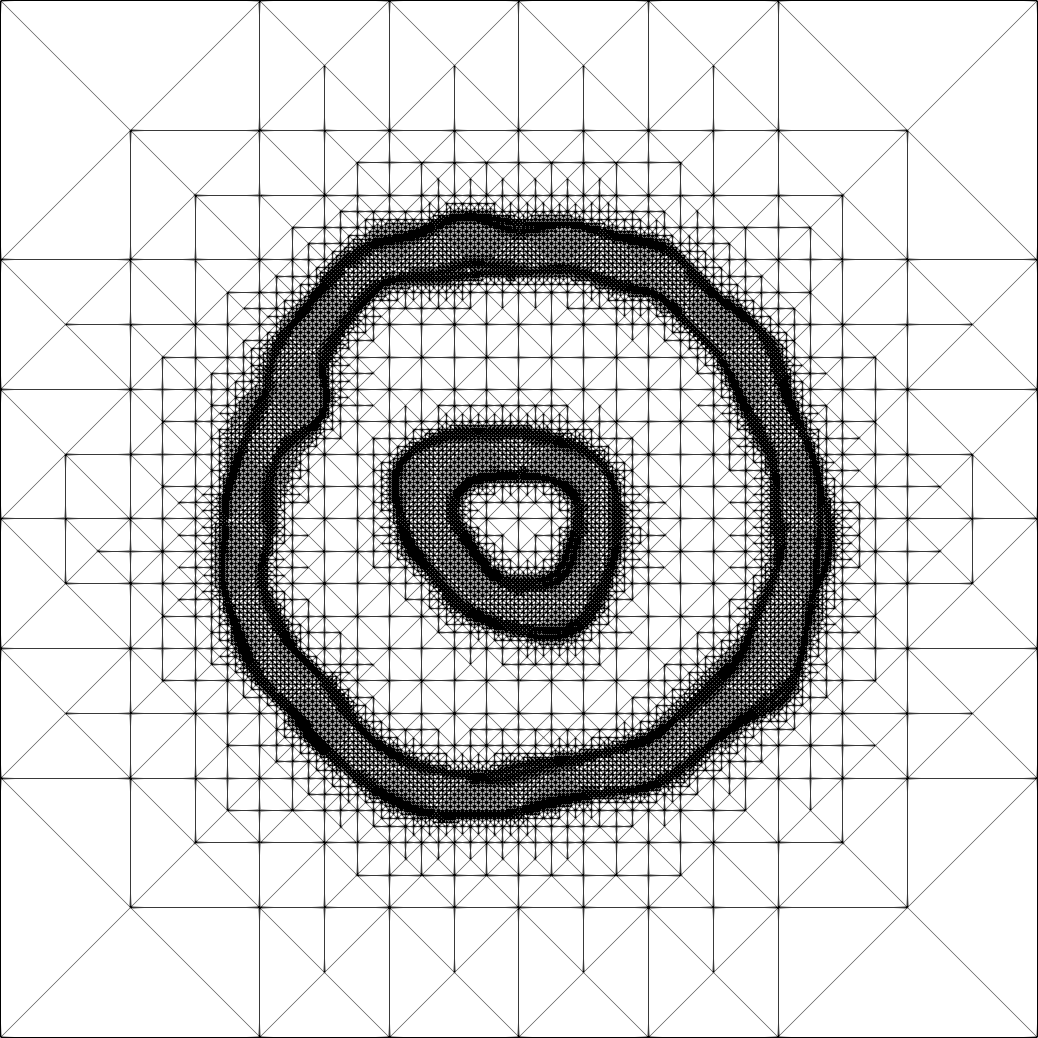}
\includegraphics[width=0.24\textwidth]{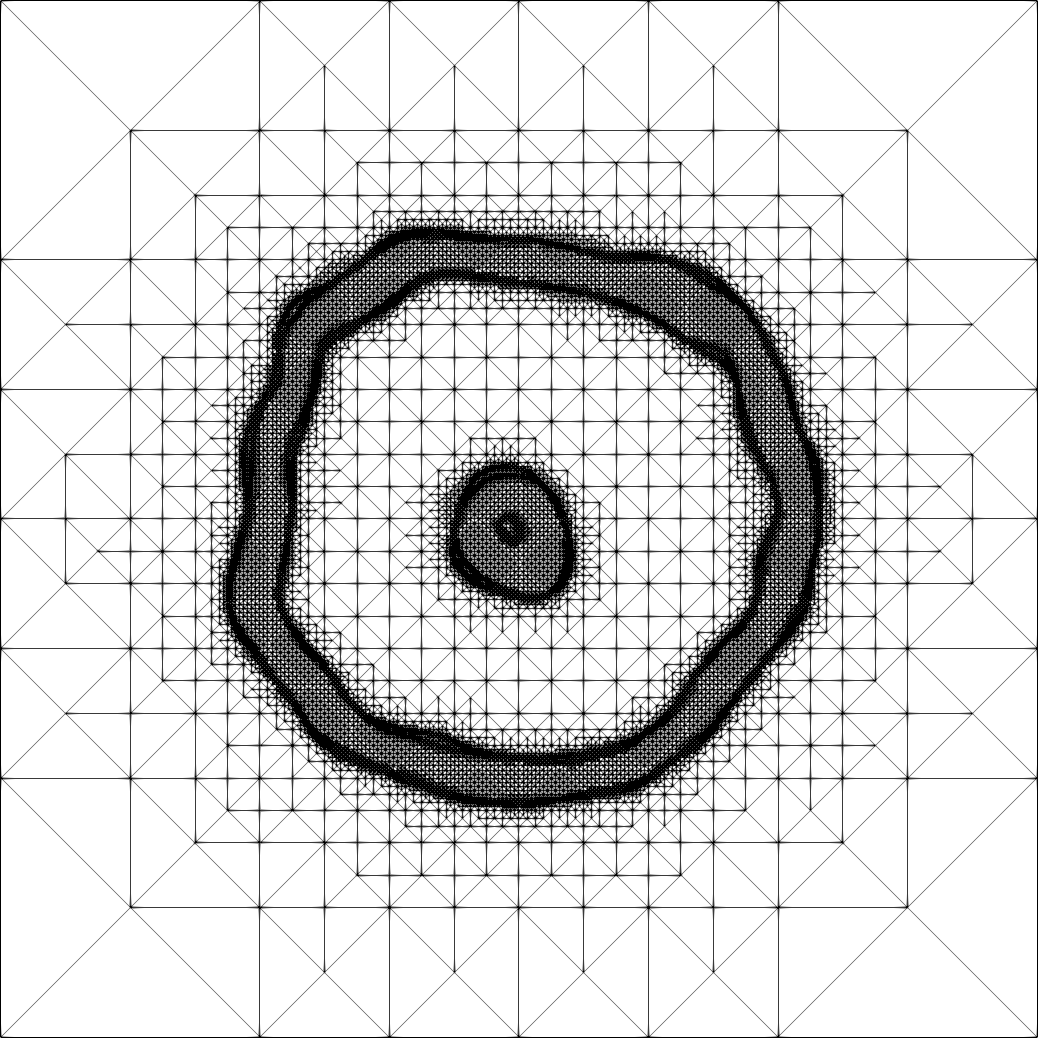}
\includegraphics[width=0.24\textwidth]{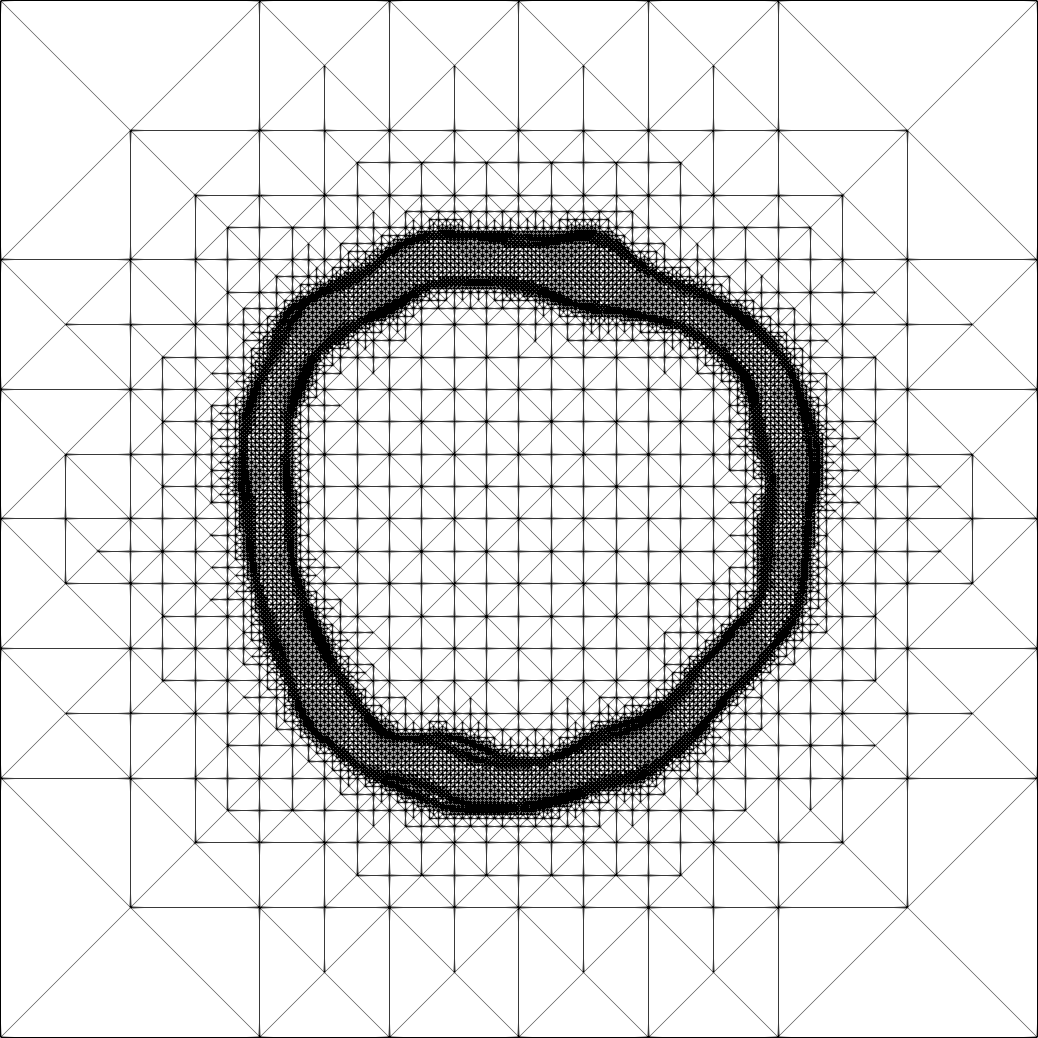}
\caption{Finite element mesh for {one path of} the stochastic numerical solution ($\nu=1.6$) at time $t=0,2.5,6.5,8 \times 10^{-3}$.}
\label{fig_void_mesh}
\end{figure}

The evolution of the energy $\displaystyle\mathcal{E}(u) = \tfrac{\varepsilon}{2} \|\nabla u \|^2 + \tfrac{1}{\varepsilon}\int_{\mathcal{D}}(F_\delta(u)-\tfrac{u^2}{2})\dx$ of the deterministic problem computed with scheme  (\ref{eq:model:disc:inequality}) and the expected values of the energy of the stochastic
problem computed with $500$ samples for schemes (\ref{eq:model:disc:inequality})  and (\ref{eq:model:disc:regularised})  are displayed in Figure~\ref{fig_void_ener} (left).
{Since $F_\delta$ is negative for arguments slightly outside of $\tekla{-1,+1}$, the energy of the regularised problem is smaller than the one of computed with the active-set strategy.}
Furthermore, in Figure~\ref{fig_void_ener} (right) we display the expected value of the area the inner circle (i.e., the sum of areas of triangles of the mesh enclosed by the larger circle where $u \leq 0$)
\begin{figure}
\includegraphics[width=0.48\textwidth]{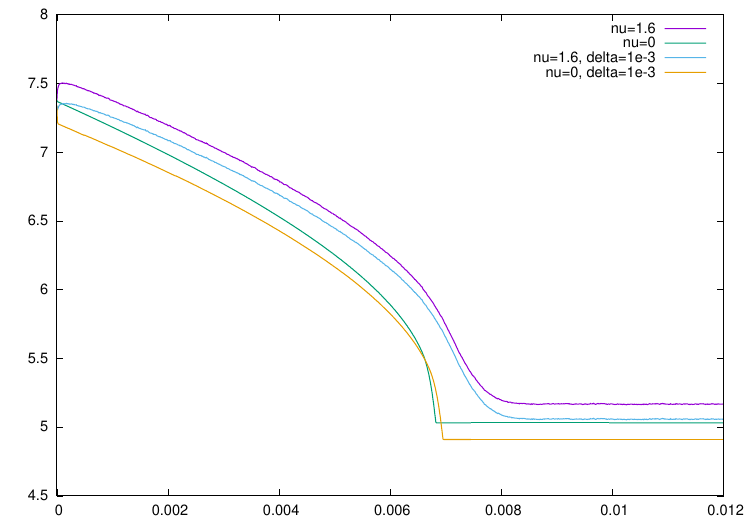}
\includegraphics[width=0.48\textwidth]{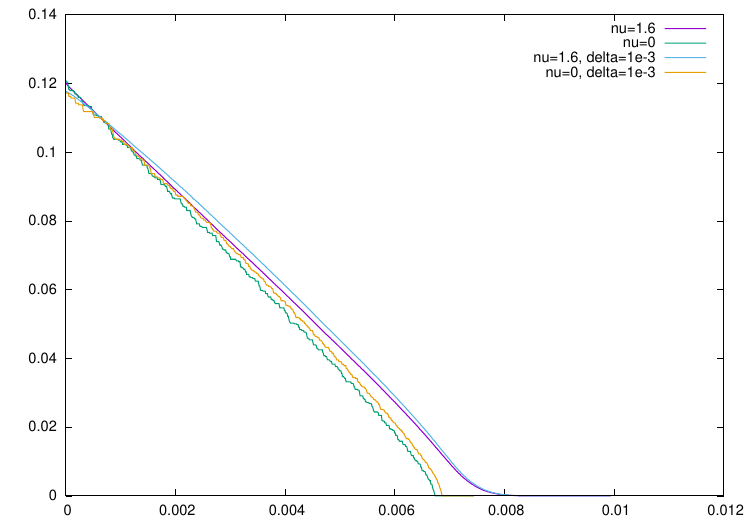}
\caption{Evolution of the expected value of the energy (left) and of the area of the middle circle (right).}
\label{fig_void_ener}
\end{figure}
In the deterministic setting the inner circle disappears at the time $t=0.00674$, the the stochastic setting the occurrence of the topological change is delayed, see {also} Figure~\ref{fig_void_hist}
where we display the histogram of the ''closing times'' of the inner circle for $n=1.6$ (left)
and for $\nu=0.16,0.8,1.6$ (right) (for comparison we also include the evolution of the area of the inner circle for the deterministic problem $\nu=0$).
\begin{figure}
\includegraphics[width=0.48\textwidth]{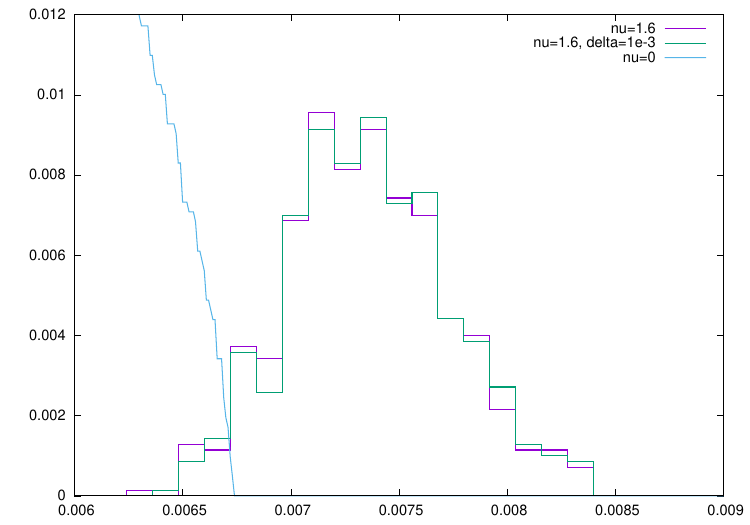}
\includegraphics[width=0.48\textwidth]{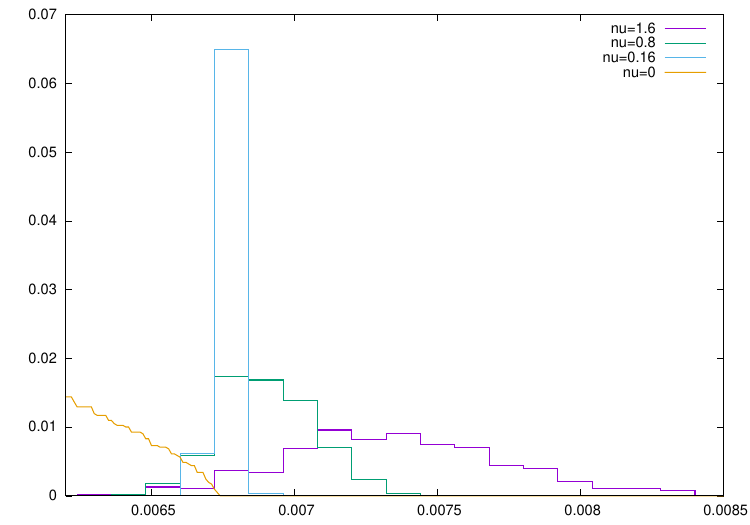}
\caption{Histogram of closing times of the inner circle for $\nu=1.6$ (left) and for $\nu=0.16,0.8,1.6$ (right) along with the area of the inner circle for $\nu=0$.}
\label{fig_void_hist}
\end{figure}

To illustrate the effect of the noise intensity $\nu$ on the evolution we display in {Figure~\ref{fig_void_ener_nu}} the evolution of the energy and the area of the inner circle for $\nu=0, 0.16, 0.8, 1.6$
{computed with the scheme (\ref{eq:model:disc:inequality}).
The results indicate that a larger noise intensity reduces the expected energy decay and delays the closing of the inner circle.}
%  and (\ref{eq:model:disc:regularised}).
\begin{figure}
\includegraphics[width=0.48\textwidth]{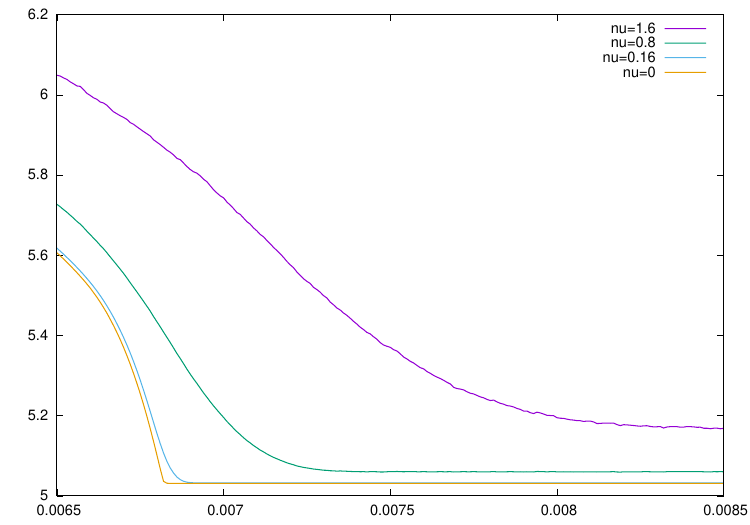}
\includegraphics[width=0.48\textwidth]{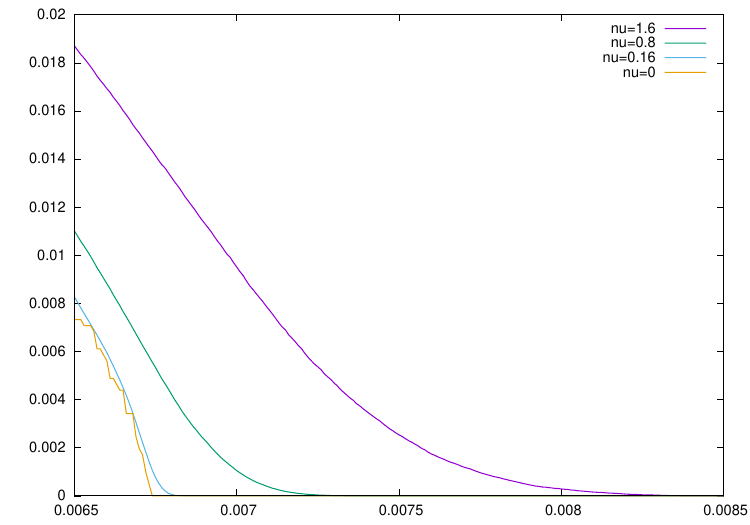}
\caption{Evolution of the expected value of the energy (left) and of the area of the inner circle (right) for $\nu=0, 0.16, 0.8, 1.6$.}
\label{fig_void_ener_nu}
\end{figure}

%\footnote{In the new computation, the convergence with respect to $\delta$ is not observable in the saturation (the reason could rather strong noise), I would not include this result. I could try to recompute with weaker noise.}
%{\color{red}
%To illustrate the effect of the regularisation parameter $\delta$
%we display in Figure~\ref{fig_void_ener_delta} the evolution of the energy and the are of the inner circle for $\delta=0, 10^{-3}, 10^{-2}$ with $\nu=1.6$.
%We note that the deterministic solution was roughly in the range $(-1.015, 1.015)$ and $(-1.0015, 1.0015)$ for $\delta=10^{-2}$ and $\delta=10^{-3}$, respectively,
%i.e., we observed linear dependence of on the regularisation parameter $\delta$ on the violation of the constraint $[-1,1]$ for the considered problem setup.
%\begin{figure}
%\includegraphics[width=0.48\textwidth]{ener_void_delta}
%\includegraphics[width=0.48\textwidth]{satur_void_delta}
%\caption{Evolution of the expected value of the energy (left) and of the area of the inner circle (right) for $\delta=0, 10^{-3}, 10^{-2}$.}
%\label{fig_void_ener_delta}
%\end{figure}
%}

\bibliographystyle{amsplain}
\bibliography{refs}

\end{document}